\documentclass[reqno,10pt]{amsart}
\usepackage{graphicx,amsfonts,amssymb,amsmath,amsthm,url,amscd,comment}
\usepackage{color}
\usepackage{chngcntr}
\usepackage[usenames,dvipsnames]{xcolor}
\usepackage[normalem]{ulem}
\usepackage{enumerate}
\usepackage{pdfsync}
\usepackage{graphicx, amsmath, amssymb, amsfonts, amsthm, stmaryrd, tikz, amscd}
\usepackage[all]{xy}
\usetikzlibrary{matrix,arrows,decorations.pathmorphing}
\newif\iflong
\longtrue

\usepackage{graphicx, amsmath, amssymb, amsfonts, amsthm, stmaryrd, tikz, amscd}
\usepackage[all]{xy}
\usetikzlibrary{matrix,arrows,decorations.pathmorphing}

 \usepackage[hyperfootnotes=false, colorlinks, citecolor=	 RoyalBlue, urlcolor=blue, linkcolor=blue         ]{hyperref}

\theoremstyle{plain} %--default
\newtheorem{theorem}             {Theorem} 
 
 \newtheorem*{hypothesisHF} {Hypothesis $H(\mathcal{F})$}

\theoremstyle{definition}

\theoremstyle{plain} %--default

\theoremstyle{remark}

\counterwithin*{definition}{subsubsection}
\newtheorem*{definition*}  {Definition}

\newtheorem*{example*}    {Example}

\newtheorem{remark}             {Remark}
\newtheorem*{remark*}            {Remark}

\newtheoremstyle{itplain} % name
    {6pt}                    % Space above
    {5pt\topsep}                    % Space below
    {\itshape}                   % Body font
    {}                           % Indent amount
    {\itshape}                   % Theorem head font
    {.}                          % Punctuation after theorem head
    {5pt plus 1pt minus 1pt}                       % Space after theorem head
    {}  % Theorem head spec (can be left empty, meaning ‘normal’)

\theoremstyle{itplain} %--default
\newtheorem{lemma}{Lemma}

\newtheorem*{lemma*}{Lemma}

\newtheorem*{corollary*} {Corollary}

\newtheorem*{conjecture*}    {Conjecture}

\theoremstyle{remark} %--default

\newtheorem*{lemmatest*}{Lemma}

\counterwithin*{lemma}{subsubsection}
\counterwithin*{corollary}{subsubsection}

\usepackage{etoolbox}
\patchcmd{\section}{\scshape}{\bfseries}{}{}
\makeatletter
\renewcommand{\@secnumfont}{\bfseries}
\makeatother

\renewcommand{\Re}{\mathrm{Re}}

\renewcommand{\geq}{\geqslant}
\renewcommand{\leq}{\leqslant}

\numberwithin{equation}{section}
\DeclareMathOperator{\sgn}{sgn}
\DeclareMathOperator{\Mp}{Mp}

\def\rtchi{\omega}

\DeclareMathOperator{\htt}{ht}

\DeclareMathOperator{\ad}{ad}
\DeclareMathOperator{\Ad}{Ad}

\def\eps{\varepsilon}

\def\PGL{\operatorname{PGL}}
\def\GL{\operatorname{GL}}
\def\SL{\operatorname{SL}}

\DeclareMathOperator{\norm}{norm}
\DeclareMathOperator{\trace}{trace}

\DeclareMathOperator{\Sob}{Sob}

\def\O{\operatorname{O}}

\DeclareMathOperator{\ram}{ram}

\DeclareMathOperator{\pr}{pr}

\DeclareMathOperator{\Eis}{Eis}

\DeclareMathOperator{\Wd}{Wd}

\DeclareMathOperator{\Sp}{Sp}

\DeclareMathOperator{\vol}{vol}

\DeclareMathOperator{\supp}{supp}

\author{Paul D. Nelson}
\address{ETH Z{\"u}rich, Department of Mathematics, R{\"a}mistrasse 101, CH-8092, Z{\"u}rich, Switzerland}
\email{paul.nelson@math.ethz.ch}
\date{\today}
\title[Subconvex equidistribution, Eisenstein observables]{Subconvex equidistribution of cusp forms: reduction to Eisenstein observables}
\hypersetup{
  pdfkeywords={},
  pdfsubject={},
  pdfcreator={Emacs 24.5.1 (Org mode 8.2.10)}}
\begin{document}

\begin{abstract}
  Let $\pi$ traverse a sequence of cuspidal automorphic
  representations of $\GL_2$ with large prime level, unramified central character and bounded
  infinity type.
  For $G \in \{\GL_1, \PGL_2\}$,
  let $H(G)$ denote the assertion that
  subconvexity holds for $G$-twists of the adjoint
  $L$-function of $\pi$, with polynomial dependence upon the
  conductor of the twist.
  We
  show that $H(\GL_1)$ implies $H(\PGL_2)$.

  In geometric terms, $H(\PGL_2)$ corresponds roughly to an instance
  of arithmetic quantum unique ergodicity with a power savings
  in the error term, $H(\GL_1)$ to the special
  case in which the relevant sequence of measures
  is tested against an Eisenstein series.
\end{abstract}
\maketitle

\setcounter{tocdepth}{2}

\begin{comment}
  Let $\pi$ traverse a sequence of cuspidal automorphic representations of GL(2) with large prime level and bounded infinity type.  For G either of the groups GL(1) or PGL(2), let H(G) denote the assertion that subconvexity holds for G-twists of the adjoint $L$-function of $\pi$, with polynomial dependence upon the conductor of the twist.  We show that H(GL(1)) implies H(PGL(2)).

  In geometric terms, H(PGL(2)) corresponds roughly to an instance of arithmetic quantum unique ergodicity with a power savings in the error term, H(GL(1)) to the special case in which the relevant sequence of measures is tested against an Eisenstein series.
\end{comment}

% keywords: subconvexity; equidistribution; quantum unique ergodicity; L-functions; automorphic forms; Eisenstein series; metaplectic group

\tableofcontents

\section{Introduction}
\label{sec-1}
In the 1980's and 1990's, it was discovered
that the subconvexity problems\footnote{
  A family of $L$-functions $L(\pi,s)$ attached to a
  family $\mathcal{F}$ of parameters $\pi \in \mathcal{F}$ with
  analytic conductors $C(\pi) \in \mathbb{R}_{\geq 1}$ is said to
  satisfy a \emph{subconvex bound} if there are fixed quantities
  $\delta = \delta(\mathcal{F}) > 0$ and
  $c = c(\mathcal{F}) \geq 0$ so that
  $|L(\pi,\tfrac{1}{2})| \leq c C(\pi)^{1/4-\delta}$ for all
  $\pi \in \mathcal{F}$; the \emph{subconvexity problem} for a
  given family consists of establishing a subconvex bound (see
  e.g. \cite{MR1826269}, \cite{MR2331346}, \cite{michel-2009}).
  In this article, $L(\pi,s)$ always
  refers to the finite part of an $L$-function,
  omitting $\Gamma$-factors at $\infty$.
}
for the families
of automorphic $L$-functions
\begin{equation}\label{sc:1}
  \text{
    $L(\pi \otimes \tau,\tfrac{1}{2})$
    \quad
    ($\tau$ on
    $\GL_2$ fixed, $\pi$ on $\GL_2$ varying),
  }
\end{equation}
\begin{equation}\label{sc:2}
  \text{
    $L(\ad(\pi) \otimes \tau,\tfrac{1}{2})$
    \quad
    ($\tau$ on
    $\PGL_2$ fixed, $\pi$ on $\GL_2$ varying)
  }
\end{equation}
were intimately related to fundamental arithmetical
equidistribution problems, concerning in the first case
\eqref{sc:1} the distribution of integral points on spheres,
representations of integers by ternary quadratic forms, Heegner
points and closed geodesics on the modular surface, and so on
(see e.g. \cite{MR1826269, MichelVenkateshICM, MR2331346} and
references), and in the second case \eqref{sc:2} the limiting
mass distribution of automorphic forms (``arithmetic quantum
unique ergodicity'', see e.g. \cite{MR1321639, MR1826269,
  MR2680499, PDN-AP-AS-que, sarnak-progress-que} and
references).  Motivated in part by such applications,
fundamental and increasingly robust methods for attacking those
problems were introduced and developed by many authors.  A
capstone of this progress was the solution by Michel--Venkatesh
\cite{michel-2009} in 2009 of the general case of the
subconvexity problem for \eqref{sc:1}.
One crucial step towards \cite{michel-2009}
was Michel's observation \cite{Mi04} that the subconvexity problem for
\eqref{sc:1} may be reduced to the corresponding problem for
$L(\tau \otimes \chi,1/2)$ ($\tau$ on $\GL_2$ fixed, $\chi$ on
$\GL_1$ varying).

By contrast, and despite sustained interest,
the subconvexity problem for \eqref{sc:2}
has seen no progress in non-dihedral cases until very recently.
(The case of \eqref{sc:2} in which $\pi$ is dihedral
reduces to \eqref{sc:1},
as exploited by
Sarnak \cite{Sar01}
in one of the first works on cases of \eqref{sc:1} in which both factors are cuspidal.)

Under important local assumptions (roughly
``prime level aspect''),
the main result of this article
reduces the subconvexity problem for
\eqref{sc:2}
to the
corresponding problem for
\begin{equation}\label{eqn:reduce-to-4}
  \text{
    $L(\pi \otimes \overline{\pi} \otimes \chi,\tfrac{1}{2})$
    \quad
    ($\chi$ on
    $\GL_1$ fixed, $\pi$ on $\GL_2$ varying).
  }
\end{equation}
In view of the factorizations\footnote{ The reader might 
  ask whether such identities are known to hold for the
  ramified Euler factors at finite places.  We believe this
  question may be answered affirmatively (e.g., by defining
  those local factors using known cases of local Langlands, and
  noting that this definition is compatible with the period
  formulas that we cite, which indeed depend only upon the
  unramified Euler factors).  On the other hand, the bad Euler
  factors at finite places are irrelevant to our purposes, since
  bounds towards Ramanujan show that their presence has no
  effect on the subconvexity problem (see \S\ref{sec-5-2}, item
  \ref{item:bad-Euler-irrel}).  The reader is thus free to replace
  every $L(\dotsb)$ in this article with $L^{(S)}(\dotsb)$,
  where $S$ denotes the set of bad places and $L^{(S)}$ the
  Euler product obtained by omitting factors from such places.  }
\begin{equation}\label{eq:factorization-1}
  L(\pi \otimes \overline{\pi } \otimes \tau,\tfrac{1}{2})
  =
  L(\tau,\tfrac{1}{2}) L(\ad(\pi) \otimes \tau,\tfrac{1}{2})
\end{equation}
\begin{equation}\label{eq:factorization-2}
  L(\pi \otimes \overline{\pi } \otimes \chi,\tfrac{1}{2})
  =
  L(\chi,\tfrac{1}{2}) L(\ad(\pi) \otimes \chi,\tfrac{1}{2}),
\end{equation}
the latter problem
reduces further to
that for
\begin{equation}\label{eq:munshi-aspect}
  \text{$L(\ad(\pi) \otimes \chi, \tfrac{1}{2})$
    \quad
    ($\chi$ on
    $\GL_1$ fixed, $\pi$ on $\GL_2$ varying)
  }.
\end{equation}
Our result thus mildly strengthens that indicated in the
abstract.

Essential motivation for this work came from a talk by R. Munshi
at ETH Zurich in May 2015, where he announced a proof of a
subconvex bound for \eqref{eq:munshi-aspect} in a specific
aspect ($\pi$ corresponding to a holomorphic form of large prime
level over $\mathbb{Q}$); a detailed draft
\cite{2017arXiv170905615M} of that proof has been available
since April 2017.  The families \eqref{sc:2} specialize to
\eqref{eq:munshi-aspect} upon restricting $\tau$ to be an
Eisenstein series.  The motivating applications indicated above
require also the cuspidal case of \eqref{sc:2}.
That case should now
follow from the reduction established here, leading to
subconvex bounds for \eqref{sc:2} and hence strong quantitative
forms of arithmetic quantum unique ergodicity in the prime level
aspect.  We have proposed to describe such applications jointly
with Munshi once our respective contributions have been
finalized.  An intriguing open problem is to what extent these
methods may be generalized to other aspects;
we indicate some of the challenges associated with doing so
at the end of this paper.

\subsection{Statement of main result}
Fix a number field $F$; all results are new already
when $F = \mathbb{Q}$.
Let $\mathfrak{q}$ traverse a sequence of finite primes in $F$,
with norms tending off to $\infty$.
We assume given for each such $\mathfrak{q}$
a cuspidal automorphic representation $\pi$ of
$\GL_2(\mathbb{A})$ (by convention, $\pi$ is unitary)
satisfying the following assumptions:
\begin{enumerate}[(i)]
\item The local component $\pi_{\mathfrak{q}}$
  is a twist of the special representation.
  Equivalently,
  some twist of $\pi_\mathfrak{q}$
  has unramified
  central character
  and conductor $\mathfrak{q}$.
\item $\pi$ is ``essentially unramified away from
  $\mathfrak{q}$'' in the sense that
  $\prod_{\mathfrak{p} \neq \mathfrak{q}} C(\pi_\mathfrak{p})
  = \norm(\mathfrak{q})^{o(1)}$, where the product is over all
  places $\mathfrak{p}$,
  $C(\dotsb)$ denotes the analytic conductor,
  and $o(1)$ denotes a
  quantity tending to zero as
  the norm of $\mathfrak{q}$ tends to $\infty$.
\end{enumerate}
These assumptions may seem
artificial, but
we
show in
\S\ref{sec:comments-other-aspects}
that
they are essential to a sufficiently
restricted form of our method.
Informally, they mean that (some twist of)
$\pi$ has essentially unramified central character and level
$\approx \mathfrak{q}$.
(Twisting
matters little here,
since it
does not change the quantities \eqref{sc:2} and
\eqref{eqn:reduce-to-4}.)
For example, we might take
$F = \mathbb{Q}$, so that $\mathfrak{q}$ corresponds to a prime
number $p$, and take $\pi$ corresponding to a normalized weight
two newform on $\Gamma_0(p)$.

Let $\tau$ be a cuspidal automorphic representation of
$\PGL_2(\mathbb{A})$.  We allow $\tau$, like $\pi$, to depend
upon the varying prime $\mathfrak{q}$, but our results are nontrivial
only when this dependence is mild.  For technical
convenience, we impose the following local assumption:
\begin{itemize}
\item[(PS)] Each local component of $\tau$ belongs to the principal series.
\end{itemize}
For instance, this assumption is satisfied if $\tau$ corresponds
to a spherical Maass form on
$\SL_2(\mathbb{Z}) \backslash \mathbb{H}$, which is the case
relevant for the basic ``configuration space'' forms of
quantum unique ergodicity.
This assumption could likely be removed
with further work
orthogonal to the primary
novelty of this paper
(see the
conjecture of \S\ref{sec-3-15-1}).

We will show that a subconvex bound for
$L(\pi \otimes \overline{\pi} \otimes \chi,\tfrac{1}{2})$ 
in the $\pi$-aspect with polynomial dependence
upon the Hecke character
$\chi$
implies a subconvex bound for
$L(\ad(\pi) \otimes \tau,\tfrac{1}{2})$
in the $\pi$-aspect with polynomial dependence upon $\tau$.
More precisely,
we impose the following hypothesis,
inspired by Munshi's work \cite{2017arXiv170905615M}:
\begin{itemize}
\item[(H)]
  For some fixed
  $c, A \geq 0, \delta > 0$ and all
  unitary characters $\chi$ of $\mathbb{A}^\times / F^\times$,
  \[
    |L(\pi \otimes \overline{\pi} \otimes
    \chi,\tfrac{1}{2})|
    \leq c
    C(\pi \otimes \overline{\pi} \otimes \chi)^{1/4-\delta}
    C(\chi)^{A}.
  \]
\end{itemize}
Here and henceforth ``fixed''
means ``independent of $\mathfrak{q}$.''
We will derive from this
the following conclusion:
\begin{itemize}
\item [(SC)]
  For some fixed
  $c, A \geq 0$ and $\delta > 0$,
  \[
    |L(\ad(\pi) \otimes  \tau,\tfrac{1}{2})|
    \leq c
    C(\ad(\pi) \otimes \tau)^{1/4-\delta}
    C(\tau)^{A}.
  \]
\end{itemize}
\begin{theorem}\label{thm:main-theorem-1}
  Under the stated assumptions,
  \emph{(H)} implies
  \emph{(SC)}.
\end{theorem}
We formulate this slightly more precisely
in
\S\ref{sec-5-1-4}.

\begin{remark}\label{rmk:classical-interp}
  As indicated above,
  Theorem \ref{thm:main-theorem-1} is already new in the special
  case that
  \begin{itemize}
  \item  $F = \mathbb{Q}$, 
  \item $\pi$ corresponds to a weight two
    cuspidal $L^2$-normalized newform $\varphi$ on $\Gamma_0(p)$ for some large prime
    $p$, and
  \item $\tau$ corresponds to an essentially fixed Maass cusp form $\Psi$ on
    $\SL_2(\mathbb{Z})$.
  \end{itemize}
  (In this informal discussion, we refer to certain quantities
  $X$ as ``essentially fixed'': this means all estimates are
  required to depend ``polynomially''
  upon such quantities in a sense which should be clear
  in each case.)
  Under the period-to-$L$-value dictionary (see for instance
  \cite[\S1]{PDN-AP-AS-que} and references),
  the informal content of our result in this special case
  is that ``subconvex
  bounds'' for
  $\langle \varphi(z), E(z) \varphi(\ell^2 z) \rangle$, where
  \begin{itemize}
  \item $\ell \in \mathbb{Z}_{\geq 1}$ is essentially fixed,
    and
  \item $E$ is an essentially fixed unitary Eisenstein series
    on $\Gamma_0(\ell^2)$ of parameter $1/2 + i t$,
    completed with the factor $\xi(1 + 2 i t)$,
  \end{itemize}
  imply ``subconvex bounds'' (with weaker exponents) for
  $\langle \varphi, \Psi \varphi \rangle$.
  An   unconditional logarithmic savings for the latter
  was obtained in
  \cite{PDN-HQUE-LEVEL} using the Holowinsky--Soundararajan
  method \cite{MR2680499}.  

  Our result may thus be interpreted roughly as follows: to establish
  equidistribution with a power savings as
  $p \rightarrow \infty$ of the sequence of probability measures
  $\mu_{\varphi}$ on $\SL_2(\mathbb{Z}) \backslash \mathbb{H}$
  obtained by pushforward of the $L^2$-mass of $\varphi$
  (cf. \cite{PDN-HQUE-LEVEL}),
  it suffices to prove a power savings
  estimate for the quantities $\mu_\varphi(E)$
  when $E$ is a fixed Eisenstein series, together with mild generalizations of such quantities.
  
  Our result is not the first concerning equidistribution
  on arithmetic quotients
  to exhibit a distinguished role played by Eisenstein
  series.
  Earlier works
  exhibiting this role include
  \begin{enumerate}
  \item
    Lindenstrauss's \cite{MR2195133} on arithmetic quantum
    unique ergodicity (AQUE) for Maass forms on
    $\SL_2(\mathbb{Z}) \backslash \mathbb{H}$,
    through the (implicit) role played by Eisenstein series in
    conditionally
    ruling out ``escape of mass''
    prior to the unconditional results of
    Soundararajan \cite{MR2680500};
  \item Holowinsky--Soundararajan's \cite{MR2680499} on
    AQUE for holomorphic forms
    on $\SL_2(\mathbb{Z}) \backslash \mathbb{H}$
    in that the additional smoothing implemented
    by \cite[Thm 1.1]{MR2680498}
    is necessary only when testing against Eisenstein series;
  \item Einsiedler, Lindenstrauss, Michel and Venkatesh's
    \cite{MR2776363} on Duke's theorem for cubic fields, through
    the use of Eisenstein series to establish \emph{a priori}
    bounds concerning tightness and positivity of entropy.
  \item Ghosh, Reznikov and Sarnak's \cite{MR3102912} concerning
    nodal domains of Maass forms, through its invocation of
    AQUE for
    $\SL_2(\mathbb{Z}) \backslash \mathbb{H}$ tested (exclusively) against
    Eisenstein series.
  \end{enumerate}
  All of these works point to a distinguished role played by
  Eisenstein series
  in arithmetic equidistribution problems.
  The precise implication observed
  in this article seems particularly direct, striking, and
  counterintuitive.
  For instance,
  it shows also that the Eisenstein case of (prime level aspect) AQUE
  on congruence covers of $\SL_2(\mathbb{Z}) \backslash
  \mathbb{H}$
  controls the general case on compact quotients
  $\Gamma \backslash \mathbb{H}$
  attached to non-split quaternion algebras.
\end{remark}

\begin{remark}
  In the context of Remark \ref{rmk:classical-interp},
  it remains an open problem to obtain
  analogous savings on \emph{compact} arithmetic quotients
  attached to non-split quaternion algebras
  (see e.g. \cite[\S2]{2012arXiv1210.1243N}); those would likely
  follow from the proof of Theorem \ref{thm:main-theorem-1}
  and strong enough
  \emph{logarithmic} savings over the trivial bound in the
  Eisenstein case on 
  congruence covers of $\SL_2(\mathbb{Z}) \backslash \mathbb{H}$ (e.g., the Eisenstein case of \cite[Thm
  2]{PDN-AP-AS-que}
  for $q$ prime
  but with $\delta_2$ large enough).
\end{remark}

\begin{remark}
  The dependence of the quantities $\delta,A$ in (SC) upon those
  in (H) is effective. We do not explicate it here.
\end{remark}

\subsection{Overview of the proof}\label{sec:overview-proof}
Over the past couple decades,
many people have attempted to estimate
$L(\ad(\pi) \otimes \tau,\tfrac{1}{2})$ or the closely related
quantity
$L(\pi \otimes \overline{\pi } \otimes \tau,\tfrac{1}{2})$
(see \eqref{eq:factorization-1})
by embedding $\pi$ in a family $\mathcal{F}$ and trying to estimate some
(possibly amplification-weighted) moment,
such as the mean value
\begin{equation}\label{eq:mean-value-for-pi-over-family-intro}
  S(\mathcal{F}) :=
  \sum_{\pi' \in \mathcal{F}}
  L(\ad(\pi') \otimes \tau,\tfrac{1}{2}).
\end{equation}
It is known that
each summand
$L(\ad(\pi') \otimes \tau,\tfrac{1}{2})$  is nonnegative,
hence that
$L(\ad(\pi) \otimes \tau,\tfrac{1}{2}) \leq S(\mathcal{F})$.
In the
context of Remark \ref{rmk:classical-interp}, one might take
for $\mathcal{F}$ the set of weight two normalized newforms on
$\Gamma_0(p)$; a Lindel{\"o}f-consistent estimate
\begin{equation}\label{eq:hopeful-lindelof-consistent}
  S(\mathcal{F}) \ll p^{1+\eps}
\end{equation}
would then recover the
convexity bound for $L(\ad(\pi) \otimes \tau,\tfrac{1}{2})$.
In general,
one might hope to derive
a subconvex bound
from a sharp mean value estimate
over a sufficiently small (possibly amplification-weighted)
family.

Although this approach has succeeded spectacularly for superficially
similar problems (see e.g. \cite{MR1923476, MR2207235} and references), it remains a fantasy
in the present
setting: the estimates for
$L(\ad(\pi) \otimes \tau,\tfrac{1}{2})$ achieved this way fail
even to approach the convexity bound
(compare with \cite{luo-sarnak-mass, MR3554896}, for instance).
A proof of
\eqref{eq:hopeful-lindelof-consistent}
(let alone its amplified variant)
seems inaccessible by existing techniques
(cf. \cite{MR1833252}).

Alternatively, one could try to estimate
$L(\pi \otimes \overline{\pi} \otimes \tau,\tfrac{1}{2})$ by
using the triple product formula to relate it to a period
integral $\langle \varphi_1, \varphi_2 \varphi_3 \rangle$ on
$[\PGL_2] :=\PGL_2(F) \backslash \PGL_2(\mathbb{A}_F)$, with unit vectors
$\varphi_1, \varphi_2 \in \pi$ and $\varphi_3 \in \tau$,
and
applying the technique developed by Michel--Venkatesh \cite{michel-2009}
in their resolution of the subconvexity problem
for $\GL_2$.
That technique proceeds via the Cauchy--Schwarz inequality
followed by a spectral expansion
on $L^2([\PGL_2])$ (together with an ``amplification'' step
that we elide here):
\begin{equation}\label{eq:spectral-decopm-as-in-MV-sketch}
  |\langle \varphi_1, \varphi_2 \varphi_3 \rangle|^2
\leq 
\langle \varphi_2 \varphi_3, \varphi_2 \varphi_3 \rangle
= 
\langle |\varphi_2|^2, |\varphi_3|^2 \rangle
=
\int_{\psi}
\langle |\varphi_2|^2, \psi \rangle \langle \psi, |\varphi_3|^2 \rangle.
\end{equation}
One encounters in such an attempt the unfortunate
\emph{circularity}  that one cannot adequately estimate the
contribution from $\psi \in \tau$ to the RHS of
\eqref{eq:spectral-decopm-as-in-MV-sketch}
without already
knowing a subconvex bound for the quantity
$L(\pi \otimes \overline{\pi} \otimes \tau,\tfrac{1}{2})$
of primary interest.

The strategy pursued in this article is similar to that of
Michel--Venkatesh, but with triple product integrals on
$\PGL_2(\mathbb{A}_F)$ replaced by Shimura--type integrals on
the metaplectic double cover of $\SL_2(\mathbb{A}_F)$.
A circularity similar to that noted above arises in this
approach;
we manage to break it, as discussed below, by employing
a crucial observation made in \cite{nelson-theta-squared}.

Turning to details,
we use a period formula of Qiu \cite[Thm 4.5]{MR3291638} and
local estimates to derive an integral
representation
\begin{equation}\label{eq:period-integral-inrto}
  \frac{L(\ad(\pi) \otimes \tau, \tfrac{1}{2})}{
    C(\ad(\pi) \otimes \tau)^{1/4+o(1)}}
  =
  |\langle \varphi_1 \varphi_2, \varphi_3 \rangle|^2,
\end{equation}
where $\langle , \rangle$ denotes the Petersson inner product on
$[\SL_2] := \SL_2(F) \backslash \SL_2(\mathbb{A}_F)$,
$\varphi_1$ belongs to $\pi$, $\varphi_2$ is an elementary theta
function defined on the metaplectic double cover of
$\SL_2(\mathbb{A}_F)$, and $\varphi_3$ belongs to a suitable
Waldspurger lift $\theta(\tau)$ of $\tau$; moreover,
$\|\varphi_i\| = 1$ for $i=1,2,3$.  (The proof of
\eqref{eq:period-integral-inrto} uses the Shimizu
correspondence; the idea of applying it here arose as a natural
continuation of the works
\cite{2012arXiv1210.1243N,nelson-variance-73-2,
  nelson-variance-II}.)  We summarize how these vectors vary
(see \S\ref{sec:interesting-local-estimates} for details):
\begin{itemize}
\item $\varphi_1$ is essentially a newvector in the varying
  representation $\pi$.
\item $\varphi_2$ is essentially fixed.
\item $\varphi_3$ is a varying vector in the essentially fixed
  representation $\theta(\tau)$, given in the ``line model at
  $\mathfrak{q}$'' by an $L^2$-normalized multiple of the
  characteristic function of the maximal ideal.
\end{itemize}
By Cauchy--Schwarz, the RHS of \eqref{eq:period-integral-inrto}
is
bounded by
$\|\varphi_1 \varphi_2\|^2 =
\langle |\varphi_1|^2, |\varphi_2|^2 \rangle$.
We consider the
spectral expansion of the latter inner product
(regularized via \cite{nelson-theta-squared}):
\begin{equation}\label{eq:expansion}
  \langle |\varphi_1|^2, |\varphi_2|^2 \rangle
  = 
  \langle |\varphi_1|^2, 1 \rangle \langle 1, |\varphi_2|^2
  \rangle
  +
  \sum_{\psi}
  \langle |\varphi_1|^2, \psi \rangle \langle \psi, |\varphi_2|^2
  \rangle
  + (\operatorname{CSC}),
\end{equation}
where $\psi$ traverses an orthonormal basis for the space of
cusp forms on $[\SL_2]$ and (CSC) denotes the contribution of
the continuous spectrum; it is an integral
of
$\langle |\varphi_1|^2, E \rangle \langle E, |\varphi_2|^2 \rangle$
for some Eisenstein series $E$.
We estimate the RHS of \eqref{eq:expansion} as follows:
\begin{itemize}
\item[(i)] Hypothesis (H) may be seen to imply an adequate
  estimate for (CSC).  (The fact that $\varphi_2$ is essentially
  fixed forces $\langle E, |\varphi_2|^2 \rangle$ to decay
  rapidly with respect to the parameters of $E$, so it suffices
  to bound the individual factors
  $\langle |\varphi_1|^2, E \rangle$, which are related to the
  $L$-values \eqref{eqn:reduce-to-4} via Rankin--Selberg
  theory.)
\item[(ii)]
  $|\varphi_2|^2$ is orthogonal to cusp forms
  \cite{nelson-theta-squared}, so the sum over
  $\psi$ vanishes.
\item[(iii)] We ``amplify away'' the contribution of the constant
  function, as in \cite{michel-2009}.
\end{itemize}
The proof is then complete.

We highlight three reasons
why 
our approach may have been unanticipated:
\begin{enumerate}
\item The observation (ii) that
  $\langle |\varphi_2|^2, \psi \rangle = 0$ for cusp forms
  $\psi$ is the \emph{deus ex machina} needed to break the
  apparent
  circularity of the argument:
  by the triple product formula, a bound for
  $\langle |\varphi_1|^2, \psi \rangle$ with $\psi \in \tau$
  amounts to a bound for
  $L(\pi \otimes \overline{\pi } \otimes \tau,\tfrac{1}{2})$,
  hence (by the factorization
  \eqref{eq:factorization-1})
  to one for $L(\ad(\pi) \otimes \tau,\tfrac{1}{2})$.
\item The analytic theory of period integrals on $\PGL_2$ was
  developed substantially and applied to the subconvexity
  problem in work of Bernstein--Reznikov \cite{MR2726097},
  Venkatesh \cite{venkatesh-2005}, Michel--Venkatesh
  \cite{michel-2009}, and others.  We initiate here an analogous
  theory on the metaplectic double cover of $\SL_2$, which we
  have attempted to develop robustly so as to be of general use.
  Before doing so, it was not obvious that an approach as
  indicated above should exist.
\item The strategy indicated following
  \eqref{eq:mean-value-for-pi-over-family-intro} may be
  implemented in our setting by applying Cauchy--Schwarz
  following \eqref{eq:period-integral-inrto} to the vector
  $\varphi_1 \in \pi$ belonging to the varying representation,
  as in the work of Michel--Venkatesh \cite{michel-2009},
  Iwaniec--Michel \cite{MR1833252}, and many others.  We instead
  apply Cauchy--Schwarz to $\varphi_3$; this amounts to
  embedding the \emph{fixed} representation $\tau$ (rather than
  the varying representation $\pi$) into an (implicit) family.
  In this respect, our basic strategy is counter-traditional.

  A
  similarly counter-traditional application of Cauchy--Schwarz
  may be seen in the method of Bykovsky \cite{MR1433344} and its
  generalization by Fouvry--Kowalski--Michel \cite{MR3334236},
  where an algebraically-twisted sums of the Fourier
  coefficients of a fixed $\GL_2$ automorphic form are estimated
  by averaging over a varying family containing that form.  It
  seems likely to us that some of Munshi's recent arguments
  (e.g., those in \cite{MR3418527}) may also be understood from this
  perspective.
\end{enumerate}

\begin{remark}\label{rmk:annoying-restriction}
  The restriction that $\pi$ have
  ``essentially prime level and trivial central character''
  seems
  serious at the moment.
  One could just as easily treat squarefree levels
  (i.e., allowing multiple independent varying primes
  $\mathfrak{q}$), but any further extension 
  would be
  interesting.
  An analogue of Theorem \ref{thm:main-theorem-1}
  for archimedean, depth, or even ``prime-squared level''
  aspects remains open.
  
  The precise source of the present limitation of our method
  remains poorly understood by us.
  The issue is roughly that
  for aspects other than those pursued here, there do not seem
  to exist vectors for which \eqref{eq:period-integral-inrto}
  holds and for which the strategy indicated following \eqref{eq:expansion}
  succeeds
  (see \S\ref{sec:comments-other-aspects}).
  We remain hopeful that a viable further extension
  of the strategy exists.
\end{remark}

% \subsection{Convention}
% We refer to main results (final or intermediary) original to
% this article as theorems and propositions; everything else
% (including deep cited work) is called a lemma.  This convention
% is not a commentary.

\subsection*{Acknowledgements}
We thank Ph. Michel for many helpful and extended discussions concerning the
work \cite{michel-2009}.
We gratefully acknowledge the support of
NSF grant OISE-1064866
and SNF grant SNF-137488 during some of the work leading
to this paper.
We thank the referees for helpful corrections and suggestions
on an earlier draft.

\section{Local preliminaries}
\label{sec-3}
The inputs to this section are:
\begin{itemize}
\item a local field $F$ 
  not of characteristic $2$, thus $F$ is a finite extension of either
  $\mathbb{R}$ or $\mathbb{Q}_p$ or (for $p \neq 2$) $\mathbb{F}_p(t)$;
\item a nontrivial
  unitary character $\psi : F \rightarrow \mathbb{C}^{(1)} := \{
  z \in \mathbb{C}^\times : |z| = 1 \}$; and
\item 
  some Haar measures
  on the groups $G \in \{ \GL_1(F) = F^\times, \SL_2(F), \PGL_2(F)\}$;
  we often denote simply by $\int_{g \in G} (\dotsb)$ the
  corresponding integral.
\end{itemize}
We equip $F$ with the $\psi$-self dual Haar measure $d x$.
We denote by
$|.| : F \rightarrow \mathbb{R}_{\geq 0}$ the normalized
absolute value,
so that $d(c x) = |c| \, d x$.

\subsection{Groups, measures, and general notation}
\label{sec-3-1}

\subsubsection{The metaplectic group\label{sec:local-metaplectic-group}}
\label{sec-3-1-2}
Denote by $\Mp_2(F)$ the metaplectic double cover of
$\SL_2(F)$,
defined using Kubota cocycles
as the set of all pairs
$(\sigma,\zeta) \in \SL_2(F) \times \{\pm 1\}$
with the multiplication law
$(\sigma_1,\zeta_1) (\sigma_2, \zeta_2)= (\sigma_1 \sigma_2, \zeta_1 \zeta_2 c(\sigma_1,\sigma_2))$,
where $c$ is the cocycle
\begin{equation}\label{eq:local-kubota-2}
  c(\sigma_1,\sigma_2)
  := \left(
    \frac{x(\sigma_1 \sigma_2)}{ x(\sigma_1)},
    \frac{x(\sigma_1 \sigma_2)}{ x(\sigma_2)}
  \right),
  \quad x
  \left(
    \begin{pmatrix}
      \ast & \ast \\
      c & d
    \end{pmatrix}
  \right)
  := \begin{cases}
    d & \text{ if }c = 0, \\
    c & \text{ if } c \neq 0
  \end{cases}
\end{equation}
and
$(,) : F^\times/F^{\times 2} \times F^\times/F^{\times 2}
\rightarrow \{\pm 1\}$ denotes the Hilbert symbol.
\subsubsection{Group elements}
\label{sec-3-1-3}
As generators for
$\Mp_2(F)$
we take
for $a \in F^\times, b \in F$
and $\zeta \in \{\pm 1\}$
the elements
\[
  n(b)
  = \left( \begin{pmatrix}
      1 & b \\
      & 1
    \end{pmatrix}, 1 \right),
  \quad 
  t(a)
  = \left( \begin{pmatrix}
      a &  \\
      & a^{-1}
    \end{pmatrix}, 1 \right),
\]
\[
  n'(b) := \left( \begin{pmatrix}
      1 &  \\
      b & 1
    \end{pmatrix}, 1 \right),
  \quad 
  w :=
  \left( \begin{pmatrix}
      & 1 \\
      -1 & 
    \end{pmatrix}, 1 \right),
  \quad
  \eps(\zeta) = (1,\zeta).
\]
They
satisfy the relations
$n(b_1) n(b_2) = n(b_1 + b_2)$,
$n'(b_1) n'(b_2) = n'(b_1 + b_2)$,
$t(a_1) t(a_2) = t(a_1 a_2) \eps((a_1,a_2))$,
$t(a) n(b) = n(a^2 b) t(a)$,
$t(a) n'(b) = n'(a^{-2} b) t(a)$,
$w t(a) = t(a^{-1}) w$,
$w^2 = t(-1)$,
$n'(b) =  w n(-b) w^{-1}$,
and
$w n(-a^{-1}) = n(a) t(a) w n(a) w^{-1}$.

For $y \in F^\times$, set
\[
  a(y)
  = \begin{pmatrix}
    y &  \\
    & 1
  \end{pmatrix} \in \GL_2(F).
\]
There are natural maps
\begin{equation}\label{eq:natural-maps-local-groups}
  \Mp_2(F) \rightarrow \SL_2(F) \rightarrow \GL_2(F)
  \rightarrow  \PGL_2(F).
\end{equation}
The images
in $\PGL_2(F)$ 
of 
$t(y)$
and
$a(y^2)$ 
coincide.

\subsubsection{Genuine representations}
\label{sec-3-1-4}
Recall that a
\emph{genuine} representation $\pi$ of $\Mp_2(F)$ is one
for which $\pi(\eps(\zeta)) v = \zeta v$
for all $\zeta \in \{\pm 1\}$.
An irreducible representation of $\Mp_2(F)$ is thus genuine
if and only if it does not factor through $\SL_2(F)$.

\subsubsection{Identifications}
\label{sec-3-1-5}
When
appropriate, we allow ourselves to confuse elements of
$\Mp_2(F)$ with their images in $\SL_2(F)$ and functions on
$\SL_2(F)$ with their pullbacks to $\Mp_2(F)$.
For example, given a pair $\pi_1, \pi_2$ of genuine
unitary representations of $\Mp_2(F)$
and vectors $u_1,v_1 \in \pi_1$, $u_2,v_2 \in \pi_2$,
the ``product of matrix coefficients'' function
$\Mp_2(F) \ni g \mapsto \langle g u_1, v_1 \rangle
\langle g u_2, v_2 \rangle \in \mathbb{C}$
identifies with a function
on
$\SL_2(F)$.

\subsubsection{Notation in the non-archimedean case}
\label{sec-3-1-7}
For $F$ non-archimedean,
we denote by $\mathfrak{o}$ its maximal order,
$\mathfrak{p}$ its maximal ideal,
and $q := \# \mathfrak{o}/\mathfrak{p}$.
Set $\zeta_F(s) := (1 - q^{-s})^{-1}$.

\subsubsection{``The unramified case''}
\label{sec-3-1-8}
\label{sec:local-unramified-case}
We use this phrase to signify that
\begin{itemize}
\item $F$ is non-archimedean and of odd residual characteristic (i.e., $q$ is odd),
\item the additive character $\psi$ is unramified
  (i.e., restricts trivially to $\mathfrak{o}$ but not to $\mathfrak{p}^{-1}$), and
\item the chosen Haar measures
  on $F$, $\GL_1(F) \cong F^\times, \SL_2(F)$ and $\PGL_2(F)$
  all assign volume one to maximal compact subgroups.
\end{itemize}

\subsubsection{Implied constants}
\label{sec-3-1-9}
We want implied constants to be uniform when $F$
traverses the set of completions of a given number field
and the character $\psi$ and the choices of Haar measure
traverse the local components of associated global data
(as in \S\ref{sec:global-groups-measures-norms} below).
For this reason, we call an object $X$ \emph{fixed} if it
admits the following dependencies:
\begin{itemize}
\item We always allow $X$ to depend upon the degree of the local field
  $F$, and also upon fixed quantities mentioned earlier within a given argument.
\item
  \begin{itemize}
  \item  If we are not in the unramified case,
  then we allow $X$ to depend upon $F$, $\psi$
  and all choices of Haar measures.
\item In the unramified case, we  do not allow such dependence.
  \end{itemize}
\end{itemize}
Standard asymptotic notation
is defined in terms of this notion:
\begin{itemize}
\item The equivalent
  notations $X = O(Y), X \ll Y, Y \gg X$
  signify that $|X| \leq c |Y|$
  for some fixed $c \in \mathbb{R}_{\geq 0}$;
  notation such as $X \ll_f Y$
  allows an additional dependence of $c$ upon $f$.
\item $X \asymp Y$
  signifies that $X \ll Y \ll X$.
\end{itemize}
The symbols $\eps$ and $A$ denote
respectively
a sufficiently small and large fixed positive quantity whose precise value may
change from one occurrence to another.
\subsubsection{Norms}
\label{sec-3-1-10}
For $G$ one of the groups $\SL_2$, $\PGL_2, \GL_2$,
we fix once and for all a linear embedding
$G \hookrightarrow \SL_N$
and define norms $\|.\|$ on $G(F)$
using this embedding, as in \cite[\S2.1.3,
\S2.1.7]{michel-2009}.
We extend this definition to $\Mp_2(F)$ by pulling back
under the projection $\Mp_2(F) \rightarrow \SL_2(F)$.

We also define for $g \in \PGL_2(F)$
the adjoint norm $\|\Ad(g)\|$ as in \cite[\S2.1.3]{michel-2009}
and extend this definition to the groups
$\GL_2(F), \SL_2(F), \Mp_2(F)$
by pulling back under the maps \eqref{eq:natural-maps-local-groups}.
For $y \in F^\times$,
one has $\|\Ad(t(y))\| \asymp |y|^2 + |y|^{-2}$.
\subsection{Weil representation\label{sec:weil-repn}}
\label{sec-3-2}
Denote by $\rho_\psi$
the Weil representation of $\Mp_2(F)$
attached to the additive character $\psi$.
We realize it on the Schwartz--Bruhat space $\mathcal{S}(F)$
and equip it with the unitary structure coming from $L^2(F)$.
It splits as a sum
$\rho_\psi = \rho_\psi^+ \oplus \rho_\psi^-$
of two invariant irreducible
subspaces
$\rho_\psi^+ =
\{\phi \in \mathcal{S}(F) : \phi(-x) = \phi(x)\}$
and
$\rho_\psi^- =
\{\phi \in \mathcal{S}(F) : \phi(-x) = -\phi(x)\}$.

We record the action of generators.
For $\xi \in F^\times$,
let $\psi_\xi : F \rightarrow \mathbb{C}^{(1)}$
denote the nontrivial unitary character of $F$ given by $\psi_\xi(x)
:= \psi(\xi x)$.
Let $\gamma(\psi_\xi) \in \{z \in \mathbb{C} : z^8 = 1\}$ denote
the Weil constant,
thus
$\gamma(\psi_\xi) = |2 \xi|^{1/2} \int_{x \in F} \psi(\xi x^2)
  \, d x$
as generalized functions of $\xi \in F^\times$.
Set
$\chi_{\psi}(\xi) := \gamma(\psi)/\gamma(\psi_\xi) \in \{z \in
\mathbb{C} : z^4 = 1\}$;
then $\chi_\psi(\xi_1 \xi_2) = (\xi_1,\xi_2) \chi_\psi(\xi_1) \chi_\psi(\xi_2)$.
For $\phi \in \mathcal{S}(F)$,
define $\phi^\wedge \in \mathcal{S}(F)$
by
$\phi^\wedge(y) :=
|2|_F^{1/2} \int _{F} \phi(x) \psi (2 x y) \, d x$;
the normalization
is so that
$(\phi^\wedge )^\wedge(x) = \phi(-x)$.
\begin{itemize}
\item $\rho_{\psi}(n(b)) \phi(x) = \psi(b x^2) \phi(x)$.
\item 
$\rho_{\psi}(w)
\phi =
\gamma(\psi)
\phi^\wedge$.
\item $\rho_{\psi}(t(a))
\phi(x)
= \chi_{\psi}(a)
|a|^{1/2}
\phi(a x)$.
\item $\rho_{\psi}(\eps(\zeta))
\phi
=
\zeta \phi$.
\end{itemize}

\subsection{Induced representations}
\label{sec-3-3}
Let
$\chi$ be a character of $F^\times$.
\subsubsection{$\SL_2$}
\label{sec-3-3-1}
Denote by $\mathcal{I}_{\SL_2}(\chi)$
the representation of $\SL_2(F)$ unitarily
induced by the character $t(y) \mapsto \chi(y)$
of the diagonal torus.
Its \emph{induced model}
is the space of smooth functions
$f : \SL_2(F) \rightarrow \mathbb{C}$
satisfying $f(n(x) t(y) g) = |y| \chi(y) f(g)$.
\subsubsection{$\Mp_2$}
\label{sec-3-3-2}
Denote by $\mathcal{I}_{\Mp_2}^{\psi}(\chi)$
the representation of $\Mp_2(F)$
unitarily induced by the character
$t(y) \eps(\zeta) \mapsto \zeta \chi_{\psi}(y) \chi(y)$ of the diagonal torus,
where $\chi_{\psi}$ is as in \S\ref{sec:weil-repn}.
Its \emph{induced model} 
is defined as in the $\SL_2$ case
to be the space of smooth functions
$f : \Mp_2(F) \rightarrow \mathbb{C}$
satisfying $f(n(x) \eps(\zeta) t(y) g) = |y| \zeta
\chi_{\psi}(y) \chi(y) f(g)$.
We henceforth abbreviate
$\mathcal{I}_{\Mp_2}(\chi) := \mathcal{I}_{\Mp_2}^{\psi}(\chi)$.
\subsubsection{$\PGL_2$}
\label{sec-3-3-3}
Denote by $\mathcal{I}_{\PGL_2}(\chi)$
the representation of $\PGL_2(F)$
unitarily induced by the character
$a(y) \mapsto \chi(y)$ of its diagonal torus.
Its \emph{induced model}
is the space of smooth functions
$f : \PGL_2(F) \rightarrow \mathbb{C}$
satisfying $f(n(x) a(y) g) = |y|^{1/2} \chi(y) f(g)$.
\subsubsection{The line model}
\label{sec-3-3-4}
For any of the above induced representations $\mathcal{I}_{G}(\chi)$,
the restriction map sending $f$ to the function
$F \ni x \mapsto f(n'(x))$ is injective.
We call its image the \emph{line model}.
The line model contains $C_c^\infty(F)$.
We shall occasionally use the line model to specify vectors.
\subsubsection{Unitarity}
\label{sec-3-3-5}
Let $G \in \{\SL_2(F), \Mp_2(F), \PGL_2(F)\}$.
If $\chi$ is unitary,
then the induced representation $\mathcal{I}_{G}(\chi)$
is unitary; it is also irreducible 
unless $G = \SL_2(F)$ and $\chi$ is a nontrivial quadratic character.
In any event, we normalize its unitary structure
via the line model:
$\|f\|^2 := \int_{x \in F} |f(n'(x))|^2$.

\subsubsection{From $\SL_2$ to $\PGL_2$ \label{sec:induced-rep-local-from-sl2-to-pgl2}}
\label{sec-3-3-7}
Let $\chi, \rtchi$ be characters of $F^\times$ satisfying
$\rtchi^2 = \chi$.
Let $j : \SL_2(F) \rightarrow \PGL_2(F)$
denote the natural map.
Every element of $\PGL_2(F)$
is of the form $a(y) j(s)$
for some $y \in F^\times, s \in \SL_2(F)$.
Given $f \in \mathcal{I}_{\SL_2}(\chi)$,
write
\begin{equation}\label{eq:formula-defining-f-omega-local}
  f^\rtchi (a(y) j(s)) :=
  |y|^{1/2} \rtchi(y) f(s)
\end{equation}
We verify readily
that
the formula \eqref{eq:formula-defining-f-omega-local} defines
  a function $f^\rtchi : \PGL_2(F) \rightarrow \mathbb{C}$
  for which
  $f^\rtchi \circ j = f$,
  hence that restriction of functions induces an
  $\SL_2(F)$-equivariant isomorphism
  \begin{equation}\label{eq:restriction-pgl2-sl2-local-isomorphism}
      \mathcal{I}_{\PGL_2}(\rtchi)|_{\SL_2(F)} \cong
  \mathcal{I}_{\SL_2}(\chi)
  \end{equation}
  with inverse $f \mapsto f^{\rtchi}$.
We note the following:
\begin{itemize}
  \item The isomorphism
\eqref{eq:restriction-pgl2-sl2-local-isomorphism}
identifies the respective line models.
\item 
The map $f \mapsto f^{\rtchi}$
is an isometry whenever $\chi,\rtchi$ are unitary.
\end{itemize}

\subsubsection{Metaplectic intertwining operators}\label{sec:metapl-intertw-oper}
Temporarily abbreviate $\sigma_\chi :=
\mathcal{I}_{\Mp_2}(\chi)$.
We shall have occasion to consider the standard intertwining operators
\[
  M_\chi : \sigma_\chi \rightarrow \sigma_{\chi^{-1}}
\]
\[M_\chi f(g) := \int_{x \in F} f(w n(x) g) \, d x.\]
These integrals and those that follow should be interpreted in the
usual ways, e.g., by regularized integration or meromorphic continuation in $\chi$; we suppress
discussion of this point for the sake of brevity.
We record here a detailed asymptotic study of $M_\chi$;
this will be applied subsequently
to certain technical estimates
involving complementary series representations.

Recall the local $\gamma$-factor $\gamma(\chi,s,\psi) = \eps(\chi,s,\psi) L(\chi^{-1},
1-s)/L(\chi,s)$,
characterized by Tate's local functional equation:
for each $\phi$ in the Schwartz--Bruhat space $\mathcal{S}(F)$,
\begin{equation}\label{eqn:Tate-local-func-eqn}
  \int_{x \in F^\times}
  \phi(x)
  \chi(x) |x|^s \, \frac{d x}{|x|}
  = 
  \frac{\int_{x \in F^\times}
    (\int_{y \in F} \phi(y) \psi(-y x) \, d y)
    \chi^{-1}(x) |x|^{1-s} \, \frac{d x}{|x|}}{  \gamma(\chi,s,\psi)}.
\end{equation}
It is holomorphic for $0 < \Re(\chi) + \Re(s) < 1$,
and satisfies the Stirling asymptotic (see \cite[\S3.1.12]{michel-2009})
\begin{equation}\label{eq:stirling}
  \gamma(\chi,s,\psi)
  \asymp C(\chi)^{1/2 - \Re(\chi) - \Re(s)}
  \text{ if }
  \eps < \Re(\chi) + \Re(s) < 1 - \eps
\end{equation}
(here $C(\cdot)$ denotes the analytic conductor),
the distributional identity
\begin{equation}\label{eq:LFE-cor-1}
  \int_{x \in F^\times}
  \psi(x)
  \chi(x) |x|^s \, \frac{d x}{|x|}
  = 
  \frac{\chi(-1)}{\gamma(\chi,s,\psi)}
\end{equation}
and the relation
\begin{equation}\label{eq:LFE-cor-2}
  \gamma(\chi,1/2,\psi)
  \gamma(\chi^{-1},1/2,\psi)
  = \chi(-1).
\end{equation}
We define the normalized intertwining operators
\[R_\chi := \frac{|2|_F^{-1/2} \gamma(\chi^2,0,\psi)}{\gamma(\psi) \gamma(\chi,1/2,\psi)} M_\chi,\]
which vary
holomorphically on $\{\chi : \Re(\chi) > -1/2\}$;
the normalization is justified further below.
These operators
are ``diagonalized'' by the map
\[
  \mathcal{K} : \sigma_\chi \rightarrow \{\text{functions } F^\times \rightarrow \mathbb{C}\}
\]
\[
  \mathcal{K} f(\xi) := \int_{u \in F} f(w n(u)) \psi(-\xi u) \, d u \text{ for } \xi \in F^\times,
\]
which packages
together the standard Whittaker functionals on $\sigma_\chi$;
for detailed discussion of this and what follows,
see \cite{MR2528059, MR2837718} and references.
By Fourier inversion, $\mathcal{K}$ is injective.
Its image
is closely related to the Kirillov-type model
of $\sigma_\chi$ (see \cite[p513]{MR2906917}, \cite{MR1103429});
one has
\[
  \mathcal{K} n(x) f(\xi) = \psi(\xi x) \mathcal{K} f(\xi),
  \quad 
  \mathcal{K} t(y) f(\xi) = |y| \chi_\psi(y) \chi^{-1}(y) \mathcal{K} f(y^2 \xi).
\]
Using the change of variables
$x \mapsto -1/x$ and the identity
$w n(-1/x) w n(u) = n(x) t(x) w n(x + u)$, we see that
\[
  \mathcal{K} R_\chi f(\xi)
  =
  G(\chi,\psi,\xi)
  \mathcal{K} f(\xi),
\]
where
\[
  G(\chi,\psi,\xi)
  :=
  \frac{|2|_F^{-1/2} \gamma(\chi^2,0,\psi)}{\gamma(\psi) \gamma(\chi,1/2,\psi)}
  \int_{x \in F} \psi(\xi x) \chi(x) \chi_\psi(x) \, \frac{d
    x}{|x|}.
\]
Let $\chi_\xi(x) := (x,\xi)$ denote the quadratic character
given by the Hilbert symbol.
\begin{lemma}
  One has
  \begin{equation}\label{eq:G-via-gamma}
    G(\chi,\psi,\xi)
    =
    \chi^{-1}(4 \xi) \chi_\psi(\xi)^{-1}
    \frac{\gamma(\chi \chi_\xi,1/2,\psi)}{\gamma(\chi,1/2,\psi)}.
  \end{equation}
\end{lemma}
\begin{proof}
  This is essentially the main result of \cite{MR2528059} and
  \cite[\S6]{MR3544420}.
  More precisely, the latter is equivalent to the identity\footnote{
      For the convenience of the reader and as a check of
  normalizations, we sketch a proof of
  \eqref{eq:metaplectic-coef} modulo convergence issues, which
  may be addressed by suitably interpreting each step as an
  identity of distributions on $F^\times$ or $F^\times \times F^\times$.  
  We apply Fourier inversion to $F \ni b \mapsto \psi(b^2 x)$
  to obtain
  $\psi(x) = \int_{y \in F} \psi(y) (\int_{b \in F} \psi(b^2 x - b y) \, d b ) \, d y$.
  By completing the square and using the identity $\gamma(\psi_x)
  = |2 x|^{1/2} \int_{b \in F} \psi(x b^2) \, d b$,
  we deduce that
  $\psi(x) \chi_\psi(x)
    =
    \gamma(\psi) |2 x|^{-1/2}
    \int_{y \in F} \psi(y - \frac{y^2}{4 x}) \, d y$.
  We insert this into the LHS of \eqref{eq:metaplectic-coef},
  giving
  \[
    \int_{x \in F} \psi(x) \chi(x) \chi_\psi(x) \, \frac{d
      x}{|x|}
    =
    \gamma(\psi) 
    \int_{x,y \in F}
    |2 x|^{-1/2}
    \chi(x)
    \psi(y - \frac{y^2}{4 x})
    \, \frac{d x}{|x|}
    \, d y.
  \]
  We apply the substitution
  $x \mapsto -y ^2/ 4 x$.
  The result factors as a product of an $x$-integral and a
  $y$-integral.
  We apply  \eqref{eq:LFE-cor-1}
  to each.
  Using \eqref{eq:LFE-cor-2},
  we readily obtain the required identity.
}
  \begin{equation}\label{eq:metaplectic-coef}
    \int_{x \in F} \psi( x) \chi(x) \chi_\psi(x) \, \frac{d
      x}{|x|}
    =
    \gamma(\psi)
    |2|_F^{1/2}
    \chi^{-1}(4)
    \frac{\gamma(\chi,1/2,\psi)}{\gamma(\chi^2,0,\psi)}.
  \end{equation}
  One may apply to the integral in the definition
  of $G(\chi,\psi,\xi)$
  the change of variables
  $x \mapsto x/\xi$ and the identity
  $\chi_\psi(x/\xi) = \chi_\psi(\xi)^{-1} \chi_\psi(x) \chi_\xi(x)$
  to derive \eqref{eq:G-via-gamma} from \eqref{eq:metaplectic-coef}.
\end{proof}
It follows in particular that
\begin{equation}\label{eq:R_chi-isom}
  R_{\chi^{-1}} \circ R_\chi = 1,
\end{equation}
which is equivalent to identities stated in \cite[4.11,
4.17]{MR2876383}.

\begin{lemma}\label{lem:asymp-for-G}
  Suppose that
  $\chi = \eta |.|^c$, where
  $\eta^2 = 1$
  and
  $c$ is a real number
  with $|c| \leq 1/2 - \eps$ for some fixed $\eps > 0$.
  Then
  $G(\chi,\psi,\xi) > 0$
  and  
  \begin{equation}\label{eq:G-asymp}
    G(\chi,\psi,\xi) =
    |\xi|^{-c} (C(\eta \chi_\xi)/C(\eta))^{-c}
    \beta(\chi,\psi,\xi),
  \end{equation}
  where
  \begin{enumerate}[(i)]
  \item $\beta(\chi,\psi,\xi) \asymp 1$,
  \item $\beta(\chi,\psi,\xi) = \beta(\chi,\psi,\xi z^2)$
    for all $z \in F^\times$,
  \item $\beta(\chi,\psi,\xi) > 0$, and
  \item if $c = 0$, then $G(\chi,\psi,\xi) = \beta(\chi,\psi,\xi) = 1$.
  \end{enumerate}
\end{lemma}
\begin{proof}
  (i) and (ii) follow from \eqref{eq:G-via-gamma} and
  \eqref{eq:stirling},
  while (iv) follows from (iii)
  and the consequence $|G(\eta,\psi,\xi)| = 1$
  of \eqref{eq:G-via-gamma}.
  We now verify (iii).
  Write $A \sim B$ to denote that $A/B > 0$,
  and write $\eta = \chi_t$ for some $t \in F^\times$.
  Then
  $\gamma(\chi \chi_{\xi}, 1/2,
  \psi)/\gamma(\chi,1/2,\psi)
  \sim \gamma(\chi_{t \xi}, 1/2,
  \psi)/\gamma(\chi_{t},1/2,\psi)$
  and $\chi^{-1}(4 \xi) \chi_\psi(\xi)^{-1}
  \sim \chi_t(\xi) \chi_\psi(\xi)^{-1}
  = \chi_\psi(t) / \chi_\psi(t \xi)$,
  so
  the positivity
  follows from the identity
  \begin{equation}\label{eqn:Weil-vs-Tate}
    \chi_{\psi}(\xi) = \gamma(\chi_{\xi},1/2,\psi)
    \text{ for all } \xi \in F^\times,
  \end{equation}
  or equivalently,
  $\chi_{\psi}(\xi) \gamma(\chi_{\xi},1/2,\overline{\psi}) = 1$,
  for which we refer to
  \cite{doi:10.1080/00927872.2017.1399407}
  (see also \cite[App. B, (2d)]{MR546613}
  and \cite{MR848382}).
  % which follows in turn from the formula given in \cite[App. B,
  % (2d)]{MR546613} for the $\gamma$-factor of the base change
  % under $F(\sqrt{\xi })/F$ of the trivial character of
  % $F^\times$ (see also \cite{doi:10.1080/00927872.2017.1399407,
  %   MR848382}).\footnote{For convenience and to check
  %   normalizations, we sketch a proof of
  %   \eqref{eqn:Weil-vs-Tate},
  %   ignoring convergence issues.  If
  %   $\xi \in F^{\times 2}$, then both sides are $1$, so suppose
  %   $\xi \notin F^{\times 2}$. Then
  %   $\chi_\psi(\xi) = \gamma(\psi) \gamma(\psi_{-\xi}) \sim
  %   \int_{x,y \in F} \psi(x^2 - \xi y^2) \, d x \, d y \sim
  %   \int_{z \in F^\times} \psi(z) 1_{N(\xi)}(z) \, d z$,
  %   where $N(\xi):= \{x^2 - \xi y^2 : x, y \in F\}$ denotes the
  %   set of norms from the quadratic field extension
  %   $F(\sqrt{\xi })/F$.  By
  %   expanding $1_{N(\xi)}(z) = (1 + \chi_\xi(z))/2$ and appealing
  %   to \eqref{eq:LFE-cor-1},
  %   we obtain
  %   $\chi_\psi(\xi) \sim \gamma(\chi_\xi,0,\psi)^{-1}
  %   \sim \gamma(\chi_\xi,1/2,\psi)^{-1}$.
  %   Since $|\chi_\psi(\xi)| = |\gamma(\chi_\xi,1/2,\psi)| = 1$,
  %   we deduce the required identity.
  % }
\end{proof}

\subsection{Unramified representations and vectors\label{sec:local-unram-stuff}}
\label{sec-3-4}
Assume for \S\ref{sec:local-unram-stuff} that $F$ is
non-archimedean.
As usual,
we say that a vector in a representation of one of the groups
$\GL_1(F) = F^\times, \SL_2(F), \PGL_2(F)$ is 
\emph{unramified} if it is invariant by the standard maximal
compact subgroup.
In the unramified case
(\S\ref{sec:local-unramified-case}),
a vector in a
representation of $\Mp_2(F)$ is \emph{unramified} if it is
invariant by the image of the standard lift to $\Mp_2(F)$ of the
standard maximal compact subgroup of $\SL_2(F)$ (see
e.g. \cite[\S4.5]{nelson-theta-squared}).  A
representation  of
any of the above groups is \emph{unramified} if it is
irreducible and contains a nonzero unramified vector.
For example:
\begin{itemize}
\item The unramified characters of $F^\times$
  are of the form $|.|^s$ for $s \in \mathbb{C}$.
\item $\rho_{\psi}^+$ is unramified
  in the unramified case.
\item 
  For an unramified character
  $\chi$ of $F^\times$
  and $G \in \{\SL_2, \PGL_2, \Mp_2\}$,
  the induced representation
  $\mathcal{I}_G(\chi)$
  is unramified if it is irreducible
  and if,
  when
  $G = \Mp_2$,
  we are in the unramified case.
\end{itemize}

\subsection{Whittaker models}
\label{sec-3-5}
Recall that an irreducible representation
$\pi$ of $\GL_2(F)$ is called \emph{generic}
if for some (equivalently, any) nontrivial unitary character $\psi'$ of $F$,
$\pi$ admits a $\psi'$-Whittaker model, i.e.,
a realization in the space of functions
$W : G \rightarrow \mathbb{C}$
satisfying $W(n(x) g) = \psi'(x) W(g)$
on which $G$ acts by right translation;
this is the case precisely when $\dim(\pi) \neq 1$.

The restriction map sending $W$ to the function
$F^\times \ni y \mapsto W(a(y))$ is injective.
Its image is called the \emph{Kirillov model} (or more verbosely, $\psi'$-Kirillov model).
The Kirillov model contains $C_c^\infty(F^\times)$.

Assume that $\pi$ is unitary.  Then the the norm $\|W\|^2 := \int_{y \in
  F^\times} |W(a(y))|^2$
is $G$-invariant.
When we write ``let $\pi$ be a generic unitary representation of
$\GL_2(F)$,
realized in its $\psi'$-Whittaker model,''
we always normalize the unitary structure in this way.

\subsection{Local Waldspurger packets\label{sec:local-waldspurger-packets}}
\label{sec-3-6}
Given an irreducible representation $\tau$ of $\PGL_2(F)$, one
may define (see \cite{MR1103429, gan-shim-a-la-wald})
a Waldspurger
packet $\Wd_{\psi}(\tau)$ consisting of either one or two
genuine irreducible representations of $\Mp_2(F)$; it is denoted
$\{\sigma^+\}$ in the former case and $\{\sigma^+, \sigma^-\}$
in the latter, where the labeling by $\pm $ is defined using the
local $\psi$-theta correspondence and Jacquet--Langlands
correspondence.  One has $\# \Wd_{\psi}(\tau) = 2$ if
and only if $\tau$ belongs to the discrete series.
If $\pi$ is generic, then the $\sigma^{\pm}$
are not isomorphic to even Weil
representations
$\rho_{\psi '}^{+}$.

If $\tau$ is a generic irreducible \emph{principal series}
representation
of $\PGL_2(F)$,
say
$\tau = \mathcal{I}_{\PGL_2}(\chi)$,
then $\Wd_{\psi}(\tau)$ is a singleton $\{\sigma\}$
consisting of the generic irreducible principal series
representation
$\sigma = \mathcal{I}_{\Mp_2}(\chi)$
of $\Mp_2(F)$.

Every generic unramified irreducible
representation
$\tau$ of $\PGL_2(F)$
is
of the form
$\mathcal{I}_{\PGL_2}(\chi)$
for some unramified character $\chi$.
Its Waldspurger packet
is the singleton $\{\sigma^+ = \sigma\}$
with $\sigma \cong \mathcal{I}_{\Mp_2}(\chi)$.
In the unramified case, $\sigma$ is unramified.

\subsection{Complementary series\label{sec:local-temperedness}}
\label{sec-3-7}

\subsubsection{Definitions}
Let
$G \in \{\Mp_2(F), \PGL_2(F), \GL_2(F)\}$.
Let
$\pi$ be an irreducible unitary representation of $G$,
assumed genuine in the case $G = \Mp_2(F)$.
For $c \in (0,1/2)$,
we say that $\pi$
is a \emph{complementary series of parameter $c$}
if:\footnote{
  For $G \in \{\PGL_2(F), \GL_2(F)\}$,
  these exhaust the non-tempered generic unitary
  representations.
  For $G = \Mp_2(F)$ and $F \neq \mathbb{C}$,
  they exhaust the non-tempered genuine irreducible unitary
  representations that are not isomorphic to even Weil
  representations $\rho_{\psi '}^+$,
  or equivalently, those irreducible representations belonging
  to the Waldspurger packet
  of a non-tempered generic irreducible unitary representation.
  When $F = \mathbb{C}$, there are also genuine complementary series
  representations of $\Mp_2(\mathbb{C}) \cong \SL_2(\mathbb{C})
  \times \{\pm 1\}$
  of parameter
  $1/2 < c < 1$,
  which play no role here.
}
\begin{itemize}
\item For $G \in \{\Mp_2(F), \PGL_2(F)\}$,
  there is a quadratic character $\eta$ of $F^\times$
  so that
  $\pi \cong \mathcal{I}_G(|.|^c \eta)$.
\item For $G = \GL_2(F)$, there is a unitary character $\xi$ of
  $F^\times$ so that $\pi$ is isomorphic to the unitarily
  normalized induction of $(|.|^c \xi, |.|^{-c} \xi)$.
\end{itemize}

\subsubsection{Temperedness}
For $\vartheta \in (0,1/2)$, we say that $\pi$ is
\emph{$\vartheta$-tempered} if either $\pi$ is tempered
(see \cite{MR946351}) or $\pi$ is a
complementary series of parameter $c \leq \vartheta$.
We henceforth set $\vartheta := 7/64$.
It is known then that if $G \in
\{\PGL_2(F), \GL_2(F)\}$
and $\pi$ is
obtained as the local component of a cuspidal automorphic
representation, then $\pi$ is $\vartheta$-tempered \cite{MR1937203, MR2811610};
if $G = \PGL_2(F)$ and $\sigma \in \Wd_{\psi}(\pi)$,
then it follows from the discussion of \S\ref{sec:local-waldspurger-packets}
that $\sigma$ is $\vartheta$-tempered.

\begin{remark*}
  We make use of the fact that $\vartheta < 1/6$
  to address some technical convergence issues in our local
  arguments.
  Assuming that those issues can be addressed
  more generally,
  it seems that any value $\vartheta < 1/4$
  would suffice for the purposes of proving Theorem
  \ref{thm:main-theorem-1}.
\end{remark*}

\subsubsection{Unitarity}\label{sec:local-compl-series}
Temporarily abbreviate
$\sigma_\chi := \mathcal{I}_{\Mp_2}(\chi)$ for a character
$\chi$ of $F^\times$.
We have sesquilinear pairings
$(,) : \sigma_\chi \otimes \sigma_{\overline{\chi^{-1}}} \rightarrow
\mathbb{C}$
given by $(f_1,f_2) := \int_{x \in F} f_1(n'(x)) \overline{f_2(n'(x))}$.
We noted in \S\ref{sec-3-3-5} that if $\chi$ is unitary,
so that $\chi = \overline{\chi }^{-1}$, then
$\sigma_{\chi}$ is unitary, with invariant norm
$\|f\|^2
= (f,f)$,
which may be expressed
in terms of the Kirillov-type map $f \mapsto \mathcal{K} f$ from \S\ref{sec:metapl-intertw-oper}
as $\|f\|^2 =  \int_{\xi \in F^\times} |\mathcal{K} f(\xi)|^2 \, d \xi$.

Suppose now that $\sigma_\chi$ is a complementary
series of parameter $c \in (0,1/2)$.
We then have $\chi = |.|^{\pm c} \eta$ for some quadratic character $\eta$ of
$F^\times$,
and
$\|f\|^2 := (R_{\chi} f, f)$ ($f \in \sigma_\chi$) defines an
invariant norm
(cf. \S\ref{sec:metapl-intertw-oper} and \cite[p271, p278]{MR2876383});
in terms of $f \mapsto \mathcal{K} f$ from \S\ref{sec:metapl-intertw-oper},
\[
\|f\|^2
= \int_{\xi \in F^\times }
|\mathcal{K} {f}(\xi)|^2
G(\chi,\psi,\xi) \, d \xi.
\]
(Recall from \S\ref{sec:metapl-intertw-oper} that $G(\chi,\psi,\xi) > 0$.)

\subsection{Newvectors\label{sec:local-newvectors}}
\label{sec-3-8}
Assume in \S\ref{sec:local-newvectors} that $F$ is non-archimedean.

\subsubsection{Notation}
\label{sec-3-8-1}
For a generic representation $\pi$ of $\GL_N(F)$,
the conductor $C(\pi)$ may be written $q^{c(\pi)}$
for some $c(\pi) \in \mathbb{Z}_{\geq 0}$.

For $n \in \mathbb{Z}_{\geq 0}$,
denote by $K_0[n]$ the subgroup of $\GL_2(\mathfrak{o})$
consisting of elements with lower-left entry in $\mathfrak{p}^{n}$.

\subsubsection{Summary of newvector theory}
\label{sec-3-8-2}
Let $\pi$ be a generic irreducible representation
of $\GL_2(F)$ with central character $\omega_{\pi} : F^\times
\rightarrow \mathbb{C}^\times$.
It is known that $c(\pi) \geq c(\omega_{\pi})$.
For $n \geq c(\omega_{\pi})$ denote also by $\omega_{\pi}$
the character of $K_0[n]$
given by
\[
\omega_{\pi}(\begin{pmatrix}
  a & b \\
  c & d
\end{pmatrix}) :=
\begin{cases}
  \omega_\pi(d) & \text{ if } d \in \mathfrak{o}^\times, \\
  1 & \text{ otherwise}
\end{cases}
\]
and by $\pi[n]$ the space
of vectors $\varphi \in \pi$
satisfying
$\pi(g) \varphi = \omega_{\pi}(g) \varphi$
for all $g \in K_0[n]$.
It is known that
$\dim \pi[n] = \min(0,1 + n - c(\pi))$.
A \emph{newvector} is a nonzero element of the one-dimensional space
$\pi[c(\pi)]$.

\subsubsection{Principal series}
\label{sec-3-8-3}
Let $\chi$ be a character of $F^\times$
with $\chi^2 \neq |.|^{\pm 1}$.
Then $\pi := \mathcal{I}_{\PGL_2}(\chi)$
is irreducible and generic with $c(\pi) = 2 c(\chi)$.

\subsubsection{The special representation}
\label{sec-3-8-4}
Let $\Sp$ denote the special
representation of $\PGL_2(F)$; it is the irreducible quotient of
$\mathcal{I}_{\PGL_2}(|.|^{1/2})$
and satisfies $c(\pi) = 1$.
Assume that $\psi$ is unramified.
A newvector may then be given in the Kirillov model of $\pi$
by the formula $W(a(y)) := 1_{\mathfrak{o}}(y) |y|$
(see for instance \cite{Sch02}).

\subsection{Sobolev norms}
\label{sec-3-9}
Given an irreducible unitary representation $\pi$ of
one of the groups $\SL_2(F), \Mp_2(F), \PGL_2(F), \GL_2(F)$,
we define a family of Sobolev norms $\mathcal{S}_d$ ($d \in
\mathbb{R}$)
on $\pi$ following the recipe of   \cite[\S2]{michel-2009} and \cite[\S4.6, \S5.3]{nelson-theta-squared}.
They have the form $\mathcal{S}_d(v) := \|\Delta^d v\|$,
where $\Delta$ is a certain invertible positive self-adjoint linear operator.
We write $\mathcal{S}$ to denote
a Sobolev norm of the form $\mathcal{S}_d$ for some
unspecified fixed large enough $d$
(the ``implied index'').

\begin{remark*}
  To follow the main arguments of the paper,
it suffices to keep in mind the following property
of the Sobolev norms: if $F$ is
non-archimedean and $\varphi$ is a vector invariant by the $r$th
standard principal congruence subgroup, then
$\mathcal{S}(\varphi) = \|\varphi \| q^{O(r)}$.
\end{remark*}

\subsection{Change of polarization\label{sec:local-I-chi}}
\label{sec-3-10}
Let $\chi$
be a character of $F^\times$ with $\Re(\chi) > -1$.
In \cite[\S4.12]{nelson-theta-squared},
we defined an $\Mp_2(F)$-equivariant map
$I_\chi : \rho_{\psi} \otimes \overline{\rho_{\psi}}
\rightarrow \mathcal{I}_{\SL_2}(\chi)$.
The definition is not important
for our present purposes,
but we record it for convenience:
for $\phi \in \rho_{\psi} \otimes \overline{\rho_\psi}$,
the function
$I_\chi(\phi) : \SL_2(F) \rightarrow \mathbb{C}$
is given by the Tate integral
\[
\mathcal{I}_\chi(\phi)(\sigma)
=
\int_{y \in k^\times} |y| \chi(y)
\mathcal{F} \phi(y e_2 \sigma) \, d^\times y,
\]
where
$\mathcal{F} \phi$ is the partial
Fourier transform
\[
\mathcal{F} \phi(y_1,y_2)
:=
\int_{t \in k}
\phi(\rho(y_1,t))
\psi(y_2 t) \, d t
=
\int_{t \in k}
\phi(\frac{y_1 + t}{2}, \frac{y_1 - t}{2})
\psi(y_2 t) \, d t.
\]
For our purposes,
the relevant properties of the intertwiner $I_\chi$
are the following
(see \cite[\S4.12]{nelson-theta-squared}):
\begin{itemize}
\item If we are in the unramified case
  and if $\chi$ is unramified,
  then 
  $L(\chi,1)^{-1} I_\chi$ preserves unramified elements;
  more precisely, it sends $1_\mathfrak{o} \otimes
  1_\mathfrak{o}$
  to the unramified vector
  taking the value $1$ at the identity.
\item If $\chi$ is unitary, then
  for each fixed $d$ one has
  $\mathcal{S}_{d}(I_\chi(\phi_1 \otimes \phi_2))
  \ll \mathcal{S}(\phi_1) \mathcal{S}(\phi_2)$.
\end{itemize}
\subsection{Bounds for matrix coefficients and varia\label{sec:bounds-mx-coefs}}
\label{sec-3-11}

\subsubsection{}
Let
$\Xi : \PGL_2(F) \rightarrow \mathbb{R}^\times_+$ denote the
Harish--Chandra function for $\PGL_2(F)$, i.e., the
matrix coefficient of the spherical vector in the
normalized induction of the trivial character of the Borel,
normalized so that $\Xi(1) = 1$.
Explicitly,
if we define $\htt : \PGL_2(F) \rightarrow \mathbb{R}^\times_+$
using the Iwasawa decomposition by the formula
$\htt(n(x) a(y) k) := |y|$,
then
\[
\Xi(g) = \int_{k \in K_{\PGL_2(F)}} \htt(k g)^{1/2},
\]
where the integral is taken with respect to the probability Haar
on the standard maximal compact subgroup $K_{\PGL_2(F)}$ of $\PGL_2(F)$.
One has
\begin{equation}
  \|\Ad(s)\|^{-1} \ll \Xi(s) \ll \|\Ad(s)\|^{-1+\eps}.
\end{equation}
More precisely,
with respect to the Cartan decomposition $s = k_1 a(y) k_2$,
one has $\|\Ad(s)\| \asymp t := |y| + |y|^{-1}$
and $\Xi(s) \asymp t^{-1} \log(t)$.

We denote also by $\Xi$ its pullback to $\SL_2(F)$ or to
$\Mp_2(F)$.  (The pullback to $\SL_2(F)$ is the Harish--Chandra
function for $\SL_2(F)$.)

By \cite{MR946351},
the function $\Xi$ controls the matrix coefficients of any
tempered irreducible unitary representation $\pi$ of
$G \in \{\Mp_2(F), \GL_2(F)\}$: for $s \in G$
and $\varphi, \varphi ' \in \pi$,
one has $\langle \pi(s) \varphi, \varphi ' \rangle
\ll \Xi(s) \mathcal{S}(\varphi) \mathcal{S}(\varphi ')$.
We record below the generalizations and refinements of this
estimate in the $\vartheta$-tempered case.

One has $\int_G \Xi^{2+\eps} < \infty$ for any $G$ as above.

\subsubsection{}
Let $\pi$ be a $\vartheta$-tempered irreducible unitary
representation of $\GL_2(F)$;
in particular, $\pi$ is generic.
For each $W$ in the Whittaker model
$\mathcal{W}(\pi,\psi')$
corresponding to some fixed nontrivial unitary character
$\psi '$ of $F$,
we have (see \cite[3.2.3]{michel-2009})
\begin{equation}\label{eqn:W-GL2-bound}
  W(a(y))
  \ll
  \min(|y|^{1/2-\vartheta},|y|^{-A}) \mathcal{S}(W).
\end{equation}
Let $\varphi, \varphi' \in \pi$
and $s \in \GL_2(F)$.  Then
(see  \cite[\S2.5.1]{michel-2009})
\begin{equation}\label{eq:matrix-coeff-bound-theta-tempered-gl2}
  \langle \pi(s) \varphi, \varphi' \rangle \ll \Xi(s)^{1-2 \vartheta}
  \mathcal{S}(\varphi)  \mathcal{S}(\varphi').
\end{equation}
Since $\Xi(s) \ll \log(t) / t^{1/2}$ with
$t := 3 + \|\Ad(s)\|$,
it follows that
\begin{equation}
  \langle \pi(s) \varphi, \varphi' \rangle \ll \|\Ad(s)\|^{-1/2 + \vartheta + \eps}
  \mathcal{S}(\varphi)  \mathcal{S}(\varphi').
\end{equation}

\subsubsection{}
Let $\phi, \phi' \in \rho_{\psi}$.
It follows from the Sobolev lemma (see
\cite[\S4.9]{nelson-theta-squared})
that
$\|\phi \|_{L^\infty}, \|\phi \|_{L^1} \ll \mathcal{S}(\phi)$,
hence
for $y \in F^\times$
that
\begin{equation}\label{eqn:V-phi-bound}
  \rho_{\psi}(t(y)) \phi (1)
  =
  \chi_\psi(y) |y|^{1/2} \phi(y)
  \ll \min(|y|^{1/2}, |y|^{-A}) \mathcal{S}(\phi)
\end{equation}
and
\begin{equation}
  \langle \rho_{\psi}(t(y)) \phi, \phi ' \rangle
  \ll (|y| + |y|^{-1})^{-1/2}
  \mathcal{S}(\phi) \mathcal{S}(\phi ').
\end{equation}
Using the Cartan decomposition on
$\Mp_2(F)$, it follows that
\begin{equation}
  \langle \rho_{\psi}(s) \phi, \phi' \rangle \ll \|\Ad(s)\|^{-1/4}
  \mathcal{S}(\phi) \mathcal{S}(\phi')
  \ll
  \Xi(s)^{1/2}
  \mathcal{S}(\phi) \mathcal{S}(\phi').
\end{equation}

\subsubsection{}
Let $\sigma$ be a unitarizable $\vartheta$-tempered principal
series representation of $\Mp_2(F)$, thus
$\sigma = \mathcal{I}_{\Mp_2}(\chi)$ where $\chi$ is either
unitary or is of the form $\eta |.|^c$ with $\eta$ quadratic and
$0 < c \leq \vartheta$.
Let $f \in \sigma$.
Denote by $K$ the standard maximal compact
subgroup of $\SL_2(F)$.
Then
$\sup_K |f|, \sup_K |R_\chi f| < \infty$,
hence for $g \in \Mp_2(F)$,
\begin{equation}\label{eq:needed-for-convergence}
  \int_{k \in K} |f(k g) R_\chi f(k)| \ll_f
  \int_{k \in K} (\htt(k g)^{1/2})^{1 + c}
  \ll \Xi(g)^{1 - \vartheta}.
\end{equation}
(One can refine this estimate to feature a Sobolev--type
dependence upon $f$, but we require
\eqref{eq:needed-for-convergence} only for establishing that certain
integrals converge absolutely.)

\subsubsection{}
Let $\sigma$ be an element of the Waldspurger packet
$\Wd_{\psi}(\tau)$ of some $\vartheta$-tempered generic
irreducible unitary representation $\tau$ of $\PGL_2(F)$, so that
$\sigma$ is $\vartheta$-tempered.
Then for
$\varphi, \varphi ' \in \sigma$ and $s \in \Mp_2(F)$, we have
\begin{equation}
  \langle \sigma(s) \varphi, \varphi ' \rangle
  \ll \Xi(s)^{1 - \vartheta} \mathcal{S}(\varphi) \mathcal{S}(\varphi ')
\end{equation}
This may be proved as in the case of $\GL_2(F)$,
see \cite[\S9.1.1]{venkatesh-2005}.

\subsection{Local triple product periods on the general linear group}
\label{sec-3-12}
Let $\pi_1, \pi_2, \pi_3$ be $\vartheta$-tempered generic irreducible unitary
representations of $\GL_2(F)$ with trivial product of central characters.
For $\varphi_i \in \pi_i$ ($i=1,2,3$), set
\[
\mathcal{P}_{\PGL_2(F)}(\varphi_1,\varphi_2,\varphi_3)
:=
\int_{\PGL_2(F)}
\prod_{i=1,2,3}
\langle \pi_i(g) \varphi_i, \varphi_i \rangle.
\]
The integral converges absolutely (see \cite[Lem 2.1]{MR2449948}).
If we are in the unramified case
and the $\varphi_i$ are unramified unit
vectors,
then (see \cite[Lem 2.2]{MR2449948})
\begin{equation}\label{eqn:P-PGL2-unram-calc}
  \mathcal{P}_{\PGL_2(F)}(\varphi_1,\varphi_2,\varphi_3)
=
\frac{\zeta_F(2)^{2}
  L(\pi_1 \otimes \pi_2 \otimes \pi_3, \tfrac{1}{2})
}{\prod_{i=1,2,3} L(\ad(\pi_i), 1)}.
\end{equation}
We note that some references work with the variant of
$\mathcal{P}_{\PGL_2(F)}$ obtained by dividing through by the
RHS of \eqref{eqn:P-PGL2-unram-calc}, which is thus normalized
to take the value $1$ on unramified unit vectors in the
unramified case.  This difference in normalization has no effect
on asymptotics, since the $\vartheta$-temperedness assumption
implies that the RHS of \eqref{eqn:P-PGL2-unram-calc} has size
$\asymp 1$.
\subsection{Local triple product periods on the metaplectic group}
\label{sec-3-13}
\label{sec:triple-integrals-matrix-coefficients-mp2}
In this section (and others to follow), we consider the following trios of representations:
\begin{itemize} \item $\pi$: a $\vartheta$-tempered generic irreducible unitary
  representation of $\GL_2(F)$.
\item $\rho_{\psi}$: the Weil representation
  of $\Mp_2(F)$ on $\mathcal{S}(F)$.
\item
  $\sigma$: an element of the Waldspurger packet
  $\Wd_{\psi}(\tau)$
  of some $\vartheta$-tempered generic irreducible unitary
  representation $\tau$ of $\PGL_2(F)$.
\end{itemize}
For $\varphi_1 \in \pi$, $\varphi_2 \in \rho_{\psi}$ and $\varphi_3 \in \sigma$, set
\[
\mathcal{P}_{\SL_2(F)}(\varphi_1,\varphi_2,\varphi_3)
:=
\int_{\SL_2(F)}
\langle \pi(g) \varphi_1, \varphi_1 \rangle
\langle \rho_{\psi}(g) \varphi_2, \varphi_2 \rangle
\overline{\langle \sigma(g) \varphi_3, \varphi_3 \rangle}
\]
The integral converges absolutely (see \cite[Lem 4.3]{MR3291638}) because $\vartheta < 1/6$.
If we are in the unramified
case and the $\varphi_i$ are unramified unit
vectors,
then (see \cite[Lem 4.4]{MR3291638})
\begin{equation}\label{eqn:unram-trip-prod-metaplectic}
  \mathcal{P}_{\SL_2(F)}(\varphi_1,\varphi_2,\varphi_3)
  =
  \frac{\zeta_F(2)
    L(\ad(\pi) \otimes \tau, \tfrac{1}{2})
  }{L(\ad(\tau), 1) L(\ad(\pi), 1)}.
\end{equation}
In general, we may write $\sigma = \sigma^{\eps}$
where $\eps \in \{\pm \}$ indexes the Waldspurger packet
$\Wd_{\psi}(\tau)$ as in \S\ref{sec:local-waldspurger-packets}.
Then  (see \cite[Prop 8]{MR3291638})
\begin{equation}\label{eqn:local-nonvanishing-of-sl2-periods}
  \text{$\mathcal{P}_{\SL_2(F)}$ does not vanish on $\pi \otimes
    \rho_{\psi} \otimes \sigma$}
  \iff
  \eps = \eps(\pi \otimes \overline{\pi} \otimes \tau,\tfrac{1}{2} ).
\end{equation}
Moreover, if $\eps := \eps(\pi \otimes \overline{\pi} \otimes
\tau,\tfrac{1}{2} )$,
then $\Wd_{\psi}(\pi)$
contains $\sigma^\eps$;
equivalently, if $\eps = -1$,
then $\tau$ belongs to the discrete series
(see \S\ref{sec:local-waldspurger-packets}).

\subsection{Linearizing local triple product periods on the metaplectic group\label{sec:linearize-local-shimura}}
\label{sec-3-14}
\subsubsection{Setting and aim}
\label{sec-3-14-1}
Let $\pi, \tau$ and $\sigma \in \Wd_{\psi}(\tau)$ be as in \S\ref{sec:triple-integrals-matrix-coefficients-mp2},
but suppose now that $\tau$ and hence $\sigma$ belongs to the principal series.
Then $\Wd_{\psi}(\tau) = \{\sigma = \sigma^+\}$ is a singleton
and $\eps( \pi \otimes \overline{\pi}\otimes \tau) = +1$,
so by \eqref{eqn:local-nonvanishing-of-sl2-periods}, the local Shimura period $\mathcal{P}_{\SL_2(F)}$
does not vanish identically on $\pi \otimes \rho_{\psi} \otimes \sigma$.
The purpose of this section is to express the nonzero invariant
hermitian form $\mathcal{P}_{\SL_2(F)}$
in terms of explicit invariant trilinear forms.
In practice,
it is simpler
to analyze asymptotically
the linear forms
than the hermitian form.
\subsubsection{Local metaplectic Rankin--Selberg integral}
\label{sec-3-14-2}
We realize $\pi$ in its $\overline{\psi}$-Whittaker model,
$\rho_{\psi}$ on $\mathcal{S}(F)$ as usual,
and $\sigma$ in its induced model
$\sigma = \mathcal{I}_{\Mp_2}(\chi)$
for some character $\chi$ of $F^\times$ with $|\Re(\chi)| \leq
\vartheta$.
Set $G := \SL_2(F)$
and let $N = \{ n(x) : x \in F\} \leq G$ denote the standard
upper-triangular unipotent subgroup.
Equip $N$ with the Haar transported from $F$
and $N \backslash G$ with the quotient Haar.
Define $\ell : \pi \otimes \rho_{\psi} \otimes \sigma
\rightarrow \mathbb{C}$
by
\[
\ell(W,\phi,f)
:=
\int_{g \in N \backslash G}
W(g)
\cdot
(\rho_{\psi}(g) \phi)(1)
\cdot \overline{f(g)}
\]
This is a local integral of Rankin--Selberg type
that was studied by Gelbart--Jacquet, following Shimura.
\begin{lemma*}\label{lem:Absolute-convergence-ell-as-in-GJ}
  The integral defining $\ell$ converges absolutely.
\end{lemma*}
\begin{proof}
  Write $g = t(y) k$, so that $d g$ is a constant multiple
  of $|y|^{-2} \, d^\times y \, d k$.
  Write $c := \Re(\chi)$,
  so that $|c| \leq \vartheta$.
  Then for fixed $W, \phi$ and $f$,
  we have $W(g) \ll \min(|y|^{1-2 \vartheta}, |y|^{-A})$,
  $\rho_\Psi(g) \phi(1) \ll \min(|y|^{1/2}, |y|^{-A})$
  and $f(g) \ll |y|^{1+c}$.
  The integral in question is thus dominated by
  $\int_{y \in F^\times}
  \min(|y|^{(1 - 2 \vartheta) + 1/2 + (1 - \vartheta)},|y|^{-A}) \, d^\times y
  < \infty$,
  since $\vartheta < 1/6$.
\end{proof}
Thus
$\ell$ defines an invariant trilinear form.
\subsubsection{Statement of result}
\label{sec-3-14-3}
The (first part of) the following lemma
may be understood as a metaplectic analogue
of the oft-cited \cite[Lem 3.4.2]{michel-2009}.
\begin{lemma*} There exists $c > 0$ with $c \asymp 1$
  so that:
  \begin{itemize}
  \item If $\sigma$ is tempered, then
    \begin{equation}\label{eq:linearized-local-shimura-period}
      \mathcal{P}_{\SL_2(F)}(W, \phi, f)
      =
      c
      \left\lvert
        \ell(W,\phi,f)
      \right\rvert^2.
    \end{equation}
  \item If $\sigma$ is non-tempered,
    so that (without loss of generality)
    $\chi = \eta |.|^{\sigma}$
    for some quadratic character $\eta$ of $F^\times$
    and some $0 < \sigma \leq \vartheta$,
    then
    \begin{equation}\label{eq:linearized-local-shimura-period-nontempered}
      \mathcal{P}_{\SL_2(F)}(W, \phi, f)
      =
      c \,
      \ell(W,\phi,f)
      \overline{\ell(W,\phi,R_\chi f)}.
    \end{equation}
  \end{itemize}
\end{lemma*}

\begin{remark*}~
  \begin{enumerate}[(i)]
  \item
    We do not know how to adapt the method of proof of \cite[Lem
    3.4.2]{michel-2009} to the metaplectic setting, so we proceed
    more directly. The method developed here may be applied also
    to give a new (and perhaps more natural) proof 
    of \emph{loc. cit.} that avoids discussion
    of Plancherel measure and continuity in the tempered dual.
  \item The analogue of
    \eqref{eq:linearized-local-shimura-period-nontempered} is
    not considered in \cite[Lem 3.4.2]{michel-2009}:
    in that work, the induced representation is always tempered,
    since it
    arises ultimately as the local component of a unitary Eisenstein
    series.
  \end{enumerate}
\end{remark*}

More precisely, the quantity $c$ depends upon the normalizations of Haar measures.
Suppose for the remainder of \S\ref{sec:linearize-local-shimura} that
\begin{itemize}
\item the Haar $d x$ on $F$ is $\psi$-self-dual (as we have already assumed),
\item the Haar $d^\times y$ on $F^\times$ and the Haar $d x$ on $F$
  are related by $d^\times y = |y|^{-1} d y$, and
\item the Haar on $G$
  has been normalized so
  that for $\alpha \in C_c(N \backslash G)$,
  \[
    \int_{N \backslash G} \alpha  =
    \int_{y \in F^\times} \int_{x \in F} \alpha(t(y) n'(x)) |y|^{-2} \, d^\times y \, d x.
  \]
\end{itemize}
(In the non-archimedean case,
these measures assign volume $\asymp 1$ to maximal compact
subgroups.)
Recall also that we have normalized
the unitary structures on $\pi, \rho_{\psi}, \sigma$ to be given respectively
in the $\overline{\psi}$-Kirillov model, on $\mathcal{S}(F)$,
and in the induced model as in \S\ref{sec-3-3-5},
\S\ref{sec:local-compl-series}.
Under these assumptions, we shall verify that
the above identities hold with $c = 1$.
\subsubsection{Reduction to an identity}
\label{sec-3-14-4}
We now reduce the proof of
\eqref{eq:linearized-local-shimura-period}
and \eqref{eq:linearized-local-shimura-period-nontempered}
to that of a common identity.
Define $f_1, f_2 : G \rightarrow \mathbb{C}$ in the tempered case
by $f_1 := f_2 := f$
and in the non-tempered case
by $f_1 := f$ and $f_2 := R_\chi f$,
so that $\langle g f, f \rangle = (g f_1, f_2)$.
Consider the functions
$\Psi_1 \Psi_2 : \Mp_2(F) \rightarrow \mathbb{C}$ defined by
$\Psi_i(g) := \rho_{\psi}(g) \phi(1) \cdot \overline{f_i(g)}$;
they descend to functions $\Psi_i : \SL_2(F) \rightarrow
\mathbb{C}$
satisfying
$\Psi_i(n(x) g) = \psi(x) \Psi_i(g)$.
Set
\begin{equation}\label{eq:gPsi-Psi}
  \langle g \Psi_1, \Psi_2 \rangle
  :=
  \int_{h \in N \backslash G} \Psi_1(h g) \overline{\Psi_2}(h).
\end{equation}
By writing $h = t(y) n'(x)$,
we see that \eqref{eq:gPsi-Psi}
converges absolutely
and evaluates to
$\langle g \Psi_1, \Psi_2 \rangle
= \langle g \phi, \phi  \rangle \overline{\langle g f, f
  \rangle}$,
so that
$\mathcal{P}_{\SL_2(F)}(W,\phi,f)
=
\int_{g \in G}
\langle g W, W \rangle \langle g \Psi_1, \Psi_2 \rangle$
(compare with
\cite[Lem 3.2.7]{michel-2009}).
Our goals
\eqref{eq:linearized-local-shimura-period}
and \eqref{eq:linearized-local-shimura-period-nontempered}
thereby reduce to the identity of absolutely-convergent integrals
\begin{equation}\label{eq:linearized-local-shimura-period-2}
  \int_{g \in G}
  \langle g W, W \rangle
  \langle g \Psi_1, \Psi_2 \rangle
  = 
  (\int_{N \backslash G} W \Psi_1)
  \overline{(\int_{N \backslash G} W \Psi_2)}.
\end{equation}
\subsubsection{Heuristic argument}
\label{sec-3-14-5}
The LHS
of \eqref{eq:linearized-local-shimura-period-2}
formally expands to
\begin{equation}\label{eq:intermediary-divergent-step}
  \int_{g \in G}
  \int_{h \in N \backslash G}
  \Psi_1(h g) \overline{\Psi_2}(h)
  \langle g W, W \rangle
\end{equation}
and then folds after the substitution $g \mapsto h^{-1} g$ to
\begin{equation}\label{eq:intermediary-divergent-step-2}
  \int_{g \in N \backslash G}
  \int_{h \in N \backslash G}
  \Psi_1(g) \overline{\Psi_2}(h)
  \int_{x \in F}
  \psi(x)
  \langle h^{-1} n(x) g W, W \rangle,
\end{equation}
which
evaluates to the RHS
of \eqref{eq:linearized-local-shimura-period-2}
by 
the identity
\[\langle h^{-1} n(x) g W, W \rangle
  = \langle n(x) g W, h W \rangle,\]
the normalization
$d^\times y = |y|^{-1} \, d y$,
and the following Fourier inversion formula
stated in the proof of
\cite[Lem 3.4.2]{michel-2009}:
\begin{equation}\label{eq:fourier-transform-Whittaker}
  \int_{x \in F}
  \psi(x)
  \langle n(x) g W, h W \rangle
  = W(g) \overline{W(h)}
\end{equation}
(Recall here that $W$ belongs to the $\overline{\psi}$-Whittaker
model.)

Unfortunately,
the intermediary expressions \eqref{eq:intermediary-divergent-step},
\eqref{eq:intermediary-divergent-step-2}
(and the LHS of \eqref{eq:fourier-transform-Whittaker})
do not in general converge absolutely, so
further care is required
to convert the
argument sketched here
into a proof.
\subsubsection{Proof of the identity}
\label{sec-3-14-6}
We now prove \eqref{eq:linearized-local-shimura-period-2},
roughly following
the heuristic argument indicated above.
For $\Phi \in \mathcal{S}(F)$,
the quadratic Fourier transform
\begin{equation}\label{eqn:from-Phi-to-P}
  P(x) :=
  \int_{y \in F}
  \psi(-y^2 x)
  \Phi(y) \, d y
\end{equation}
defines a smooth function $P : F \rightarrow \mathbb{C}$.
% which satisfies $P(x) \ll (1 + |x|)^{-1/2}$ (for $\Phi$ fixed, $x \in F$ varying).
\begin{lemma}\label{lem:quadratic-fourier-isometry-property}
  For $W_1, W_2$ in the $\overline{\psi}$-Whittaker
  model of $\pi$,
  one has the identity of absolutely-convergent integrals
  \begin{equation}\label{eqn:quadratic-fourier-isometry-property}
    \int_{x \in F}
    \langle n(x) W_1, W_2 \rangle P(x) \, d x
    =
    \int_{y \in F^\times}
    W_1(t(y))
    \overline{W_2}(t(y))
    \Phi(y)
    \, \frac{d^\times y}{|y|}.
  \end{equation}
\end{lemma}
\begin{proof}
  This is inspired by a lemma of Qiu \cite[Lem
  3.5]{2013arXiv1308.2353Q} concerning representations of
  metaplectic groups, and may be proved similarly.
  For completeness, we record a proof.
  Observe first that if $\Phi$ is an odd function, then $P = 0$,
  and so both sides of
  \eqref{eqn:quadratic-fourier-isometry-property}
  converge absolutely and vanish identically.
  We may thus assume that $\Phi$ is an even function.
  We treat first the special case in which
  $\Phi$ vanishes in a neighborhood
  of $0$.
  We may then define a Schwartz function $Q \in \mathcal{S}(F)$
  supported on the nonzero squares $y^2$ ($y \in F^\times$)
  by the formula $Q(y^2) := \Phi(y)$.
  Setting $z := y^2$,
  we have
  $d z = |2|_F |y| \, d y$,
  and the map $y \mapsto z$ is two-to-one on $F^\times$,
  hence
  \begin{equation}\label{eqn:P-vs-Q}
    P(x) =
    \frac{2}{|2|_F}
    \int_{z \in F}
    \psi(-z x) \frac{Q(z)}{|z|^{1/2}} \, d z.
  \end{equation}
  In particular, $P \in \mathcal{S}(F)$.
  Similarly, since $W_1(t(y)) \overline{W_2(t(y))} = W_1(a(y^2))
  \overline{W_2(a(y^2))}$,
  the RHS of \eqref{eqn:quadratic-fourier-isometry-property}
  may be written
  \[
  \frac{2}{|2|_F}
  \int_{z \in F^\times}
  W_1(a(z))
  \overline{W_2}(a(z))
  \frac{Q(z)}{|z|^{1/2}}
  \, d^\times z.
  \]
  By Fourier inversion on $\mathcal{S}(F)$,
  it suffices then to verify for each $P \in \mathcal{S}(F)$
  that
  \begin{equation}\label{eqn:fourier-inversion-actual-for-W1-W2}
    \int_{x \in F}
    \langle n(x) W_1, W_2 \rangle
    P(x)
    =
    \int_{z \in F^\times}
    W_1(a(z))
    \overline{W_2}(a(z))
    (\int_{x \in F} P(x) \psi(-z x) \, d x)
    \, d^\times z.
  \end{equation}
  For this we expand
  $\langle n(x) W_1, W_2 \rangle = \int_{z \in F^\times}
  \psi(-z x) W_1(a(z)) \overline{W_2}(a(z))
  \, d z$ on the LHS of \eqref{eqn:fourier-inversion-actual-for-W1-W2};
  this gives an absolutely-convergent double integral
  which rearranges to the RHS of
  \eqref{eqn:fourier-inversion-actual-for-W1-W2}.

  We turn to the general case.  We may smoothly decompose $P$ as
  a sum of a function supported away from the origin -- to which
  the previous paragraph applies -- and a function supported near
  the origin.
  By this reduction,
  we may assume that  $\Phi$ is supported on $\{x : |x| \leq 1\}$,
  say.
  We fix a smooth function $\nu \in
  C_c^\infty(\mathbb{R}^\times_+)$
  so that $\sum_{j \geq 0} \nu(2^j t) = 1$ for $0 < t \leq 1$.
  For $j \geq 0$,
  set $\Phi_j(x) := \nu(2^j |x|) \Phi(x)$.
  Then $\Phi = \sum_{j} \Phi_j$;
  moreover,
  $\int_{y \in F} \sum_j |\Phi_j(y)| \, d y < \infty$,
  so that $P = \sum_j P_j$ pointwise
  with $P_j$ attached to $\Phi_j$
  as in \eqref{eqn:from-Phi-to-P}.
  The previous paragraph shows that
  the desired identity
  \eqref{eqn:quadratic-fourier-isometry-property}
  is satisfied by each pair $(P_j,\Phi_j)$,
  so to verify the corresponding identity
  for $(P,\Phi)$
  it suffices by Fubini to check that
  \begin{equation}\label{eqn:abs-conv-involving-P-j}
    \sum_j \int_{x \in F - \{0\}}
    |\langle n(x) W_1, W_2 \rangle P_j(x)| \, d x < \infty
  \end{equation}
  and
  \begin{equation}\label{eqn:abs-conv-involving-Phi-j}
    \sum_j \int_{y \in F^\times}
    |W_1(t(y) W_2(t(y)) \Phi_j(y)| \, \frac{d^\times y}{|y|} < \infty.
  \end{equation}

  For the proofs
  of \eqref{eqn:abs-conv-involving-P-j}
  and \eqref{eqn:abs-conv-involving-Phi-j},
  we temporarily replace our general conventions
  (\S\ref{sec-3-1-9})
  on asymptotic notation with the following:
  ``fixed'' means ``depending at most upon $W_1, W_2$, and
  $\Phi$,''
  and asymptotic notation is
  then defined as in \S\ref{sec-3-1-9}.
  In particular, implied constants are independent of $j$.

  The proof of
  \eqref{eqn:abs-conv-involving-P-j}
  reduces,
  via
  the estimate $\langle n(x) W_1, W_2 \rangle \ll (1 + |x|)^{-(1 - 2
    \vartheta)}$
  and the inequality $\vartheta < 1/4$,
  to verifying that
  \begin{equation}\label{eqn:key-estimate-P-j}
    P_j(x) \ll 2^{-j} (1 + |2^{-2 j} x|)^{-10}.
  \end{equation}
  To that end,
  let $Q_j$ be attached to $\Phi_j$ as above.
  Then:
  \begin{itemize}
  \item
    $Q_j(z) \neq 0$ only if $|z| \asymp 2^{-2 j}$.
  \item
    In the non-archimedean case,
    $Q_j$ is invariant under dilation by a fixed open subgroup of
    $\mathfrak{o}^\times$.
    In the archimedean case,
    we have, for each fixed invariant differential operator
    $\mathcal{D}$ on $F^\times$,
    that $\int_{F^\times} |\mathcal{D} Q_j(y)| \, d^\times y
    \ll 1$.
  \end{itemize}
  We deduce \eqref{eqn:key-estimate-P-j}
  from \eqref{eqn:P-vs-Q}
  via these observations and elementary Fourier analysis.
  
  To establish \eqref{eqn:abs-conv-involving-Phi-j},
  we use that $W_i(t(y)) \ll \min(|y|^{1-2 \vartheta},
  |y|^{-10})$,
  that $\Phi_j(y)$ is supported on $y \asymp 2^{-j}$,
  and that $\vartheta < 1/4$.
\end{proof}
For $g,h \in G$
and $y \in F$,
define
\begin{equation*}
  I(g,h;y) :=
  \rho_{\psi}(g) \phi(y)
  \cdot 
  \overline{f_1(g)}
  \cdot 
  \overline{\rho_\psi(h) \phi(y)}
  \cdot 
  f_2(h).
\end{equation*}
Then
$I(g,h;\cdot) \in \mathcal{S}(F)$.
For $y \in F^\times$,
we verify readily
that
\[
I(g,h;y)
=
\Psi_1(t(y) g)
\overline{\Psi_2}(t(y) h)
|y|^{-3}.
\]
Set
\begin{equation}
  I(g,h) :=
  \int_{y \in F}
  I(g,h;y) \, d y
  =
  \langle g \phi, h \phi  \rangle
  \overline{f_1(g)} f_2(h).
\end{equation}
The following may be understood as a refinement of the absolute
convergence of
the matrix coefficient integral $\mathcal{P}_{\SL_2(F)}$:
\begin{lemma}\label{lem:annoying-abs-conv-for-local-linearization-metaplectic}
  $\int_{g \in G}
  \int_{k \in K}
  \left\lvert
    \langle g W, W \rangle
    I(k g, k)
  \right\rvert
  < \infty$.
\end{lemma}
\begin{proof}
  We have
  $I(k g, k) = \langle g \phi, \phi \rangle \overline{f_1( k g)}
  f_2(k)$,
  so that by \eqref{eq:needed-for-convergence},
  $I(k g, k) \ll_{\phi,f} \Xi(g)^{3/2 - \vartheta}$;
  since $\langle g W, W \rangle \ll_W \Xi(g)^{1-2 \vartheta}$
  (see \eqref{eq:matrix-coeff-bound-theta-tempered-gl2})
  and $3 \vartheta < 1/2$,
  the integral in question is dominated for some $\eps > 0$
  by $\int_G \Xi^{2+\eps} < \infty$.
\end{proof}

Turning to \eqref{eq:linearized-local-shimura-period-2}, note
first that its validity is independent of the choice of Haar
measure on $G$.  Let us suppose (for convenience) that
it is given in Iwasawa coordinates $g = n(x) t(y) k$ by
$d g= |y|^{-2} d x \, d^\times y \, d k$, where $d k$ denotes
any Haar measure on the standard maximal compact subgroup $K$ of
$G$.
The LHS
of \eqref{eq:linearized-local-shimura-period-2} may then be written
\begin{equation}
  \int_{g \in G}
  \langle g W, W \rangle
  (\int_{k \in K}
  I(k g, k)  ).
\end{equation}
By Lemma
\ref{lem:annoying-abs-conv-for-local-linearization-metaplectic},
we may rearrange this to
$\int_{k \in K}
\int_{g \in G}
\langle g W, W \rangle
I(k g, k)$.
The substitution
$g \mapsto k^{-1} g$
yields
$\int_{k \in K}
\int_{g \in G}
\langle g W, k W \rangle
I(g, k)$.
By folding up the $g$-integral
and using the identity
$I(n(x) g,k)
=
\int_{y \in F}
\psi(-x y^2)
I(g,k;y) \, d y$,
we obtain
\begin{equation}
  \int_{k \in K}
  \int_{g \in N \backslash G}
  \int_{x \in F}
  \langle n(x) g W, k W \rangle
  (
  \int_{y \in F}
  \psi(-x y^2)
  I(g,k;y) \, d y
  ).
\end{equation}
By Lemma \ref{lem:quadratic-fourier-isometry-property},
this evaluates to the (absolutely convergent) integral
\begin{equation}
  \int_{k \in K}
  \int_{g \in N \backslash G}
  \int_{y \in F^\times}
  W(t(y) g) \overline{W}(t(y) k)
  I(g,k;y) \, \frac{d^\times  y}{|y|}.
\end{equation}
We swap the $g$ and $y$ integrals,
apply the substitution $g \mapsto t(y)^{-1} g$,
and expand the definition of $I(g,k;y)$
to arrive at
\begin{equation}
  \int_{g \in N \backslash G}
  W(g) \Psi_1(g)
  \int_{k \in K}
  \int_{y \in F^\times}
  \overline{W(t(y) k) \Psi_2(t(y) k) } \, \frac{d^\times  y}{|y|^2},
\end{equation}
which equals the RHS of
\eqref{eq:linearized-local-shimura-period-2}.

\subsection{Lower bounds at uninteresting places\label{sec:uninteresting-lower-bounds}}
\label{sec-3-15}

\subsubsection{Statement of result}
\label{sec-3-15-1}
We record here some unsurprising polynomial-type lower bounds
for local triple product periods on the metaplectic group.
They supply the polynomial dependence of our main result
on the ``essentially fixed'' quantities.

Here we use the specialized notation
$A \ll_{\pi,\tau} B$ or $B \gg_{\pi,\tau} A$
to signify that
$A \ll (C(\pi) C(\tau))^{O(1)} B$,
where $C(.)$ denotes the analytic conductor as in
\cite[\S3.1.8]{michel-2009},
and
$A \asymp_{\pi,\tau} B$ to denote that $A \ll_{\pi,\tau} B \ll_{\pi,\tau} A$.

\begin{conjecture*}\label{conj:polynomial-lower-bound}
  Let $\pi, \tau$ and $\sigma \in \Wd_{\psi}(\tau)$ be as in
  \S\ref{sec:triple-integrals-matrix-coefficients-mp2}.
  Assume that \eqref{eqn:local-nonvanishing-of-sl2-periods}
  holds, so that $\mathcal{P}_{\SL_2(F)}$
  is not identically zero on $\pi \otimes \rho_{\psi} \otimes \sigma$.
  Then there exist $\varphi_1 \in \pi, \varphi_2 \in \rho_\psi, \varphi_3 \in
  \sigma$
  so that
  \[
  \mathcal{P}_{\SL_2(F)}(\varphi_1,\varphi_2,\varphi_2)
  \gg_{\pi,\tau} 1,
  \]
  \[
  \mathcal{S}(\varphi_i) \ll_{\pi,\tau} 1 \text{ for } 
  i=1,2,3.
  \]
\end{conjecture*}
We are content here to address
the case that $\sigma$ belongs to the principal series,
in which case the condition
\eqref{eqn:local-nonvanishing-of-sl2-periods}
is automatic (see \S\ref{sec:linearize-local-shimura}).
\begin{lemma*}
  The conclusion of the conjecture
  holds under the additional
  assumption that $\sigma$ belongs to the principal series.
\end{lemma*}

\begin{remark*} The general case of the conjecture is likely
  accessible by brute-force analysis of matrix coefficients as
  in \cite{PDN-HQUE-LEVEL, 2014arXiv1409.8173H}.  It also seems
  likely to follow by effectizing the proof of
  \eqref{eqn:local-nonvanishing-of-sl2-periods}.  It would be
  desirable to have a soft yet direct approach.  The question of
  to what extent the conjecture can be made uniform is also
  interesting; we address it partially in
  \S\ref{sec:comments-other-aspects}.
\end{remark*}

We observe first that if we are in the unramified case
and $\pi$ and $\tau$ are unramified,
then the conclusion of the lemma follows from
\eqref{eqn:unram-trip-prod-metaplectic}
upon taking $\varphi_1,\varphi_2,\varphi_3$ to be unramified
unit vectors.
We may thus assume either that
we are not in the unramified case,
or that $F$ is non-archimedean and at least one of $\pi, \tau$
is ramified.
In particular, if $F$ is non-archimedean,
we may assume that
\begin{equation}\label{eqn:q-ess-bdd}
  q \ll_{\pi,\tau} 1.
\end{equation}
We may and shall assume that the Haar on $N \backslash G$
is as in \S\ref{sec-3-14-3}.

\subsubsection{Choice of models}
\label{sec-3-15-2}
\begin{itemize}
\item We realize $\pi$ in its $\overline{\psi}$-Whittaker model.
\item We realize $\rho_{\psi}$ on $\mathcal{S}(F)$, as usual.
\item We write $\sigma = \mathcal{I}_{\Mp_2}(\chi)$
  for some character $\chi$ of $F^\times$
  and realize $\sigma$ in its induced model.
  Our assumptions imply that
  $c := \Re(\chi)$ satisfies $|c| \leq \vartheta$.
\end{itemize}
We accordingly write $W, \phi, f$
instead of $\varphi_1,\varphi_2,\varphi_3$.

\subsubsection{Some estimates}\label{sec:some-estimates}
The proof of the lemma of \S\ref{sec-3-14-2}
shows that the integral
\[
\mathcal{L}_{\chi}(W,\phi)
:=
\int_{y \in F^\times}
W(t(y))
\phi(y) \overline{\chi }(y) \, \frac{d^\times y}{|y|^{1/2}}
\]
converges absolutely
for $W \in \pi, \phi \in \rho_{\psi}$
and defines
a functional $\mathcal{L}_{\chi} : \pi \otimes \rho_\psi \rightarrow
\mathbb{C}$
so that for $f \in \sigma$,
we have (cf. \S\ref{sec-3-14-2})
\[
\ell(W,\phi,f)
= \int_{x \in F}
\overline{f(n'(x))}
\mathcal{L}_{\chi}(n(x) W, n(x) \phi) \, d x.
\]
We require some very crude estimates
for $\mathcal{L}_{\chi}$:
\begin{lemma*}~
  \begin{enumerate}[(i)]
  \item  $\mathcal{L}_{\chi}(W,\phi) \ll \mathcal{S}(W)
    \mathcal{S}(\phi)$.
  \item
    $\mathcal{L}_{\chi}(n(x) W, n(x) \phi) = \mathcal{L}_{\chi}(W,\phi)
    + \O(|x| \mathcal{S}(W)
    \mathcal{S}(\phi))$
    for $x \in F$ with $|x| \ll 1$.
  \end{enumerate}
\end{lemma*}
\begin{proof}
  (i) follows from \eqref{eqn:W-GL2-bound} and
  \eqref{eqn:V-phi-bound}.  For (ii), we write
  $n(x) W = W + W'$, $n(x) \phi = \phi + \phi '$ and
  $\mathcal{L}_{\chi}(n(x) W,n(x) \phi) = \mathcal{L}_{\chi}(W,\phi) +
  \mathcal{L}_{\chi}(W',\phi) + \mathcal{L}_{\chi}(W,\phi') +
  \mathcal{L}_{\chi}(W',\phi')$.
  By (i), we then reduce to the following property of the
  Sobolev norms, which follows readily from their definition:
  \[\mathcal{S}_d(n(x) v - v) \ll |x| \mathcal{S}_{d+1}(v)
  \text{ for
    $|x| \ll 1$ and $d$ fixed.}
  \]
\end{proof}

\subsubsection{Choice of $W, \phi$}
We fix a nonnegative function $\alpha \in
C_c^\infty(F^\times)$,
not identically zero on $F^{\times 2}$;
in the unramified case, we assume that $\alpha =
1_{\mathfrak{o}^\times}$.
We choose $W$ to be given in the Kirillov model by
$W(a(y)) = \alpha(y)$.
Let $\omega_\pi$ denote the central character of $\pi$.
Then $W(t(y)) = \omega_\pi(y)^{-1} \alpha(y^2)$.
We choose $\phi(y) := \overline{\chi }^{-1}(y) \omega_\pi(y)
\alpha(y^2)$.
Then
$W(t(y)) \phi(y) = \alpha(y^2)^2$,
so
\[
\mathcal{L}_{\chi}(W,\phi) = \int_{y \in F^\times}
\alpha(y^2)^2 \, \frac{d^\times y}{|y|^{1/2}}
\asymp 1.
\]
We have $\|W\|, \|\phi\| \asymp 1$
and the crude estimates
\begin{equation}\label{eqn:estimate-W-phi-choice}
  \mathcal{S}(W), \mathcal{S}(\phi) \ll_{\pi,\tau}
  1,
\end{equation}
which may be verified as follows:
\begin{itemize}
\item In the non-archimedean case, the construction of $W$ shows
  that it is invariant by $n(x)$ and $a(y)$ for all
  $x \in \mathfrak{p}^{O(1)}$ and
  $y \in \mathfrak{o}^\times \cap (1 + \mathfrak{p}^{O(1)})$.
  To deduce the required estimate for
  $\mathcal{S}(W)$, it suffices to verify that $W$ is invariant
  by $n'(x)$ for all $x \in \mathfrak{p}^{c(\pi) + \O(1)}$, or
  equivalently, that $w W$ is invariant by $n(x)$ for all such
  $x$, i.e., that $w W(a(y))$ is supported on
  $y \in \mathfrak{p}^{-(c(\pi) + \O(1))}$;
  for this we decompose $\alpha$ as a linear combination
  of the functions
  $\nu \cdot 1_{\varpi^{-n}
    \mathfrak{o}^\times}$,
  where $\nu$ is a unitary character of $F^\times$ and $n$ is
  integer,
  and appeal to the Jacquet--Langlands local functional equation
  (see, e.g., \cite[\S3.2.2, \S3.2.3]{michel-2009}).
  The corresponding estimate for $\mathcal{S}(\phi)$ is deduced
  similarly using the Tate local functional equation.
\item In the archimedean case, the estimate for $\mathcal{S}(W)$
  follows from \cite[\S3.2.5]{michel-2009}.  The estimate for
  $\mathcal{S}(\phi)$ follows from the fact that the Lie algebra
  of $\Mp_2(F)$ acts on $\rho_{\psi}$ via differential operators.
\end{itemize}

\subsubsection{The tempered case}
Suppose first that $\sigma$ is tempered,
so that $\chi$ is unitary.
We choose $f \in \sigma$ to be given in the line model by
$f(n'(x)) := |X|^{1/2} \alpha(X x)$ for some parameter
$X \in F^\times$ with $|X| \geq 1$, to be chosen later.  Then
$\|f\| \asymp 1$.
By the lemma of \S\ref{sec:some-estimates}
and the estimate \eqref{eqn:estimate-W-phi-choice},
we see that
\[
\ell(W,\phi,f)
-
|X|^{-1/2} \mathcal{L}_{\chi}(W,\phi) \int_{x \in F} \alpha(x) \, d x
\ll_{\pi,\tau} |X|^{-1}.
\]
By choosing $X$ so that $|X|$ is a sufficiently large but fixed power of
$C(\pi) C(\tau)$, we obtain $\ell(W,\phi,f) \gg_{\pi,\tau} 1$.
Moreover, $\mathcal{S}(f) \ll_{\pi,\tau} 1$;
we may see this in
the non-archimedean case by considering the invariance of $f$
and in the archimedean case by noting that the Lie algebra of
$\Mp_2(F)$ acts on the line model of $\sigma$ via differential
operators with coefficients bounded polynomially in $C(\chi)$.
We conclude via \eqref{eq:linearized-local-shimura-period}.

\subsubsection{The non-tempered
  case}\label{sec:non-tempered-case-for-lower-bounds}
Suppose now that $\sigma$ is non-tempered,
thus $\chi = \eta |.|^c$ with $\eta$ quadratic, $c$ real,
and $0 \neq |c| \leq \vartheta$.
The map
\[
  \mathcal{H} : \sigma \otimes \overline{\sigma }
  \rightarrow \mathbb{C} 
\]
\[
  \mathcal{H}(f_1,f_2) :=
  \int_{g \in G} \langle g W, W \rangle
  \langle g \phi, \phi  \rangle
  \langle g f_1, f_2 \rangle
\]
defines a hermitian form on $\sigma$,
i.e., $\mathcal{H}(f_1, f_2) = \overline{\mathcal{H}(f_2,f_1)}$.  Set
$\mathcal{H}(f) := \mathcal{H}(f,f)$.  Our task is to find
$f \in \sigma$ for which
$\mathcal{S}(f) \ll_{\pi,\tau} 1 \ll_{\pi,\tau} \mathcal{H}(f)$.
It will suffice to find
$f_1,f_2 \in \sigma$
for which
\[\mathcal{S}(f_1), \mathcal{S}(f_2) \ll_{\pi,\tau} 1
\ll_{\pi,\tau} \Re(\mathcal{H}(f_1,f_2)),\]
because then the polarization identity
$\mathcal{H}(f_1 + f_2)
= \mathcal{H}(f_1) + \mathcal{H}(f_2)  + 2
\Re(\mathcal{H}(f_1,f_2))$
gives the required conclusion
for some
$f \in \{f_1, f_2, f_1 + f_2\}$.
To that end,
let $f_0 \in \sigma$ and $f_0^* \in \sigma^* =
\mathcal{I}_{\Mp_2}(\chi^{-1})$
be as in the tempered case,
given in the line model by
$x \mapsto |X|^{1/2} \alpha(X x)$
where $X \in F^\times$
is chosen so that $|X|$ is a sufficiently large but fixed  power
of $C(\pi) C(\tau)$.
Take \[
f_1 := |X|^{c/2} f_0,
\quad f_2 := |X|^{-c/2} R_\chi^{-1}(f_0^*).
\]
By \eqref{eq:linearized-local-shimura-period-nontempered},
we then have
\[\mathcal{H}(f_1,f_2)
= \ell(W,\phi,f_0) \overline{\ell(W,\phi,f_0^*)}.\]
Arguing as in the tempered case,
we see that
$\mathcal{H}(f_1,f_2)
- |X|^{-1/2} \kappa  \ll_{\pi,\tau} |X|^{-1}$
for a positive real
$\kappa$ with $\kappa \asymp 1$;
explicitly,
\[\kappa = 
\mathcal{L}_{\eta |.|^c}(W, \phi)
\overline{\mathcal{L}_{\eta |.|^{-c}}(W, \phi)}
= \int_{y \in F^\times} \alpha(y^2)^2 \, \frac{d^\times
  y}{|y|^{1/2}}
\int_{y \in F^\times} \alpha(y^2)^2 \, \frac{d^\times
  y}{|y|^{1/2-2 c}}.\]
Thus $\Re(\mathcal{H}(f_1,f_2)) 
\gg_{\pi,\tau} 1$ for $X$ as indicated.
We have
\[
\|f_1\|^2
= \|w f_1\|^2
= \int_{\xi \in F}
G(\chi,\psi,\xi)
|
\int_{x \in F}
|X|^{1/2+c} \alpha(X x)
\psi(\xi x)
\, d x |^2.
\]
Write $\alpha '(\xi) := \int_{x \in F} \alpha(x) \psi(\xi x) \, d x$.
Using \eqref{eq:G-asymp}
and
\eqref{eqn:q-ess-bdd},
we see that
\[
\|f_1\|^2
\asymp_{\pi,\tau}
\int_{\xi}
|\xi|^{-c}
\left\lvert |X|^{-1/2+c/2}
  \alpha '(\xi /X)
\right\rvert^2
\, d \xi 
\asymp 1.
\]
Using the isometry property
\eqref{eq:R_chi-isom} of $R_\chi$,
we deduce similarly that $\|f_2\|^2 \asymp_{\pi,\tau} 1$
and then as in the tempered case
that $\mathcal{S}(f_1), \mathcal{S}(f_2) \ll_{\pi,\tau}
1$.
The proof is then complete.

\subsection{Upper bounds at uninteresting places\label{sec:uninteresting-upper-bounds}}
\label{sec-3-16}
\begin{lemma*}
  Let $\pi$ be a $\vartheta$-tempered
  generic irreducible unitary representation of $\GL_2(F)$.
  Let $\rtchi$ be a unitary character of $F^\times$.
  Let $\varphi_1 \in \pi, \varphi_1' \in \overline{\pi}, 
  \Phi \in \mathcal{I}_{\PGL_2}(\rtchi)$.
  Then
  \[
    \mathcal{P}_{\PGL_2(F)}(\varphi_1,\varphi_1',\Phi)
    \ll
    C(\rtchi)^{-A} \mathcal{S}(\varphi_1)^2 \mathcal{S}(\varphi_1')^2 \|\Phi\|^2.
  \]
\end{lemma*}
\begin{proof}
  This is a weak form of  \cite[Lemma 3.5.2]{michel-2009},
  taking into account that
  $\log C_{\Sob}(\mathcal{I}_{\PGL_2}(\rtchi)) \asymp \log C(\rtchi)$
  (see  \cite[Lem 2.6.6]{michel-2009} and the accompanying footnote).
\end{proof}

\begin{corollary*}
  Let $\pi$ be a $\vartheta$-tempered generic irreducible
  unitary representation of $\GL_2(F)$.
  Let $\rtchi$ be a unitary character of $F^\times$
  for which $\rtchi^2 = \chi$.
  Let $\varphi_1 \in \pi, \varphi_1' \in \overline{\pi},
  \varphi_2 \in \rho_{\psi}, \varphi_2' \in \overline{\rho_{\psi}}$.
  Set $\Phi := I_\chi(\varphi_2 \otimes \varphi_2') \in
  \mathcal{I}_{\SL_2}(\chi)$ (see \S\ref{sec:local-I-chi}).
  Then
  \[
    \mathcal{P}_{\PGL_2(F)}(\varphi_1, \varphi_1', \Phi^\rtchi)
    \ll C(\rtchi)^{-A} \prod_{i=1,2} \mathcal{S}(\varphi_i)^2 \mathcal{S}(\varphi_i')^2.
  \]
\end{corollary*}
\begin{proof}
  Recall from \S\ref{sec:induced-rep-local-from-sl2-to-pgl2} 
  that $\|\Phi^{\rtchi}\| = \|\Phi\|$.
  The estimate of \S\ref{sec:local-I-chi}
  implies that $\|\Phi\| \ll \mathcal{S}(\varphi_2) \mathcal{S}( \varphi_2')$,
  so we may conclude by the previous lemma.
\end{proof}
\subsection{Estimates at the interesting place\label{sec:interesting-local-estimates}}
\label{sec-3-17}
\subsubsection{Statement of result}
\label{sec-3-17-1}
\begin{lemma*}
  Assume we are in the unramified case
  (see \S\ref{sec:local-unramified-case}).
  Let $\pi, \tau$ and $\sigma \in \Wd_{\psi}(\tau)$ be
  as in \S\ref{sec:triple-integrals-matrix-coefficients-mp2},
  but assume now also that
  \begin{itemize}
  \item $\tau$ (and hence $\sigma$) is an unramified principal
    series representation, and that
  \item $\pi$
    is a twist of the special representation
    (\S\ref{sec-3-8-4}).
  \end{itemize}
  Then there exist $\varphi_1 \in \pi, \varphi_2 \in \rho_{\psi}, \varphi_3 \in
  \sigma$
  so that
  \begin{enumerate}[(i)]
  \item $\|\varphi_i\| \asymp 1$ for $i=1,2,3$.
  \item \begin{equation}\label{eqn:int-low-bnd}
      \mathcal{P}_{\SL_2(F)}(\varphi_1,\varphi_2,\varphi_3) \gg C(\ad(\pi) \otimes \tau)^{-1/4}.    \end{equation}
  \item
    Let $\chi, \rtchi$ be unitary characters of $F^\times$
    with $\rtchi^2 = \chi$.
    Set $\Phi := I_\chi(\varphi_2 \otimes \overline{\varphi_2}) \in
    \mathcal{I}_{\SL_2}(\chi)$.
    Then
    \begin{equation}\label{eqn:int-up-bnd}
            \mathcal{P}_{\PGL_2(F)}(\varphi_1, \overline{\varphi_1}, \Phi^{\rtchi})
      \ll C(\rtchi)^{-A}
      C(\pi \otimes \overline{\pi} \otimes \rtchi)^{-1/2}. 
    \end{equation}
  \end{enumerate}
\end{lemma*}
\begin{remark*}
  It is natural to ask whether the conclusion of the lemma
  holds for more general classes of $\pi$.
  For some negative results in that direction, see
  \S\ref{sec:comments-other-aspects}.
\end{remark*}

We note that
the conclusion of the lemma
depends upon $\pi$
only through its restriction to $\SL_2(F)$
and the quantities $C(\ad(\pi) \otimes \tau)$ and $C(\pi \otimes
\overline{\pi} \otimes \omega)$.
These are unchanged upon replacing $\pi$ by a twist.
For the proof, we may and shall thus assume
that $\pi$ is the special representation itself,
rather than a twist thereof.

\subsubsection{Conductor formulas}
\label{sec-3-17-2}
\begin{itemize}
\item Since $\tau$ is unramified and $\pi$ is the special
  representation, one has
  $C(\ad(\pi) \otimes \tau) = C(\ad(\pi))^2 = q^4$.  (Indeed,
  the first of these identities follows from the fact that if
  $\tau$ is the normalized induction of a pair of unramified
  characters $(\nu_1, \nu_2)$, then
  $C(\ad(\pi) \otimes \tau) = C(\ad(\pi) \otimes \nu_1)
  C(\ad(\pi) \otimes \nu_2)$
  and $C(\ad(\pi) \otimes \nu_i) = C(\ad(\pi))$.  The second of
  these identities follows from the identity
  $C(\ad(\pi)) = C(\pi \otimes \overline{\pi })$ together with
  \cite[Prop 1.4]{MR533066} and the relation between
  $\eps$-factors and analytic conductors recalled in, e.g.,
  \cite[\S3.1.12]{michel-2009}.)
\item 
  If $\rtchi$ is unramified,
  then
  $C(\pi \otimes \overline{\pi} \otimes \rtchi)
  = C(\ad(\pi)) = q^2$.
\end{itemize}
\subsubsection{Choice of models}
As in \S\ref{sec-3-15},
we realize $\pi$ in its $\overline{\psi}$-Whittaker model,
$\rho_{\psi}$ on $\mathcal{S}(F)$,
and $\sigma = \mathcal{I}_{\Mp_2}(\chi)$ in its induced model.
We again write $W,\phi,f$ instead of
$\varphi_1,\varphi_2,\varphi_3$.

\subsubsection{Choice of $W,\phi$}
\begin{itemize}
\item Let $W \in \pi$ be given in the Kirillov model by
$W(a(y)) := 1_{\mathfrak{o}}(y) |y|$; it is then a
newvector with $\|W\| \asymp 1$.  
\item Let $\phi \in \rho_{\psi}$ be given by $1_{\mathfrak{o}}$.
\end{itemize}
The assertions concerning
$L^2$-norms of $W$ and $\phi$ are clear by construction.
We will choose $f \in \sigma$ later,
separately according as $\sigma$ is tempered or not.

\subsubsection{Upper bounds}
\label{sec-3-17-5}
We now verify \eqref{eqn:int-up-bnd}.
If $c(\rtchi) = 0$,
then it follows from \S\ref{sec:local-I-chi} that $\Phi^\rtchi$ 
is a newvector of norm $\O(1)$,
and the required estimate
is given by \cite[Lem 4.4]{PDN-HQUE-LEVEL}.
Write $\nu := \mathcal{I}_{\PGL_2}(\rtchi)$.
If $c(\rtchi) \geq 1$,
then $c(\nu) = 2 c(\rtchi) \geq 2$,
so $\nu$ contains no nonzero $K_0[1]$-invariant vectors;
since $\varphi_1$ transforms under a unitary character of $K_0[1]$,
it follows from the $\GL_2(F)$-invariance
of $\mathcal{P}_{\PGL_2(F)}$ that 
$\mathcal{P}_{\PGL_2(F)}(\varphi_1, \overline{\varphi_1},
\Phi^{\rtchi}) = 0$.
This establishes the required upper bound (in a sharper form).

\subsubsection{Lower bounds: tempered case}
Suppose that $\sigma$ is tempered, so that $\chi$ is unitary.
Let $f \in \sigma$ be given in the line model
by $f(n'(x)) := q^{1/2} 1_\mathfrak{p}(x)$.
Then $\|f\| = 1$.
We now verify \eqref{eqn:int-low-bnd}.
Write $g = t(y) n'(z)$.
Then $f(g) = |y| \chi_{\psi}(y) \chi(y) q^{1/2}
1_{\mathfrak{p}}(z)$.
If $f(g) \neq 0$,
then $z \in \mathfrak{p}$,
hence $W(g) = |y|^2 1_{\mathfrak{o}}(y)$
and $\rho_{\psi}(g) \phi(1)
= \chi_{\psi}(y) |y|^{1/2} 1_{\mathfrak{o}}(y)$.
Thus
\[
  W(g) \rho_{\psi}(g) \phi(1) \overline{f(g)}
  =
  \overline{\chi}(y) 
  |y|^{7/2}
  1_{\mathfrak{o}}(y) q^{1/2} 1_{\mathfrak{p}}(z).
\]
It follows from \S\ref{sec:linearize-local-shimura} that
\begin{equation}\label{eq:lower-bound-interesting-case}
  \mathcal{P}_{\SL_2(F)}(\varphi_1,\varphi_2,\varphi_3)
  \asymp
  \left\lvert
    \int_{y \in F^\times}
    \int_{z \in F}
    | y |^{7/2}
    \chi(y)
    1_{\mathfrak{o}}(y) 
    q^{1/2}
    1_{\mathfrak{p}}(z)
    \, \frac{d^\times y}{|y|^2} \, d z
  \right\rvert^2
  \asymp q^{-1},
\end{equation}
which leads to the required lower bound.

\subsubsection{Lower bounds: non-tempered case}
Suppose now that $\sigma$ is non-tempered, thus
$\chi = \eta |.|^c$ with $\eta$ unramified quadratic, $c$ real, and
$0 \neq |c| \leq \vartheta$.
We define $\mathcal{H} : \sigma \otimes \overline{\sigma }
\rightarrow \mathbb{C}$
as in \S\ref{sec:non-tempered-case-for-lower-bounds}
and reduce to finding $f_1,f_2 \in \sigma$ for which
$\|f_1\|, \|f_2\| \ll 1$
and
$\mathcal{H}(f_1,f_2) \gg q^{-1}$.
For this we choose $f_0 \in \sigma$ and $f_0^* \in \sigma^* =
\mathcal{I}_{\PGL_2}(\chi^{-1})$
to be given in the line model by $n'(x) \mapsto q^{1/2}
1_\mathfrak{p}(x)$.
We take
\[f_1 := q^c f_0,
\quad
f_2 := q^{-c} R_\chi^{-1}(f_0^*).
\]
The estimate $\mathcal{H}(f_1,f_2) \gg q^{-1}$ is proved
as in the tempered case.
Since $\eta$ is unramified
and $\int_{x \in F} f_1(n'(x)) \psi(-\xi x) \, d x
= q^{c-1/2} 1_{\mathfrak{p}^{-1}}(\xi)$,
we see from \eqref{eq:G-asymp}
that
\[
\|f_1\|^2
\asymp
q^{2 c - 1}
\int_{\xi \in F}
|\xi|^{-c}
C(\chi_\xi)^{-c}
1_{\mathfrak{p}^{-1}}(\xi) \, d \xi = \sum_{n \geq -1} \iota(n),
\]
where $\iota(n)$ denotes the contribution
from $\xi \in \mathfrak{p}^n - \mathfrak{p}^{n+1}$.
We have
$\iota(-1) \asymp 1$, while 
$\iota(2 n) \asymp q^{2 c - 1 - 2 n(1-c)}$ and
$\iota(2 n + 1) \asymp q^{c-1 - (2 n + 1)(1 - c)}$
for $n \geq 0$.  Thus
$\|f_1\|^2 \asymp 1$.
We verify similarly that $\|f_2\|^2 \asymp 1$.
The proof is then complete.

\subsection{Sobolev--type bounds for twisting isometries\label{sec:local-twisting-bounds}}
\label{sec-3-18}
Let $\pi$ be an irreducible unitary representation of
$\GL_2(F)$,
and let $\chi$ be a unitary character of $F^\times$.
We may then form the tensor product $\pi \otimes \chi$
of $\pi$ by the one-dimensional representation spanned by the
function
$\GL_2(F) \ni g \mapsto \chi(\det g)$.
The map $v \mapsto v \otimes 1$ defines an
isomorphism $j_\chi : \pi \rightarrow \pi \otimes \chi$ of
vector spaces.
The representation $\pi \otimes \chi$
admits a natural unitary structure.

The map $j_\chi$ is an isometry,
and is $\SL_2(F)$-equivariant,
but is typically not $\GL_2(F)$-equivariant,
and does not in general preserve Sobolev norms.
However, it only polynomially distorts the latter:
\begin{lemma*}
  Let $d \geq 0$ be fixed.
  Then
  $\mathcal{S}_d(j_\chi(v))
  \ll C(\chi)^{O(1)} \|v\| + \mathcal{S}_d(v) \ll C(\chi)^{O(1)}
  \mathcal{S}_d(v)$
  for all $v \in \pi$.
\end{lemma*}
\begin{proof}
  We assume familiarity with \cite[\S2]{michel-2009}.
  Suppose first that $F$ is
  non-archimedean.  Let $K[n] \leq \GL_2(\mathfrak{o})$ denote
  the $n$th principal congruence subgroup.  Write
  $C(\chi) = q^c$ with $c \in \mathbb{Z}_{\geq 0}$.  Then the
  function $g \mapsto \chi(\det(g))$ is $K[c]$-invariant.
  Expanding the definition of $\mathcal{S}_d$
  (see \cite[\S2.5]{michel-2009}),
  we obtain
  $\mathcal{S}_d(j_\chi(v))^2 \leq q^{d c} \|v\|^2 +
  \mathcal{S}_d(v)^2$.
  In the archimedean case, we argue similarly using that
  $X j_\chi(v) = j_\chi(X v + d \chi(\trace(X)) v)$ for $X$ in
  the Lie algebra of $\GL_2(F)$.
\end{proof}

\section{Global preliminaries\label{sec:global-prelims}}
\label{sec-4}
Let $F$ be a number field with adele ring $\mathbb{A}$.
Fix a nontrivial unitary character $\psi : \mathbb{A}/F \rightarrow \mathbb{C}^{\times}$.

\subsection{Generalities}
\label{sec-4-1}
\subsubsection{The metaplectic group}
\label{sec-4-1-1}
Let $\Mp_2(\mathbb{A})$ denote the metaplectic double cover of
$\SL_2(\mathbb{A})$.  It may be identified with the set of pairs
$(\sigma,\zeta) \in \SL_2(\mathbb{A}) \times \{\pm 1\}$ with the
group law
$(\sigma,\zeta) (\sigma',\zeta') = (\sigma \sigma ', \zeta \zeta
' c(\sigma,\sigma '))$
where
$c(\sigma,\sigma ') := \prod_\mathfrak{p}
c_\mathfrak{p}(\sigma_\mathfrak{p},\sigma_\mathfrak{p} ')$
with $c_\mathfrak{p}$ the local cocycle of
\S\ref{sec:local-metaplectic-group}.  Given
$u = (\sigma,\zeta) \in \Mp_2(\mathbb{A})$, we denote by
$\pr_{\SL_2(\mathbb{A})}(u) := \sigma$ its projection to
$\SL_2(\mathbb{A})$ and by
$\pr_{\SL_2(F_\mathfrak{p})}(u) := \sigma_\mathfrak{p}$ its
projection to $\SL_2(F_\mathfrak{p})$.  

Recall that a function
$f : \Mp_2(\mathbb{A}) \rightarrow \mathbb{C}$ is called
\emph{genuine} if $f(\sigma,-1) = - f(\sigma,1)$ for all
$\sigma \in \SL_2(\mathbb{A})$.  A product $f_1 f_2$ of genuine
functions descends to a function
$\SL_2(\mathbb{A}) \rightarrow \mathbb{C}$ that we also denote
by $f_1 f_2$.  We identify $\SL_2(F)$ with its image
under the canonical homomorphism
$\SL_2(F) \rightarrow \Mp_2(\mathbb{A})$ lifting the inclusion
$\SL_2(F) \hookrightarrow \SL_2(\mathbb{A})$,
which may be
characterized in turn by requiring that the elementary theta functions
defined below  be left $\SL_2(F)$-invariant.

\subsubsection{Groups, measures, norms\label{sec:global-groups-measures-norms}}
\label{sec-4-1-2}
For $G \in \{\GL_1, \SL_2, \PGL_2\}$,
we denote by $[G] := G(F) \backslash G(\mathbb{A})$
the corresponding quotient.
When $G \neq \GL_1$, we equip $[G]$ with Tamagawa measure,
so that $\vol([\SL_2]) = 1$, $\vol([\PGL_2]) = 2$.
We equip $[\GL_1]$
with an arbitrary Haar measure.
In all cases, the Haar on $[G]$
lifts to a Haar on $G(\mathbb{A})$
and then factors as a product of Haar measures on $G(F_{\mathfrak{p}})$
which, for almost all finite primes $\mathfrak{p}$,
assign volume one to maximal compact subgroups.
We thereby obtain for each place $\mathfrak{p}$ of $F$
a local field $F_{\mathfrak{p}}$ with nontrivial unitary character $\psi_{\mathfrak{p}}$
and Haar measures on each of the groups $G(F_\mathfrak{p})$.
The discussion of \S\ref{sec-3} then applies.

For $g = (g_{\mathfrak{p}})$ in one of the groups
$\SL_2(\mathbb{A})$ or $\PGL_2(\mathbb{A})$,
we denote by   $\|g \| := \prod \|g_\mathfrak{p} \|$
and
$\|\Ad(g) \| := \prod \|\Ad(g_\mathfrak{p}) \|$
the products of the local norms defined in \S\ref{sec-3-1-10}.
We extend this definition to $\Mp_2(\mathbb{A})$ via pullback.
\subsubsection{Convention on factorization of unitary structures}
\label{sec-4-1-3}
Let $\pi$ be an automorphic representation of one of the groups
$\SL_2(\mathbb{A})$, $\Mp_2(\mathbb{A})$, $\GL_2(\mathbb{A})$.
Assume that it factors as a restricted tensor product
$\pi = \otimes \pi_\mathfrak{p}$; this happens for each $\pi$
that we consider.  If $\pi$ is unitary and equipped with some
specific unitary structure, then we always fix a unitary
structure on the components $\pi_\mathfrak{p}$ that is
compatible with this factorization, so that
$\|\otimes \varphi_\mathfrak{p} \| = \prod
\|\varphi_\mathfrak{p}\|$.
\subsubsection{``Good places'' and ``unramified''}
\label{sec-4-1-4}
We say that a place $\mathfrak{p}$ of $F$ is \emph{good} if
$\mathfrak{p}$ is non-archimedean, $\norm(\mathfrak{p})$ is odd, $F_\mathfrak{p}$ is
unramified over its prime subfield, $\psi_\mathfrak{p}$ is
unramified,
and the Haar measures 
defined in \S\ref{sec:global-groups-measures-norms}
on the groups $G(F_{\mathfrak{p}})$
assign volume one to maximal compact subgroups.
Thus almost all (i.e., all but finitely many) places are good,
and the assumptions of \S\ref{sec-3-1-8} (defining the ``unramified case'')
apply whenever $\mathfrak{p}$ is good.
We say that $\mathfrak{p}$ is \emph{bad} if it is not good.

For a good place $\mathfrak{p}$ and
$G \in \{\GL_1, \SL_2, \PGL_2, \Mp_2\}$,
we say that a factorizable vector $\varphi = \otimes
\varphi_\mathfrak{p}$
in a factorizable representation $\pi = \otimes
\pi_\mathfrak{p}$
of $G(\mathbb{A})$
is \emph{unramified at $\mathfrak{p}$}
if the local component $\varphi_\mathfrak{p}$ is
unramified in the sense of \S\ref{sec:local-unram-stuff};
otherwise, we say that $\varphi$ \emph{ramifies at}
or \emph{is ramified at} $\mathfrak{p}$.

\subsection{Hecke characters}
\label{sec-4-2}
A \emph{Hecke character} is a continuous
homomorphism $\chi : \mathbb{A}^\times /F^\times \rightarrow
\mathbb{C}^\times$.
Its \emph{real part} is the real number
$\Re(\chi)$ for which $|\chi(y)| = |y|^{\Re(\chi)}$.
We denote by $\mathfrak{X}$ the group
of Hecke characters and by $\mathfrak{X}(c)$ the subset
consisting of those with real part $c$.
The space $\mathfrak{X}(c)$ comes equipped with a natural measure
dual to the given Haar measure on $\mathbb{A}^\times/F^\times$;
it may be characterized by requiring that
for all real numbers $c$ and test functions $f$ on $\mathbb{A}^\times/F^\times$,
 the inversion formula
$\int_{\chi \in \mathfrak{X}(c)} \int_{y \in \mathbb{A}^\times/F^\times} f(y) \chi(y)
= f(1)$ holds.

\subsection{Sobolev norms}
\label{sec-4-3}
Given a unitary representation
$\pi$ of one of the groups
or
$\SL_2(\mathbb{A}), \Mp_2(\mathbb{A}), \PGL_2(\mathbb{A}),
\GL_2(\mathbb{A})$,
we define Sobolev norms $\mathcal{S}_d$ on $\pi$
as in
\cite[\S2]{michel-2009} and \cite[\S4.6,
  \S5.3]{nelson-theta-squared}.
If $\pi$ factors as
a restricted tensor product $\otimes \pi_{\mathfrak{p}}$
and we fix unitary structures on its local components
compatible with this factorization,
then the Sobolev norms factor on pure tensors:
$\mathcal{S}_d(\otimes v_{\mathfrak{p}}) = \prod \mathcal{S}_d(v_{\mathfrak{p}})$.
We retain the convention of \S\ref{sec-3-9}
concerning ``implied indices.''

The refined ``automorphic'' Sobolev norms
$\mathcal{S}_d^{\mathbf{X}}$ considered in
\cite[\S2]{michel-2009} and \cite[\S4.6,
\S5.3]{nelson-theta-squared} will not be used in the present
paper.

\subsection{Elementary theta functions}
\label{sec-4-4}
Let $\rho_{\psi}$ denote the Weil representation
attached to $\psi$
of
$\Mp_2(\mathbb{A})$
acting on the Schwartz--Bruhat space $\mathcal{S}(\mathbb{A})$.
It is the restricted tensor product of the spaces considered
in \S\ref{sec:weil-repn}.
For $\phi \in \rho_{\psi}$,
the elementary theta
function
$\theta(\phi) : \SL_2(F) \backslash \Mp_2(\mathbb{A})
\rightarrow \mathbb{C}$
is the genuine automorphic form defined by the convergent series
$\theta(\phi)(g) := \sum_{\alpha \in F} (\rho_{\psi}(g)
\phi)(\alpha)$.
The map $\rho_{\psi } \ni \phi \mapsto \theta(\phi)$
is equivariant, and quite nearly unitary (see \eqref{eqn:norm-of-theta})
on the orthogonal complement
$\{ \phi : \phi(-x) = \phi(x) \text{ for all } x\}$
of its kernel $\{ \phi : \phi(-x) = - \phi(x) \text{ for all } x\}$.
\subsection{Eisenstein series}
\label{sec-4-5}
\subsubsection{Induced representations}
\label{sec-4-5-1}
For a Hecke character $\chi$ and $G \in \{\SL_2, \PGL_2\}$ we
define an induced representation $\mathcal{I}_{G}(\chi)$ of
$G(\mathbb{A})$ either by mimicking the local definitions of \S\ref{sec-3-3} or
by taking the restricted tensor products of the representations
attached there to the local components $\chi_v$.  
If $\chi$ is unitary, then $\mathcal{I}_{G}(\chi)$ 
is unitary,
and we equip it with
the tensor product of the locally-defined unitary structures.
\subsubsection{Intertwiners}
\label{sec-4-5-2}
Denote by
$\Eis_\chi$ (or simply $\Eis$ when $\chi$ is clear by context)
the standard map from $\mathcal{I}_G(\chi)$ to the space
of automorphic forms, defined for $\Re(\chi)$ large enough by
averaging over left cosets of the standard Borel in $G(F)$ and
in general by meromorphic continuation.
\subsubsection{Rankin--Selberg period formulas on $\PGL_2$}
\label{sec-4-5-3}
\label{sec:rs-period-formula-pgl2}
The following lemma paraphrases a special case of
\cite[\S4.4]{michel-2009}
(compare with \cite[\S2.2.2]{michel-2009}).
\begin{lemma*}
  Let $\pi$ be a cuspidal automorphic representation
  $\PGL_2(\mathbb{A})$.  Let $\chi$ be a unitary character of
  $\mathbb{A}^\times / F^\times$.  Set
  $\pi_1 := \pi, \pi_2 := \overline{\pi}$ and
  $\pi_3 = \mathcal{I}_{\PGL_2}(\chi)$.  Equip $\pi_1, \pi_2$
  with the norm coming from $L^2([\PGL_2])$ and $\pi_3$ with
  that coming from $\mathcal{I}_{\PGL_2}(\chi)$.  For $i=1,2,3$,
  let
  $\varphi_i = \otimes \varphi_{i \mathfrak{p}} \in \pi_i =
  \otimes \pi_{i \mathfrak{p}}$
  be a factorizable vector.  Let $S$ be a finite set of places
  of $F$, containing the bad ones, with the property that
  $\varphi_{i \mathfrak{p} }$ is an unramified unit vector for
  all $i$ and each $\mathfrak{p} \notin S$.  Then the squared
  period
  \[
  \left\lvert \int_{[\PGL_2]} \varphi_1 \varphi_2 \Eis_{\PGL_2}(\varphi_3) \right\rvert^2
  \]
  is equal to
  \[
  c
  \frac{  \zeta^{(S)}(2)
    L^{(S)}(\pi_1 \otimes \pi_2 \otimes \pi_3, \tfrac{1}{2} )}{
    L^{(S)}(\ad(\pi), 1)^2 |L^{(S)}(\chi,1)|^2
  }
  \prod_{v \in S}
  \mathcal{P}_{\PGL_2(F_v)}(\varphi_{1 v}, \varphi_{2 v}, \varphi_{3 v}).
  \]
  for some $c > 0$ depending only upon $F$.
\end{lemma*}
We note also that $L(\pi_1 \otimes \pi_2 \otimes \pi_3, s) = L(\pi \otimes \overline{\pi} \otimes \chi, s) L(\pi \otimes \overline{\pi} \otimes \chi^{-1}, s)$.
\subsubsection{$\SL_2$ vs. $\PGL_2$}
\label{sec-4-5-4}
If $\rtchi$ is a Hecke
character with $\rtchi^2 = \chi$ and $f$ is an element of
$\mathcal{I}_{\SL_2}(\chi)$, denote by $f^\rtchi$ its unique
extension to $\mathcal{I}_{\PGL_2}(\rtchi)$, as in \S\ref{sec:induced-rep-local-from-sl2-to-pgl2}.

\subsubsection{Lifting Rankin--Selberg periods on $\SL_2$ to $\PGL_2$}
\label{sec-4-5-5}
\label{sec:lift-sl2-periods-to-pgl2}
We shall encounter Rankin--Selberg integrals
on $[\SL_2]$ involving restrictions of automorphic forms on $[\GL_2]$.
In order
to relate such integrals
via \S\ref{sec:rs-period-formula-pgl2} to products of $L$-values and local integrals,
we must first lift them to $[\PGL_2]$:

\begin{lemma*}
  Let $\pi$ be a cuspidal automorphic
  representations of $\GL_2(\mathbb{A})$.
  Let $\chi$ be a Hecke character
  with $\Re(\chi) > -1$ and $\chi \neq |.|^1$.
  Let $\varphi \in \pi, \varphi' \in \overline{\pi}$
  and $f \in \mathcal{I}_{\SL_2}(\chi)$.
  Then
  \begin{equation}\label{eqn:rs-from-sl2-to-pgl2}
    \int_{[\SL_2]} \varphi \varphi' \Eis_{\SL_2}(f)
    =
    c
    \sum_
    {
      \substack{
        \rtchi \in  \mathfrak{X} :
        \rtchi^2 = \chi 
      }
    }
    \int_{[\PGL_2]} \varphi \varphi' \Eis_{\PGL_2}(f^\rtchi)
  \end{equation}
  for some $c > 0$ depending only upon $F$.
\end{lemma*}
\begin{proof}
  Denote temporarily by $K$ the standard maximal compact subgroup
  of $\SL_2(\mathbb{A})$,
  by $\overline{K}$ its image in $\PGL_2(\mathbb{A})$,
  by $N$ the standard unipotent subgroup of $\SL_2(\mathbb{A})$,
  by $T$ the standard diagonal torus of $\SL_2(\mathbb{A})$,
  and by $A$ the standard diagonal torus of
  $\PGL_2(\mathbb{A})$.
  We will exploit below the decompositions $\PGL_2(\mathbb{A}) = N A \overline{K}$
  and $\SL_2(\mathbb{A}) = N T K$.

  Denote by $\Phi(g) := \int_{x \in \mathbb{A}/F} \varphi
  \varphi '(n(x) g)$ the
  constant term of $\varphi \varphi'$.
  Both sides of \eqref{eqn:rs-from-sl2-to-pgl2}
  vary holomorphically with respect to $\chi$
  as $f$ varies in a flat section,
  so we may reduce to the case that $\Re(\chi)$ is sufficiently large.
  By the inclusion of the factor $c$, 
  the choice of Haar measures is irrelevant;
  choosing them suitably and unfolding,
  we obtain
  \begin{align*}
  \int_{[\SL_2]} \varphi \varphi' \Eis_{\SL_2}(f)
  &=
  \int_{k \in K}
  \int_{y \in \mathbb{A}^\times / F^\times}
  f(t(y) k)
  \Phi(t(y) k) |y|^{-2} \\
  &=
  \int_{k \in K}
  f(k)
  \int_{y \in \mathbb{A}^\times / F^\times}
  \Phi(a(y^2) k) |y|^{-1} \chi(y)
  \end{align*}
  and
  \begin{align*}
  \int_{[\PGL_2]} \varphi \varphi' \Eis_{\PGL_2}(f^\rtchi)
  &= 
  \int_{k \in K}
  \int_{y \in \mathbb{A}^\times / F^\times}
  f^\rtchi(a(y) k)
  \Phi(a(y) k) |y|^{-1} \\
  &=
  \int_{k \in K}
  f(k)
  \int_{y \in \mathbb{A}^\times / F^\times}
  \Phi(a(y) k) |y|^{-1/2} \rtchi(y).
  \end{align*}
  By approximating $\Phi$ by test functions,
  we reduce to verifying for each
  test function $\phi$ on $\mathbb{A}^\times / F^\times$
  and each Hecke character $\chi$
  that
  \begin{equation}\label{eqn:mellin-transform-of-restriction-to-squares}
  \int_{y \in \mathbb{A}^\times / F^\times}
  \phi(y^2) \chi(y)
  =
  \frac{1}{2}
  \sum_{\rtchi  \in \mathfrak{X} :
    \rtchi^2 = \chi}
  \int_{y \in \mathbb{A}^\times / F^\times}
  \phi(y) \rtchi(y).
  \end{equation}
  The proof of \eqref{eqn:mellin-transform-of-restriction-to-squares} is an exercise in Pontryagin duality
  and quotient measures, and left to the reader;
  it is similar to the identity $\int_{y \in \mathbb{R}_+^\times} \phi(y ^2) |y|^s \, d^\times y
  = \frac{1}{2} \int_{y \in \mathbb{R}_+^\times} \phi(y) |y|^{s/2} \, d^\times
  y$ satisfied
  by test functions $\phi$ on $\mathbb{R}^\times_+
  \cong \mathbb{R}^\times / \{\pm 1\}$ and complex numbers $s$. \end{proof}

\subsubsection*{Remark } The sum indexed by $\rtchi$ in \eqref{eqn:rs-from-sl2-to-pgl2}
 is finite
in the sense that the summand vanishes for $\rtchi$ outside some
finite set depending at most upon $\varphi, \varphi'$.
Indeed, the $[\PGL_2]$-integral vanishes unless $\rtchi$ is unramified
at all good places $\mathfrak{p}$ for $F$ at which $\varphi, \varphi'$ are unramified.
\subsection{Regularized spectral expansions of products of elementary theta functions\label{sec:reg-spectral-exp-theta}}
\label{sec-4-6}
Let $\chi$ be  a unitary Hecke character.
Denote by
$I_\chi  : \rho_{\psi} \otimes
\rho_{\overline{\psi }} \rightarrow
\mathcal{I}_{\SL_2}(\chi)$
the $\Mp_2(\mathbb{A})$-equivariant
intertwiner
defined and studied in \cite[\S5.8]{nelson-theta-squared};
it is given on pure tensors $\phi = \otimes \phi_{\mathfrak{p}}$
by $I_\chi(\phi) = L^{(S)}(\chi,1)
\cdot (\otimes I_{\chi_\mathfrak{p}}(\phi_\mathfrak{p}))$,
where $S$ is any finite set of places
containing the bad places and any at which $\phi$ is ramified,
and
$I_{\chi_\mathfrak{p}}$ is as in \S\ref{sec:local-I-chi}.
We record a special case of
\cite[Thm 2]{nelson-theta-squared}:
\begin{lemma*}
  Let $\Phi : [\SL_2] \rightarrow \mathbb{C}$
  be smooth and of rapid decay.
  Let $\phi_1, \phi_2 \in \rho_{\psi}$.
  Then
  \begin{equation}\label{eqn:theta-squared-regularized-expansion}
    \begin{split}
      \int_{[\SL_2]}
      \Phi 
      \theta(\phi_1) \overline{\theta(\phi_2)}
      &= 
      \int_{[\SL_2]}
      \Phi 
      \int_{[\SL_2]}
      \theta(\phi_1) \overline{\theta(\phi_2)}
      \\
      &\quad
      +
      \int_{\chi \in \mathfrak{X}(0)}
      \int_{[\SL_2]} \Phi \Eis_{\SL_2}(I_\chi(\phi_1 \otimes
      \overline{\phi_2})).
    \end{split}
  \end{equation}
  Moreover,
  if $\phi_i(-x) = \phi_i(x)$ for all $x \in \mathbb{A}$, then
  \begin{equation}\label{eqn:norm-of-theta}
    \int_{[\SL_2]}
    \theta(\phi_1) \overline{\theta(\phi_2)} = 2 \int_{\mathbb{A}}
    \phi_1 \overline{\phi_2}.
  \end{equation}
\end{lemma*}
As noted in \S\ref{sec-1}, the absence of a cuspidal contribution
on the RHS of \eqref{eqn:theta-squared-regularized-expansion} is
critical to our argument.

We note that the integrand in \eqref{eqn:theta-squared-regularized-expansion}
is holomorphic in $\chi$:
the simple pole of $\chi \mapsto I_\chi$
at the trivial character
is cancelled by the corresponding simple
zero of the Eisenstein intertwiner.
We note also that the second assertion \eqref{eqn:norm-of-theta} may be
applied to general $\phi_i \in \rho_{\psi}$ by first
replacing them with their even projections
$\phi_i^+(x) := (\phi_i(x) + \phi_i(-x))/2$.

\subsection{Global Waldspurger packets\label{sec:global-waldspurger-packets}}
\label{sec-4-7}
Let $\tau$ be a cuspidal automorphic representation
of $\PGL_2(\mathbb{A})$.
Given a collection of elements $\sigma_\mathfrak{p}^{\eps_\mathfrak{p}}$
of the local Waldspurger packets
$\Wd_{\psi_\mathfrak{p}}(\tau_\mathfrak{p})$
indexed by some signs
$\eps_\mathfrak{p} = \pm 1$ as in \S\ref{sec:local-waldspurger-packets} (necessarily
$\eps_{\mathfrak{p}} = + 1$ for almost all $\mathfrak{p}$),
one may form
the restricted tensor products
$\sigma^\eps = \otimes \sigma_\mathfrak{p}^{\eps_\mathfrak{p}}$.
Waldspurger showed that
$\sigma^\eps$ is automorphic
if and only if $\prod \eps_\mathfrak{p} = \eps(\tau,\tfrac{1}{2})$;
moreover, $\sigma^{\eps}$ is then cuspidal
and occurs
with multiplicity one in the space of automorphic forms
on $\SL_2(F) \backslash \Mp_2(\mathbb{A})$.
It is thus meaningful to regard $\Wd_{\psi}(\tau)
 := \{ \sigma^\eps : \prod \eps_\mathfrak{p} = \eps(\tau,\tfrac{1}{2})\}$
as a finite collection of genuine cuspidal automorphic
representations of $\Mp_2(\mathbb{A})$.

Suppose now that
$\tau_{\mathfrak{p}}$ is principal series for all places $\mathfrak{p}$.
Then each local Waldspurger packet is a singleton
$\Wd_{\psi_{\mathfrak{p}}}(\tau_\mathfrak{p})
= \{\sigma_{\mathfrak{p}}^+\}$
and $\eps(\tau,\tfrac{1}{2}) = 1$,
so the global Waldspurger packet is a singleton
$\Wd_{\psi}(\tau) = \{\sigma\}$
consisting of a cuspidal automorphic representation 
$\sigma$ of
$\Mp_2(\mathbb{A})$.

\subsection{Generalized Shimura integrals\label{sec:global-shimura-integral}}
\label{sec-4-8}
Let $\tau$ be a cuspidal automorphic representation of $\PGL_2(\mathbb{A})$.
Let $\sigma \in \Wd_\psi(\tau)$.
Let $\pi$ be a cuspidal automorphic representation of
$\GL_2(\mathbb{A})$
with unitary central character $\omega_\pi$.
Let $\rho_{\psi}$ be the Weil representation of
$\Mp_2(\mathbb{A})$
on $\mathcal{S}(\mathbb{A})$.  Consider the following
$\Mp_2(\mathbb{A})$-invariant hermitian form $\mathcal{P}_{\SL_2(F)}$
on $\pi \otimes \rho_{\psi} \otimes \sigma$:
for $\varphi_1 \in \pi, \varphi_2 \in \rho_{\psi}, \varphi_3 \in \sigma$,
\[
\mathcal{P}_{\SL_2(F)}(\varphi_1, \varphi_2, \varphi_3)
=
|
\int_{[\SL_2]}
\varphi_1
\cdot \theta(\varphi_2)
\cdot
\overline{\varphi_3}
|^2.
\]
We recall \cite[Thm 4.5]{MR3291638}:
\begin{lemma*}
Assume that the $\varphi_i$ are pure tensors $\otimes_\mathfrak{p}
\varphi_{i \mathfrak{p}}$.
Let $S$ be
a finite set of places, containing the bad ones,
with the property that $\varphi_{i \mathfrak{p}}$ is an unramified
unit vector for all $i$ and each $\mathfrak{p} \notin S$.
Set $\zeta_F^{(S)}(s) := \prod_{\mathfrak{p} \notin S}
\zeta_{F_\mathfrak{p}}(s)$.
Then
  \[
  \mathcal{P}_{\SL_2(F)}(\varphi_1,\varphi_2,\varphi_3)
  =
  \frac{1}{2^2}
  \frac{
    \zeta_F^{(S)}(2)     L^{(S)}(\ad(\pi) \otimes \tau, \tfrac{1}{2})
  }
  {
    L^{(S)}(\ad(\tau), 1) L^{(S)}(\ad(\pi), 1)
  }
  \prod_{\mathfrak{p}  \in S}
  \mathcal{P}_{\SL_2(F_\mathfrak{p})}(\varphi_{1 \mathfrak{p} }, \varphi_{2 \mathfrak{p} }, \varphi_{3 \mathfrak{p} }).
  \]
\end{lemma*}

\subsection{Bounds for $\SL_2$-matrix coefficients of automorphic representations of $\GL_2$}
\label{sec-4-9}
\label{sec:global-pgl2-to-sl2-matrix-coeff}
Let $\pi$ be a cuspidal automorphic representation of
$\GL_2(\mathbb{A})$.
The standard unitary structure on $\pi$ is given
for $\varphi, \varphi ' \in \pi$ by
$\langle \varphi', \varphi  \rangle
:=
\int_{[\PGL_2]} \varphi' \overline{\varphi}$.
In general, this differs from the modified pairing
$\langle \varphi', \varphi \rangle_{\SL_2}
:=
\int_{[\SL_2]} \varphi' \overline{\varphi }$
given by integrating over $[\SL_2] \subseteq [\GL_2]$.
The two are related by Fourier inversion:
taking into account that $\vol([\PGL_2]) = 2\vol([\SL_2])$,
we have
\begin{equation}\label{eq:decompose-global-sl2-matrix-coeff-via-pgl2}
  \langle \varphi', \varphi \rangle_{\SL_2}
=
\frac{1}{2}  \sum_{\chi  \in \mathfrak{X} : \chi^2 = 1}
\langle \varphi' , \varphi \otimes \chi \rangle
\end{equation}
where $\varphi \otimes \chi \in \pi \otimes \chi$ denotes the
automorphic form given by
$(\varphi \otimes \chi)(g) := \varphi(g) \chi(\det g)$.  The sum
on the RHS of
\eqref{eq:decompose-global-sl2-matrix-coeff-via-pgl2} may be
restricted to those $\chi$ for which
$\pi \otimes \chi \cong \pi$.  By the local estimate of
\S\ref{sec:local-twisting-bounds}
and axiom (S1d) of \cite[\S2.4]{michel-2009}, one has for each fixed $d$ the estimate
$\mathcal{S}_d(\varphi \otimes \chi) \ll C(\chi)^{O(1)}
\mathcal{S}(\varphi)$.
By the local bound for matrix coefficients given in
\S\ref{sec:bounds-mx-coefs}
and axiom (S1d) of \cite[\S2.4]{michel-2009},
we obtain for $g \in \GL_2(\mathbb{A})$ the crude but sufficient bound
\begin{equation}\label{eq:global-sl2-matrix-coeff-bound}
\langle \pi(g) \varphi ', \varphi  \rangle
\ll
\|\Ad(g)\|^{\vartheta' - 1/2}
\mathcal{S}(\varphi) \mathcal{S}(\varphi ')
\sum_{\chi \in \mathfrak{X} : \chi^2 = 1, \pi \otimes \chi \cong
  \pi}
C(\chi)^{O(1)}
\end{equation}

\section{Main result}
\label{sec-5}
\subsection{Statement\label{sec:stmt-main-result-booya}}
\label{sec-5-1}
\subsubsection{Inputs}
\label{sec-5-1-1}
Fix a number field $F$ and let $\mathbb{A}$,
$\psi : \mathbb{A}/F \rightarrow \mathbb{C}^{(1)}$
be as in \S\ref{sec:global-prelims}.
We assume given an infinite countable collection
$\mathcal{F}$ consisting of pairs $(\mathfrak{q}, \pi)$,
where
\begin{itemize}
\item $\mathfrak{q}$ is a finite prime of $F$, and
\item $\pi$ is a cuspidal automorphic representation of
  $\GL_2(\mathbb{A})$ whose local component
  $\pi_{\mathfrak{q}}$ satisfies
  the hypotheses of \S\ref{sec:interesting-local-estimates}.
\end{itemize}
We assume also that
\begin{equation}\label{eqn:p-tends-to-infinity-duh}
 \# \{ (\mathfrak{q},\pi) \in \mathcal{F}
: \norm(\mathfrak{q}) \leq X\} < \infty
\text{ for each $X> 0$.}
\end{equation}
\subsubsection{Asymptotic notation and terminology}
\label{sec-5-1-2}
We denote in what follows by $(\mathfrak{q},\pi)$
a varying element of $\mathcal{F}$.
Our convention is that all objects (scalars, places, representations, vectors, ...)
considered below
are allowed to depend implicitly upon $(\mathfrak{q}, \pi)$
unless they are explicitly designated as \emph{fixed},
in which case we require that they depend at most upon the number field $F$, the family
$\mathcal{F}$, and any aforementioned fixed quantities.  
An assertion $\alpha$ 
depending upon $(\mathfrak{q},\pi)$
will be said to hold \emph{eventually}
if there is a fixed finite subset $\mathcal{F}_0 \subseteq \mathcal{F}$
so that $\alpha$ holds
whenever $(\mathfrak{q},\pi) \notin \mathcal{F}_0$.
The standard asymptotic notation is then defined
accordingly.
For example, given complex scalar quantities $X, Y$ (possibly depending implicitly upon
the pair $(\mathfrak{q},\pi)$, per our convention), we write
\begin{itemize}
\item  $X = O(Y)$, $X \ll Y$ or $Y \gg X$
  to denote that there is a fixed $c > 0$
so that $|X| \leq c |Y|$, and
\item $X = o(Y)$ to denote that for each fixed $\eps > 0$,
  one has $|X| \leq \eps |Y|$ eventually.
\end{itemize}
Set $Q= \norm(\mathfrak{q})$.
Our assumption \eqref{eqn:p-tends-to-infinity-duh}
says that $Q$ eventually exceeds any fixed positive real.

\subsubsection{Assumptions}
\label{sec-5-1-3}
Our results are conditional on the following hypothesis:
\begin{hypothesisHF}
  There is a fixed $\delta_0 > 0$ so that for each unitary
  character $\chi$ of $\mathbb{A}^\times / F^\times$,
  \[
  L(\pi \otimes \overline{\pi} \otimes \chi, \tfrac{1}{2})
  \ll
  C(\pi \otimes \overline{\pi} \otimes \chi)^{1/4-\delta_0}
  \cdot
  \left(
    C(\chi)
    \cdot \prod_{\mathfrak{p}: \mathfrak{p}  \neq \mathfrak{q}} C(\pi_{\mathfrak{p}})
  \right)^{O(1)}.
  \]
  In other words,
  we assume a subconvex bound in the $\pi_{\mathfrak{q}}$-aspect
  with polynomial dependence upon $\chi$
  and the ramification of $\pi$ at places other than
  $\mathfrak{q}$.
\end{hypothesisHF}
\subsubsection{Main result}
\label{sec-5-1-4}
We state an equivalent form of Theorem \ref{thm:main-theorem-1}:
\begin{theorem}\label{thm:reformulated-to-be-precise}
  Let $\mathcal{F}$ be a family as above
  that satisfies Hypothesis $H(\mathcal{F})$.
  Let $(\mathfrak{q}, \pi) \in \mathcal{F}$, as above.
  There is then a fixed $\delta > 0$ with the following
  property.
  Let
  $\tau$
  be a cuspidal automorphic representation of
  $\PGL_2(\mathbb{A})$ 
  for which either
  \begin{enumerate}
  \item
    each
  local component $\tau_\mathfrak{p}$ belongs to the principal
  series, or
  \item the conclusion of the conjecture of \S\ref{sec-3-15-1}
    holds.
  \end{enumerate}
  Then
  \[
  L(\ad(\pi) \otimes \tau,\tfrac{1}{2})
  \ll
  C(\ad(\pi) \otimes \tau)^{1/4 - \delta} P^{O(1)},
  \]
  where
  $P :=
  C(\tau) \cdot  \prod_{\mathfrak{p} : \mathfrak{p} \neq
    \mathfrak{q}} C(\pi_\mathfrak{p})$.
\end{theorem}

\subsection{Preliminary reductions\label{sec:elem-redz}}
\label{sec-5-2}
Observe first (thanks to the
``pass to worst-case subsequences'' argument)
that in proving Theorem \ref{thm:reformulated-to-be-precise},
we may freely replace $\mathcal{F}$ by any infinite subset
thereof.
Next, observe that if the quantity $P$ in the statement of
Theorem \ref{thm:reformulated-to-be-precise}
satisfies $\log P \gg \log Q$,
then the required conclusion
is worse than the convexity bound.
For this reason,
we may and shall assume that
\begin{equation}\label{eq:essentially-unram-outside-p-makes-it-simpler}
  C(\tau) \cdot  \prod_{\mathfrak{p} : \mathfrak{p} \neq
    \mathfrak{q}} C(\pi_\mathfrak{p})
  = Q^{o(1)}.
\end{equation}
This reduction has
some pleasant consequences:
\begin{enumerate}
\item It implies that $\tau_{\mathfrak{q}}$ is unramified,
  so that the results of \S\ref{sec:interesting-local-estimates}
  become applicable.
\item It implies that
  \begin{equation}\label{eqn:simplified-conductor-approximation-tau-times-ad-pi}
  C(\ad(\pi) \otimes \tau) = Q^{4 + o(1)}.
\end{equation}
\item For each unitary Hecke character $\chi$,
\begin{equation}\label{eqn:simplified-conductor-approximation-pi-pi-chi}
  C(\pi \otimes \overline{\pi} \otimes \chi) = Q^{2 + o(1)}
  C(\chi)^{O(1)};
\end{equation}
\item $\# \ram(\pi) + \# \ram(\tau) = o(\log(Q))$.
\item \label{item:bad-Euler-irrel} If $S$ is a finite set of places
and $B_\mathfrak{p}$ ($\mathfrak{p}
\in S$)
are positive real quantities
for which
\begin{itemize}
\item $S$ contains at most $o(\log Q)$
  places not already in $\ram(\pi) \cup \ram(\tau)$,
  and
\item each $B_\mathfrak{p} \asymp 1$,
\end{itemize}
then $\prod_{\mathfrak{p} \in S} B_\mathfrak{p} = Q^{o(1)}$.
Note that known bounds toward Ramanujan (see \S\ref{sec:local-temperedness}) imply that
any individual non-archimedean Euler factor considered below has
magnitude $\asymp 1$.
\end{enumerate}

\subsection{Reduction to period bounds\label{sec:red-to-period-bounds}}
\label{sec-5-3}
The $L$-function $L(\ad(\pi) \otimes \tau,s)$ is self-dual, so
the global root number
$\eps(\ad(\pi) \otimes \tau,\tfrac{1}{2})$ is $\pm 1$; if it is
$-1$, then $L(\ad(\pi) \otimes \tau,\tfrac{1}{2}) = 0$, and so
the required estimate is trivial.  Assume henceforth that
\begin{equation}\label{eq:global-sign-1}
  \eps(\ad(\pi) \otimes \tau,\tfrac{1}{2}) = 1.
\end{equation}
For each place
$\mathfrak{p}$ of $F$, set
$\eps_\mathfrak{p} := \eps(\pi_\mathfrak{p} \otimes
\overline{\pi}_\mathfrak{p} \otimes
\tau_\mathfrak{p},\tfrac{1}{2})$.
Then the local Waldspurger packet
$\Wd_{\psi_\mathfrak{p}}(\tau_\mathfrak{p})$
contains an element
$\sigma_\mathfrak{p}$
with index $\eps_\mathfrak{p}$,
and the local hermitian form
$\mathcal{P}_{\SL_2(F_\mathfrak{p})}$
does not vanish identically
on $\pi_\mathfrak{p} \otimes \rho_{\psi_\mathfrak{p}} \otimes \sigma_\mathfrak{p}$
(see
\S\ref{sec:triple-integrals-matrix-coefficients-mp2}).
Moreover,
by \eqref{eq:global-sign-1},
we have
$\prod \eps_\mathfrak{p} = \eps(\pi \otimes \overline{\pi }
\otimes \tau,\tfrac{1}{2})
= \eps(\tau,\tfrac{1}{2})$,
so by results of Waldspurger recalled in \S\ref{sec:global-waldspurger-packets},
there is a cuspidal automorphic representation
$\sigma \in \Wd_{\psi}(\tau)$
with local components $\sigma_\mathfrak{p}$.

Fix isometric identifications
$\pi = \otimes \pi_\mathfrak{p}$,
$\sigma = \otimes \sigma_\mathfrak{p}$
per the conventions of \S\ref{sec-4-1-3}.
Define $\varphi_1 \in \pi, \phi_2 \in
\rho_{\psi}, \varphi_3 \in \sigma$
to be the pure tensors
obtained from the choices of local vectors
given in \S\ref{sec:uninteresting-lower-bounds} (for $\mathfrak{p} \neq \mathfrak{q}$)
and in \S\ref{sec:interesting-local-estimates} (for $\mathfrak{p} = \mathfrak{q}$).
Set $\varphi_2 := \theta(\phi_2)$.
By
the Shimura--like period formula from \S\ref{sec:global-shimura-integral},
the known global estimate
$L(\ad(\pi), 1) = Q^{o(1)}$ (see \cite{MR1344349}, \cite[\S 2.9]{2009arXiv0904.2429B})
and the local lower bounds of \S\ref{sec:uninteresting-lower-bounds} and \S\ref{sec:interesting-local-estimates},
we have
\[
\frac{L(\ad(\pi) \otimes \tau, \tfrac{1}{2})}{C(  \ad(\pi) \otimes \tau )^{1/4}}
\ll 
Q^{o(1)}
\left\lvert
  \langle \varphi_1 \varphi_2, \varphi_3 \rangle
\right\rvert^2.
\]
Recalling the estimate \eqref{eqn:simplified-conductor-approximation-tau-times-ad-pi}
for the conductor, the proof of Theorem \ref{thm:reformulated-to-be-precise}
reduces to that of the period bound
\begin{equation}\label{eqn:key-period-bound}
\langle \varphi_1 \varphi_2, \varphi_3 \rangle
\ll
Q^{-\delta}
\end{equation}
for some fixed $\delta > 0$.

We will prove \eqref{eqn:key-period-bound} by applying the amplification method
of \cite{michel-2009}
arranged so that Cauchy--Schwarz is applied the vector
$\varphi_3$.
The key input in that method is an asymptotic formula
for some mild generalizations
of the $L^2$-norms $\langle \varphi_1 \varphi_2, \varphi_1
\varphi_2 \rangle
= \langle |\varphi_1|^2, |\varphi_2|^2 \rangle$;
we establish such a formula below in
\S\ref{sec:thm-proof-key-estimate}.
In section \S\ref{sec:proof-main-term-estimates},
we refine that formula by estimating the main and error terms.
In section \S\ref{sec:constr-ampl},
we recall the construction of an amplifier, following \cite{michel-2009}.
In section \S\ref{sec:appl-ampl},
we pull everything together to deduce the
bound
\eqref{eqn:key-period-bound}.

\subsection{The key estimate\label{sec:thm-proof-key-estimate}}
\label{sec-5-4}
Recall from \S\ref{sec-4-1-1} that
to each $u \in \Mp_2(\mathbb{A})$ and each place $\mathfrak{p}$
we may attach a local component
$\pr_{\SL_2(F_\mathfrak{p})}(u) \in \SL_2(F_\mathfrak{p})$.  We say that $u$ is
\emph{reasonable} if
\begin{itemize}
\item $\pr_{\SL_2(F_{\mathfrak{q}})}(u) = 1$, and
\item $\# \{\mathfrak{p} : \pr_{\SL_2(F_{\mathfrak{p}})}(u) \neq 1 \} = o(\log Q)$.
\end{itemize}

Given a pure tensor
$\varphi = \otimes \varphi_\mathfrak{p}$
in some factorizable unitary representation (such as $\pi$ or $\rho_{\psi}$),
it will be convenient to introduce the abbreviation
$\mathcal{S}_d'(\varphi) := \|\varphi_{\mathfrak{q}}\|
\prod_{\mathfrak{p} \neq \mathfrak{q}}
\mathcal{S}_d(\varphi_{\mathfrak{p}})$;
thus
$\mathcal{S}'_d(\varphi)$ quantifies the ramification of $\varphi$
at places other than the distinguished place $\mathfrak{q}$.
We retain the standard convention concerning
implied indices of Sobolev norms (see \S\ref{sec-3-9},  \S\ref{sec-4-3}),
so that ``$\mathcal{S}'$'' means ``$\mathcal{S}_d'$ for some fixed $d$.''

It will be typographically convenient
in what follows to denote by
$\varphi^u$ the left action of a group element $u \in
\Mp_2(\mathbb{A})$ on an automorphic form $\varphi$.
Thus $\varphi_1^u := \pi(\pr_{\SL_2(\mathbb{A})}(u)) \varphi_1$
and
$\varphi_2^u := \theta(\rho_{\psi}(u) \phi_2)$
if $\varphi_2 = \theta(\phi_2)$.
The side-effect $\varphi^{u_1 u_2} = (\varphi^{u_2})^{u_1}$
of this convention will not matter for us.
\begin{lemma*}
Assume hypothesis $H(\mathcal{F})$.
There is a fixed $\delta_1 > 0$ so that
for each reasonable element $u \in \SL_2(\mathbb{A})$,
\begin{equation}\label{eqn:key-moment-estimate}
\int_{[\SL_2]}
\varphi_1^u \varphi_2^u \overline{\varphi_1}
\overline{\varphi_2}
-
\int_{[\SL_2]}
\varphi_1^u \overline{\varphi_1}
\int_{[\SL_2]}
\varphi_2^u 
\overline{\varphi_2}
\ll
Q^{-\delta_1}
\prod_{i=1,2} \mathcal{S}'(\varphi_i^u)
\mathcal{S}'(\varphi_i).
\end{equation}
\end{lemma*}
\begin{proof}
We first rearrange $\varphi_1^u \varphi_2^u \overline{\varphi_1} \overline{\varphi_2}
 = \varphi_1^u \overline{\varphi_1}  \varphi_2^u \overline{\varphi_2}$.
By the regularized spectral expansion of \S\ref{sec:reg-spectral-exp-theta},
we reduce to showing that
\begin{equation}
\label{eqn:to-be-shown-after-sl2-expn}
\int_{\chi \in  \mathfrak{X}(0)}
\int_{[\SL_2]}
\varphi_1^u \overline{\varphi_1}
\Eis_{\SL_2}(I_\chi(\phi_2^u \otimes
\overline{\phi_2}))
\ll
\mathcal{E}
\end{equation}
where
$\mathcal{E}
:= \mathcal{S}'(\varphi_1^u)
\mathcal{S}'(\varphi_1)
\mathcal{S}'(\phi_2^u)
\mathcal{S}'(\phi_2)
Q^{-\delta_1}$.
Let $A > 1$ be fixed and large enough that
$\int_{\chi \in   \mathfrak{X}(0)}
\sum_{\rtchi \in \mathfrak{X}(0) :
  \rtchi^2 = \chi}
C(\rtchi)^{-A} < \infty$.
By lifting the $\SL_2$-periods in \eqref{eqn:to-be-shown-after-sl2-expn}
to $\PGL_2$-periods as in \S\ref{sec:lift-sl2-periods-to-pgl2},
we reduce to showing for each $\rtchi,\chi \in \mathfrak{X}(0)$
with $\rtchi^2 = \chi$
that
\begin{equation}\label{eqn:to-be-shown-after-sl2-to-pgl2}
  \int_{[\PGL_2]}
\varphi_1^u \overline{\varphi_1}
\Eis_{\PGL_2}(I_\chi(\phi_2^u \otimes
\overline{\phi_2})^\rtchi)
\ll C(\rtchi)^{-A}
\mathcal{E}.
\end{equation}
As noted in \S\ref{sec:lift-sl2-periods-to-pgl2},
the LHS of \eqref{eqn:to-be-shown-after-sl2-to-pgl2}
vanishes unless, as we henceforth assume, $\rtchi_{\mathfrak{p}}$ is unramified 
for all good places $\mathfrak{p}$ of $F$
with $\pr_{\SL_2(F_{\mathfrak{p}})}(u) = 1$ and $\phi_{\mathfrak{p}}$ unramified.
By this observation and the assumption that 
$u$ is reasonable,
the set $S$ of places
at which anything is ramified satisfies $\# S = o(\log Q)$;
this property will be used in what follows to control products
of implied constants and local Euler factors,
as discussed in \S\ref{sec:elem-redz}.
By the Rankin--Selberg period formula of \S\ref{sec:rs-period-formula-pgl2}
and the definition of $I_\chi$ given in \S\ref{sec:reg-spectral-exp-theta},
the squared magnitude of the LHS of \eqref{eqn:to-be-shown-after-sl2-to-pgl2}
factors as a product
$\mathcal{G} \prod_{\mathfrak{p} \in S}
\mathcal{L}_\mathfrak{p}$
of global and local quantities,
where
\[
  \mathcal{G} \asymp \frac{|L^{(S)}(\pi \otimes \overline{\pi} \otimes
  \chi,\tfrac{1}{2})|^2}{L^{(S)}(\ad(\pi),1)^2}
\]
and (with notation as in \S\ref{sec:induced-rep-local-from-sl2-to-pgl2}, \S\ref{sec:local-I-chi})
\[
  \mathcal{L}_\mathfrak{p} =
  \mathcal{P}_{\PGL_2(F_\mathfrak{p})}
  (\varphi_{1 \mathfrak{p}}^{u_\mathfrak{p}},
  \overline{\varphi_{1 \mathfrak{p} }},
  I_{\chi_\mathfrak{p}}(\phi_{2, \mathfrak{p}}^{u_\mathfrak{p}} \otimes \overline{\phi_{2, \mathfrak{p}}})^{\rtchi_\mathfrak{p}}
  ).
\]
By the subconvexity hypothesis $H(\mathcal{F})$,
the approximation \eqref{eqn:simplified-conductor-approximation-pi-pi-chi}
for the analytic conductor and the estimate $L(\ad(\pi), 1) = Q^{o(1)}$,
we have \[\mathcal{G} \ll Q^{- (2.001) \delta_1} C(\pi \otimes \overline{\pi} \otimes \rtchi)^{1/2} 
\]
for some fixed $\delta_1
> 0$.
The local estimates of \S\ref{sec-3-16} and \S\ref{sec-3-17}
and the assumption $u_{\mathfrak{q}} = 1$ give
\[
  |\mathcal{L}_\mathfrak{p}| \ll
  C(\rtchi_{\mathfrak{p}})^{-2 A}
  \cdot 
   \begin{cases}
  \mathcal{S}(\varphi_{1\mathfrak{p} }^{u_\mathfrak{p}})^2
  \mathcal{S}(\varphi_{1 \mathfrak{p} })^2
  \mathcal{S}(\phi_{2 \mathfrak{p} }^{u_\mathfrak{p}})^2
  \mathcal{S}(\phi_{2 \mathfrak{p} })^2
  & \text{if } \mathfrak{p} \neq \mathfrak{q}, \\
  C(\pi \otimes \overline{\pi} \otimes \rtchi)^{-1/2} \|\varphi_{1 \mathfrak{q}}\|^4
  \|\phi_{2 \mathfrak{q}}\|^4 
  & \text{if } \mathfrak{p} = \mathfrak{q}.
  \end{cases}
\]
These estimates combine to give
$|\mathcal{G}
\prod_{\mathfrak {p} \in S}
\mathcal{L}_\mathfrak{p}|^{1/2} \ll C(\rtchi)^{-A} \mathcal{E}$, as required.
\end{proof}

\subsection{Bounds for matrix coefficients and Sobolev norms\label{sec:proof-main-term-estimates}}
\label{sec-5-5}
Let $u \in \Mp_2(\mathbb{A})$ be reasonable.
In this section, we refine the estimate \eqref{eqn:key-moment-estimate}
by taking into account our choice of vectors
and bounds for matrix coefficients.

Recall from \S\ref{sec:elem-redz} that we have reduced to the case that
$\prod_{\mathfrak{p} \neq \mathfrak{q}} C(\pi_\mathfrak{p}) =
Q^{o(1)}$,
and that the local component $\pi_{\mathfrak{q}}$ has
conductor $\mathfrak{q}$ and unramified central character, hence is an unramified twist
of the special representation.  In
particular, $\pi_{\mathfrak{q}}$ is not isomorphic to its
twist by any nontrivial quadratic character of
the local multiplicative group $F_{\mathfrak{q}}^\times$.  It follows that any
(quadratic) Hecke character $\chi$ for which
$\pi \otimes \chi \cong \pi$ satisfies $C(\chi) = Q^{o(1)}$;
moreover,  the number of such $\chi$ is
at most $O(2^{\# \ram(\pi)}) = Q^{o(1)}$.
By a variant
of the arguments of \S\ref{sec:global-pgl2-to-sl2-matrix-coeff}
(taking into account that $\pr_{\SL_2(F_\mathfrak{q})}(u) = 1$)
and the consequence $\mathcal{S} '(\varphi_1)^2 \ll Q^{o(1)}$
of our choice of $\varphi_1$ 
(see \S\ref{sec:uninteresting-lower-bounds}, \S\ref{sec:interesting-local-estimates}),
it follows that
\begin{equation}\label{eq:matrix-coef-global-final-1}
  \int_{[\SL_2]} \varphi_1^u \overline{\varphi_1}
  \ll
  \|\Ad(u)\|^{\vartheta' - 1/2}
  Q^{o(1)}.
\end{equation}

Similarly but more simply, the global identity \eqref{eqn:norm-of-theta}
describing the unitary structure on elementary
theta functions
and the local estimate of
\S\ref{sec:bounds-mx-coefs}
for their matrix coefficients furnish the bound
\begin{equation}\label{eq:matrix-coef-global-final-2}
  \int_{[\SL_2]} \varphi_2^u \overline{\varphi_2}
  \ll 
  \|\Ad(u)\|^{-1/4}  Q^{o(1)}.
\end{equation}

Our choice of vectors and
\eqref{eq:essentially-unram-outside-p-makes-it-simpler}
imply for $i=1,2$ that
\begin{equation}\label{eqn:bounds-for-global-sobolev-norms-of-choices}
  \mathcal{S}'(\varphi_i)
  \ll  Q^{o(1)},
  \quad 
  \mathcal{S}'(\varphi_i^u)
  \ll  \|u\|^{O(1)}  Q^{o(1)}.
\end{equation}
For future reference, we record also that
\begin{equation}\label{eqn:bounds-for-global-norm-of-phi-3}
  \|\varphi_3\| \ll 1.
\end{equation}

We conclude by
\eqref{eqn:key-moment-estimate},
\eqref{eq:matrix-coef-global-final-1},
\eqref{eq:matrix-coef-global-final-2},
and
\eqref{eqn:bounds-for-global-sobolev-norms-of-choices}
that the $\SL_2$-periods
\[\mathcal{I}_u := \int_{[\SL_2]}
\varphi_1^u \varphi_2^u \overline{\varphi_1}
\overline{\varphi_2}
\]
satisfy
\begin{equation}\label{eq:ultimate-sl2-period-estimate}
  \mathcal{I}_u \ll
  ( \|\Ad(u)\|^{\vartheta'-3/4}
  +
  \|u\|^{B}
  Q^{-\delta_1}) Q^{o(1)}
\end{equation}
for some fixed $\delta_1 > 0, B \geq 0$.

\subsection{Construction of an amplifier\label{sec:constr-ampl}}
\label{sec-5-6}
\begin{lemma*}
  For each fixed $\eta > 0$
  there is a finite measure
  $\nu$ on $\Mp_2(\mathbb{A})$
  for which the following properties hold eventually:
  \begin{enumerate}
  \item \label{item:bound-tot-var-nu}
    The total variation measure $|\nu|$ has mass
    bounded above by $Q^{C \eta}$,
    where  $C \geq 0$ is fixed and independent of $\eta$.
  \item \label{item:bound-tot-var-nu-2}
    Set $|\nu|^{(2)} := |\nu| \ast |\nu^*|$.
    For each fixed $\gamma > 1/2$ there is a fixed $\delta > 0$ so that
    \[
    \int \|\Ad(u)\|^{-\gamma} \, d |\nu|^{(2)}(u)
    \leq Q^{-\delta}.
    \]
  \item \label{item:nu-preserves-varphi-3} $\varphi_3 \ast \nu = c \varphi_3$
    for some $c \gg Q^{-o(1)}$.
  \item  \label{item:support-nu} Each $u \in \supp(\nu)$
    has the properties:
    \begin{enumerate}
    \item $u$ is reasonable in the sense of \S\ref{sec:thm-proof-key-estimate};
    \item $\pr_{\SL_2(F_\mathfrak{p})}(u) = 1$
      unless $\mathfrak{p}$ is a finite place at which
      $\pi,\sigma,\psi$ are all unramified;
    \item $\|u\| \leq Q^{\eta}$.
    \end{enumerate}
    In particular each $u \in \supp(|\nu|^{(2)})$ is reasonable.
  \end{enumerate}
\end{lemma*}
\begin{proof}
  The proof is as in \cite[\S5.2.3]{michel-2009} which in turn
  refers to \cite[\S4.1]{venkatesh-2005}.  For convenience, we
  record the construction of $\nu$ and sketch the proof.  Let
  $T$ denote the set of good primes $\mathfrak{p}$ of $F$ at
  which $\pi,\sigma,\psi$ and hence also $\varphi_3$ are
  unramified.  For each $\mathfrak{p} \in T$, the genuine
  spherical Hecke algebra of $\Mp_2(F_\mathfrak{p})$ (consisting
  of genuine compactly-supported distributions invariant on the
  left and right under the image of the standard maximal compact
  subgroup of $\SL_2(F_\mathfrak{p})$) admits a linear basis
  $1, T_\mathfrak{p}, T_{\mathfrak{p}^2}, \dotsc$, where
  $T_{\mathfrak{p}^a}$ acts on the unramified subspace of $\sigma_\mathfrak{p}$ by the
  corresponding Hecke eigenvalue $\lambda_\tau(\mathfrak{p}^a)$ of the lift
  $\tau_\mathfrak{p}$ of $\sigma_\mathfrak{p}$, normalized so that temperedness is
  equivalent to $|\lambda_\tau(\mathfrak{p}^a)| \leq a+1$.  The
  relation $T_\mathfrak{p}^2 = T_{\mathfrak{p}^2} + 1$ holds.
  We may identify each $T_{\mathfrak{p}^a}$ with a finite
  measure on $\Mp_2(\mathbb{A})$, acting on $\varphi_3$
  by $\lambda_\tau(\mathfrak{p}^a)$.
  We may find a fixed
  $C_0 \geq 0$ so that $\|u\| \leq \norm(\mathfrak{p})^{C_0 a}$
  for all $u \in \supp(T_{\mathfrak{p}^a})$.  Set
  $L := Q^{\eta/C_0}$.  For an integral ideal $\mathfrak{n}$,
  define $c_\mathfrak{n} := 0$ unless $\mathfrak{n}$ is of the
  form $\mathfrak{p}^a$ with $\mathfrak{p} \in T$ and
  $a \in \{1,2\}$ with $\norm(\mathfrak{p})^a \leq L$, in which
  case set
  $c_\mathfrak{n} := \sgn(\lambda_\tau(\mathfrak{n}))^{-1}$,
  where $\sgn(z) := z/|z|$.  Set
  $\nu := (L/\log L)^{-1} \sum_{\mathfrak{n}} c_\mathfrak{n}
  T_\mathfrak{n}$.
  Assertions \eqref{item:bound-tot-var-nu} and
  \eqref{item:support-nu} are then clear by construction.  We
  have $\varphi_3 \ast \nu = c \varphi_3$ with
  $c := (L / \log L)^{-1} \sum_{\mathfrak{n} \in \mathcal{L}}
  |\lambda_\tau(\mathfrak{n})|$,
  so assertion \eqref{item:nu-preserves-varphi-3} follows from
  the estimate $\# \mathcal{L} \gg Q^{-o(1)} L$ (a consequence
  of \S\ref{sec:elem-redz}) and the identity
  $\lambda_\tau(\mathfrak{p})^2 = \lambda_\tau(\mathfrak{p}^2) +
  1$.
  Assertion \eqref{item:bound-tot-var-nu-2} follows by direct
  calculation as in \cite[\S4.1]{venkatesh-2005}.
\end{proof}
% the modifications required to pass from $\GL_2$ to $\Mp_2$
% amount in classical terms to replacing
% the role of the identity $\lambda(p)^2 = \lambda(p^2) + 1$
% satisfied by the Hecke eigenvalues
% of modular forms of integral weight
% with that of the identity $\lambda(p^2)^2 = \lambda(p^4) + \lambda(p^2) + 1$
% satisfied by those of half-integral weight.

\subsection{Application of the amplification method\label{sec:appl-ampl}}
\label{sec-5-7}
To finish the proof of \eqref{eqn:key-period-bound},
we argue as in \cite[\S5.2]{michel-2009}:
Let $\eta >0$ be fixed and small enough in terms of
the fixed quantities $\delta_1$ and $B$ arising in \eqref{eq:ultimate-sl2-period-estimate}.
Let $\nu$ be the amplifier
constructed in \S\ref{sec-5-6} in terms of the parameter $\eta$.  
By
the first two properties of the amplifier,
our choice of $\eta$,
and the inequality  $\vartheta' - 3/4 < -1/2$,
we have
\begin{equation}\label{eqn:final-desired-estimate}
  \int_{u}
  (
  \|\Ad(u)\|^{\vartheta' -3/4}
  +
  \|u\|^{B}
  Q^{-\delta_1})
  d |\nu|^{(2)}(u)
  \ll Q^{-\delta}
\end{equation}
for some fixed $\delta > 0$.
By the third property of the amplifier,
we have $\varphi_3 = c \varphi_3 \ast \nu$ for
some eigenvalue $c \gg Q^{-o(1)}$.
The proof of \eqref{eqn:key-period-bound} thereby reduces
to the proof of an analogous bound
for
$\langle \varphi_1 \varphi_2, \varphi_3 \ast \nu \rangle$,
or equivalently, for 
$\langle (\varphi_1 \varphi_2) \ast \nu^*, \varphi_3 \rangle$.
By \eqref{eqn:bounds-for-global-norm-of-phi-3} and Cauchy--Schwarz,
we reduce to bounding $\|(\varphi_1 \varphi_2) \ast \nu^*\|^2 = \langle \varphi_1 \varphi_2, (\varphi_1 \varphi_2) \ast \nu^{(2)} \rangle$
where $\nu^{(2)} := \nu^* \ast \nu$.
Expanding the definition of $(\varphi_1 \varphi_2) \ast \nu^{(2)}$,
we reduce to showing that the $\SL_2$-periods
$\mathcal{I}_u$ of  \S\ref{sec:proof-main-term-estimates}
satisfy
$\int_{u}
\mathcal{I}_u
d \nu^{(2)}(u)
\ll
Q^{-\delta}$
for some fixed $\delta > 0$.
The latter estimate holds by \eqref{eq:ultimate-sl2-period-estimate}, \eqref{eqn:final-desired-estimate} and the triangle inequality.

\section{Comments on other aspects}
\label{sec:comments-other-aspects}
We give here some evidence that new ideas are required to adapt our
method to other aspects.
Reverting to the local notation of \S\ref{sec-3}, we will show
that a simplified form of the crucial lemma of
\S\ref{sec:interesting-local-estimates} fails for
representations $\pi$ other than those considered there.
The generality of this paper is thus optimal with respect
to a sufficiently restricted form of the method.

We assume that we are in the unramified case (\S\ref{sec:local-unramified-case}), and in particular
that $F$ is non-archimedean; similar considerations apply
more generally.  Let $\pi, \tau$
and $\sigma \in \Wd_{\psi}(\tau)$ be as in
\S\ref{sec:triple-integrals-matrix-coefficients-mp2}, with
$\tau$ (and hence $\sigma$) an unramified principal series
representation.  We assume for simplicity that $\tau$ (and hence
$\sigma$) is tempered.  We assume also that
\begin{equation}\label{eqn:ad-pi-ramified}
  C(\ad(\pi)) \neq 1,
\end{equation}
since we are ultimately interested in taking
$C(\ad(\pi))$ to $\infty$.

Let $\varphi_2 \in \rho_{\psi}$
be the unramified unit vector
given in the Schroedinger model by $1_\mathfrak{o} \in \mathcal{S}(F)$.
It follows from the proof of the lemma of
\S\ref{sec:interesting-local-estimates}
that if $\pi$ is a twist of the special representation,
then there are unit vectors
$\varphi_1 \in \pi, \varphi_3 \in \sigma$ so that
\begin{equation}\label{eqn:sec-5-recovering-convexity}
  \mathcal{P}_{\SL_2(F)}(\varphi_1, \varphi_2, \varphi_3)
\gg C(\ad(\pi) \otimes \tau)^{-1/4}.
\end{equation}
Indeed, this estimate is at the heart of that lemma.  The fact
that the inverse of the RHS of
\eqref{eqn:sec-5-recovering-convexity} looks like the convexity
bound is necessary for the success of the amplification method
as implemented in this paper.
We show that an analogous estimate cannot hold in any other aspect:
\begin{lemma*}
  If $\pi$ is not a twist of the special representation, then for all unit vectors
  $\varphi_1 \in \pi, \varphi_3 \in \sigma$,
  \begin{equation}\label{eqn:3-8-bound}
      \mathcal{P}_{\SL_2(F)}(\varphi_1, \varphi_2,
  \varphi_3)
  \ll C(\ad(\pi) \otimes \tau)^{-3/8}.
  \end{equation}
\end{lemma*}
The proof is given below.
The precise shape of $\varphi_2$ is unimportant;
a similar argument shows that if the
local field $F$ (possibly archimedean) and the additive
character $\psi$ are held fixed, then the same conclusion holds
for any fixed vector $\varphi_2$.  The interpretation is that if
one attempts to apply our method to other aspects using an
essentially fixed vector $\varphi_2$ (as may seem necessary,
cf. the proof sketch of \S\ref{sec:overview-proof}), then one
cannot reasonably hope even to recover the convexity bound: one
must save in the exponent $3/8 - 1/4 = 1/8$, which is infeasible
via amplification.
(To make matters worse, our proof
shows that the exponent $-3/8$ in \eqref{eqn:3-8-bound}
may often be improved
to $-1/2$.)
One nevertheless has a few potential avenues
for extending our method to other aspects.
\begin{enumerate}[(i)]
\item One could 
  attempt to achieve the lower bound
  \eqref{eqn:sec-5-recovering-convexity} by varying the vector
  $\varphi_2$ sufficiently.  Experience suggests then that the
  upper bound \eqref{eqn:int-up-bnd} fails (although we have not
  yet rigorously demonstrated that it must).  Thus the spectral
  expansion as in \eqref{eqn:to-be-shown-after-sl2-expn} does
  not decay uniformly rapidly with respect to $\chi$.
  Experience suggests that in such ``long'' spectral decompositions,
  the bounds obtained by estimating each component individually
  (even optimally) are inadequate.
\item The basic input to our method is not really the pair of
  vectors $\varphi_1, \varphi_2$ but rather their tensor product
  $\varphi_1 \otimes \varphi_2 \in \pi \otimes \rho_\psi$.  One
  could attempt to prove analogues involving higher rank tensors
  of the lemma of \S\ref{sec:interesting-local-estimates}.  It
  seems likely to us that this is not possible, although we have
  not yet rigorously excluded the possibility; doing so would
  involve solving a more complicated optimization problem than
  that treated below.
\item One could attempt to achieve adequate lower bounds
  involving higher rank tensors for which the resulting spectral
  expansions are ``long'' but nevertheless susceptible to
  arguments generalizing those of Munshi
  \cite{2017arXiv170905615M}.
  This approach seems to us the most promising;
  we hope to attempt it in the future.
\item One should be able to apply our method to prove
  subconvexity implications involving general variation of $\pi$
  provided that $\tau$ and hence $\sigma$ vary with $\pi$ in a
  specific manner (like in \cite{MR3489080}, but with
  ramification at the same place).  For example, this may be
  possible if at the interesting place we take normalized inductions
  $\pi = 1 \boxplus \nu$ (or twists thereof) and
  $\tau = \mathcal{I}_{\PGL_2}(\nu)$ for some highly ramified
  unitary character $\nu$.  The resulting subconvexity problem
  is then unrelated to our motivating applications (AQUE), but
  remains of basic interest.
\end{enumerate}

\begin{proof}[Proof of the lemma]
  We realize $\pi$, $\rho_{\psi}$ and $\sigma$
  in their respective models as in
  \S\ref{sec:interesting-local-estimates}.
  We assume for convenience (mildly contradicting
  the assumed conventions of \S\ref{sec:local-unramified-case})
  that the Haar measure
  on $F^\times$
  is given by $d^\times y = \frac{d y}{|y|}$.
  
  Our assumptions imply that $\sigma = \mathcal{I}_{\Mp_2}(|.|^{i v})$ for some
  $v \in \mathbb{R} / 2 \pi \log(q) \mathbb{Z}$.
  To simplify notation, we assume
  that $v = 0$.  The same argument applies for $v \neq 0$.

  Since we are in the unramified case, we have
  $\chi_\psi(u) = 1$ for all $u \in \mathfrak{o}^\times$.

  Let $\omega_\pi$ denote the central character of
  $\pi$.
  The central element $t(-1) \in \Mp_2(F)$ acts by
  $\omega_\pi(-1)$ on $\pi$ and trivially both on $1_\mathfrak{o} \in \rho_{\psi}$
  and on $\sigma$,
  so $\mathcal{P}_{\SL_2(F)}(\cdot,1_\mathfrak{o},\cdot)$ vanishes identically
  unless $\omega_\pi(-1) = 1$.
  This forces $\omega_\pi$ to be
  the square of some character, say $\omega_\pi^{1/2}$.  The
  conclusion of the lemma depends only upon the restriction of
  $\pi$ to $\SL_2(F)$, and that restriction is unchanged by
  twisting, so upon replacing $\pi$ with
  $\pi \otimes (\omega_\pi^{1/2})^{-1}$ as necessary,
  we reduce to
  the case that $\omega_\pi = 1$.
  In particular, $\pi$ is self-contragredient.
  Its local $\gamma$-factor
  thus has the form
  \[
  \gamma(\pi \otimes \chi,  s)
  :=
  \gamma(\pi \otimes \chi,  s, \overline{\psi})
  = \eps(\pi \otimes \chi, s, \overline{\psi })
  \frac{L(\pi \otimes \chi^{-1}, 1 - s)}{L(\pi \otimes \chi,
    s)},
  \] 
  while the Jacquet--Langlands local functional equation
  reads: for $W \in \pi = \mathcal{W}(\pi,\overline{\psi})$
  and characters $\chi$ of $F^\times$,
  \[
  Z(W,\chi,s)
  :=
  \int_{y \in F^\times}
  W(a(y)) \chi(y) |y|^{s-1/2} \, d^\times y
  =
  \frac{Z(w W, \chi^{-1}, 1 - s)}{
    \gamma(\pi \otimes \chi,  s),
  }
  \]
  with the integral defined first for $\Re(s)$ large enough
  and extended by meromorphic continuation.

  Let $W \in \pi$ and
  $f \in \sigma$ be unit vectors.
  Then
  $\mathcal{P}_{\SL_2(F)}(W,1_\mathfrak{o},f) \asymp \left\lvert
    \ell(W,1_\mathfrak{o},f) \right\rvert^2$.
  As noted in \S\ref{sec-3-17-2},
  one has $C(\ad(\pi) \otimes \tau) = C(\ad(\pi))^2$.
  Our task is thus to show that
  \[\ell(W,1_\mathfrak{o},f) \ll C(\ad(\pi))^{-3/8}.\]
  We
  will make use of the following
  integral formula and
  normalization of Haar measures:
  for $\Phi \in C_c(N \backslash G)$,
  \[
  \int_{N \backslash G} \Phi
  = 
    \int_{x \in \mathfrak{o}}
  \int_{y \in F^\times}
  \Phi(t(y) w n(x)) \, \frac{d^\times y}{|y|^2} \, d x
  + 
  \int_{x \in \mathfrak{p}}
  \int_{y \in F^\times}
  \Phi(t(y) n'(x)) \, \frac{d^\times y}{|y|^2} \, d x.
  \]
  For $g = t(y) w n(x)$ with $x \in \mathfrak{o}$,
  we have $W(g) = W(a(y^2) w n(x))$,
  $\rho_{\psi}(g) 1_\mathfrak{o}(1)
  = |y|^{1/2} 1_\mathfrak{o}(y)$
  and $f(g) = |y| f(w n(x))$.
  For $g = t(y) n'(x)$ with $x \in \mathfrak{o}$,
  we have $W(g) = W(a(y^2) n'(x))$,
  $\rho_{\psi}(g) 1_\mathfrak{o}(1)
  = |y|^{1/2} 1_\mathfrak{o}(y)$
  and $f(g) = |y| f(n'(x))$.
  Thus
  $\ell(W,1_\mathfrak{o},f) = \ell_1 + \ell_2$,
  where
  \[
  \ell_1
  :=
  \int_{x \in \mathfrak{o}} \int_{y \in F^\times}
  W(a(y^2) w n(x))
  1_\mathfrak{o}(y)
  f(w n(x))
  \frac{d^\times y}{|y|^{1/2}} \, d x
  \]
  \[
  \ell_2 :=
  \int_{x \in \mathfrak{p}} \int_{y \in F^\times}
  W(a(y^2) n'(x))
  1_\mathfrak{o}(y)
  f(n'(x))
  \frac{d^\times y}{|y|^{1/2}} \, d x.
  \]
  We henceforth estimate $\ell_1$;
  analogous arguments apply to $\ell_2$.
  Set
  \[
  \Psi(x) :=
  \int_{y \in F^\times}
  W(a(y^2) w n(x))
  1_\mathfrak{o}(y)
  \frac{d^\times y}{|y|^{1/2}},
  \]
  so that $\ell_1 = \int_{x \in F} \Psi(x) f(w n(x)) \, d x$.
  A change of variables gives
  \[
  1 = 
  \|f\|^2 =
  \int_{x \in F} |f(n'(x))|^2 \, d x
  = 
  \int_{x \in \mathfrak{o}} |f(w n(x))|^2 \, d x
  + 
  \int_{x \in \mathfrak{p}} |f(n'(x))|^2 \, d x.
  \]
  By Cauchy--Schwartz, positivity and Parseval,
  it follows that
  \[|\ell_1|^2
  \leq
  \int_{x \in \mathfrak{o}}
  |\Psi(x)|^2 \, d x
  \leq
  \int_{x \in F}
  |\Psi(x)|^2 \, d x
  =
  \int_{\xi \in F}
  |\Psi^\wedge(\xi)|^2 \, d \xi,
  \]
  where
  $\Psi^\wedge(\xi) := \int_{x \in F} \Psi(x) \psi(-\xi x) \, d
  x$.
  In the remainder of the proof,
  we will show that
  \begin{equation}\label{eq:opt-identity}
    \Psi^\wedge(\xi) = \Phi(\xi) W(a(\xi)) |\xi|^{-1/2}.
  \end{equation}
  for some function $\Phi : F^\times \rightarrow \mathbb{C}$
  satisfying the pointwise estimate
  \begin{equation}\label{eq:opt-estimate}
    \Phi(\xi) \ll C(\ad(\pi))^{-3/8} \text{ for all } \xi \in F^\times.
  \end{equation}
  Since $1 = \|W\|^2 = \int_{\xi \in F^\times}
  |W(a(\xi))|^2 \, \frac{d \xi }{|\xi|}$,
  the required estimate for $\ell_1$ then follows.

  We turn to \eqref{eq:opt-identity}.
  By detecting squares via quadratic characters $\eta$ of $F^\times$,
  we see that there exists $c > 0$ with $c \asymp 1$
  so that
  \[
  \Psi(x)
  =
  c
  \sum_{\eta^2 = 1}
  \int_{y \in F^\times}
  W(a(y) w n(x)) 1_\mathfrak{o}(y)
  \eta(y)
  |y|^{-1/4}
  \, d^\times y.
  \]
  Let $\zeta_F(s) = (1 - q^{-s})^{-1}$
  denote the local zeta factor.
  Choose $\eps > 0$ sufficiently small.
  By Cauchy's formula,
  we have
  $1_\mathfrak{o}(y)
  =
  \oint_{(\eps)}
  \zeta_F(s)
  |y|^{-s}$,
  where $\oint_{(\eps)}$ denotes the integral over
  $s = \eps + i t$ with $t$ sampled from the compact group
  $\mathbb{R} / 2 \pi \log(q) \mathbb{Z}$ with respect to the probability
  Haar.
  Using the estimate $W(a(y) w n(x)) \ll |y|^{1/2-\vartheta}$
  to justify exchanging integrals and applying the local functional equation,
  it follows that
  \begin{align*}
    \Psi(x)
    &=
      c
      \sum_{\eta^2 = 1}
      \int_{(\eps)}
      \zeta_F(s)
      \underbrace{\int_{y \in F^\times}
      W(a(y) w n(x))
      \eta(y)
      |y|^{-s-1/4}
      \, d^\times y}_{= Z(w n(x), \eta, 1/4-s)}
    \\
    &=
      c
      \sum_{\eta^2 = 1}
      \int_{(\eps)}
      \zeta_F(s)
      \gamma(\pi \otimes \eta, 3/4 + s)
      Z(n(x) W, \eta, 3/4 + s).
  \end{align*}
  We have $n(x) W(a(y)) = \psi(y x) W(a(y))$,
  thus
  \[
  Z(n(x) W, \eta, 3/4 + s)
  = \int_{y \in F^\times}
  \psi(y x)
  W(a(y))
  \eta(y)
  |y|^{3/4+s - 1/2}
  \, \frac{d y}{|y|},
  \]
  so by Fourier inversion,
  \[
  \int_{x \in F} \psi(-\xi x)
  Z(n(x) W, \eta, 3/4 + s) \, d x
  =
  \eta(\xi) |\xi|^{-1/4+s} W(a(\xi)) |\xi|^{-1/2} .
  \]
  Thus \eqref{eq:opt-identity}
  holds with
  $\Phi(\xi)
  = c \sum_{\eta^2 = 1} \Phi_\eta(\xi)$,
  where
  \[
  \Phi_\eta(\xi)
  :=
  \int_{(\eps)}
  \zeta_F(s)
  \gamma(\pi \otimes \eta, 3/4 + s)
  \eta(\xi) |\xi|^{-1/4+s}.
  \]
  We now verify the estimate \eqref{eq:opt-estimate} for
  $\Phi(\xi)$
  by bounding each $\Phi_{\eta}(\xi)$ individually.
  By the classification of
  $\vartheta$-tempered generic irreducible unitary
  representations of $\PGL_2(F)$
  and our assumption that $\pi$ is not a twist of the special representation,
  there are the following possibilities for $\pi$:
  \begin{enumerate}[(i)]
  \item $\pi$ is the normalized
    induction $\nu \boxplus \nu^{-1}$
    for some character $\nu$ of $F^\times$
    for which either
    \begin{itemize}
    \item $\nu$ is unitary, or
    \item $\nu = \nu_0 |.|^c$ with $\nu_0$ quadratic
      and $0 < c \leq \vartheta$.
    \end{itemize}
    We have $C(\pi \otimes \eta) = C(\nu \otimes \eta)^2$,
    $C(\ad(\pi)) = C(\nu^2)^2$
    and $L(\pi \otimes \eta, s)
    = L(\nu \eta, s) L(\nu^{-1} \eta, s)$.
    Our assumption
    \eqref{eqn:ad-pi-ramified}
    says that $\nu^2$ is ramified,
    hence that $\nu$ is unitary,
    that $\nu \eta, \nu^{-1} \eta$
    are ramified,
    and that $C(\nu) = C(\nu \eta) = C(\nu^{-1} \eta)$.
    Since $q$ is odd, we have also $C(\nu^2) = C(\nu)$.
    Thus
    \begin{equation}\label{eqn:classification-consequences-for-C-L}
      C(\pi \otimes \eta) = C(\ad(\pi)),
      \quad 
      L(\pi \otimes \eta,s) = 1.
    \end{equation}
  \item $\pi$ is supercuspidal.
    Then $L(\pi \otimes \eta, s) = 1$
    and $C(\pi) = C(\pi \otimes \eta) \geq q^2$ (see
    \cite[Prop 3.4]{MR0476703}).
    If $C(\pi) = q^{2 n}$,
    then $C(\ad(\pi)) \leq q^{2 n}$,
    while if $C(\pi) = q^{2 n + 1}$,
    then $C(\ad(\pi)) = q^{2 n + 2}$ (see e.g. \cite[Prop 2.5]{PDN-AP-AS-que}).
    In particular,
    \begin{equation}\label{eqn:classification-consequences-for-C-L-2}
      C(\pi \otimes \eta) \geq C(\ad(\pi))^{3/4},
      \quad 
      L(\pi \otimes \eta,s) = 1.
    \end{equation}
    (Equality holds if and only if $C(\pi) = q^3$.)
  \end{enumerate}
  In both cases
  it follows that
  $\gamma(\pi \otimes \eta,  s)
  = \eps_{\pi \otimes \eta} C(\pi \otimes \eta)^{1/2-s}$
  with
  $|\eps_{\pi \otimes \eta}| = 1$,
  hence that
  \begin{align*}
    \Phi_\eta(\xi)
    &\asymp
      \int_{(\eps)}
      \zeta_F(s)
      C(\pi \otimes \eta)^{-1/4-s}
      |\xi|^{-1/4+s}
    \\
    &=
      \begin{cases}
        C(\pi \otimes \eta)^{-1/4} |\xi|^{-1/4} &
        \text{ if }
        |\xi| \geq C(\pi \otimes \eta), \\
        0 & \text{ otherwise.}
      \end{cases} \\
    &\ll
      C(\pi \otimes \eta)^{-1/2} \\
    &\ll C(\ad(\pi))^{-3/8},
  \end{align*}
  as required.
\end{proof}

      %       -----------------
\bibliography{refs}{}
\bibliographystyle{plain}
      %       -----------------
\end{document}